\documentclass{amsart}

\usepackage{amsmath}
\usepackage{amsthm}
\usepackage{hyperref}
\usepackage{amsfonts,graphics,amsthm,amsfonts,amscd,latexsym}
\usepackage{epsfig}
\usepackage{flafter}
\usepackage{mathtools}
\usepackage{comment}
\usepackage{stmaryrd}
\usepackage[table,xcdraw]{xcolor}

\usepackage{mathabx,epsfig}

\hypersetup{
    colorlinks=true,    
    linkcolor=blue,          
    citecolor=blue,      
    filecolor=blue,      
    urlcolor=blue           
}
\usepackage{tikz}
\usetikzlibrary{graphs,positioning,arrows,shapes.misc,decorations.pathmorphing}

\tikzset{
    >=stealth,
    every picture/.style={thick},
    graphs/every graph/.style={empty nodes},
}

\tikzstyle{vertex}=[
    draw,
    circle,
    fill=black,
    inner sep=1pt,
    minimum width=5pt,
]
\usepackage[position=top]{subfig}
\usepackage{amssymb}
\usepackage{color}

\calclayout

\usetikzlibrary{decorations.pathmorphing}
\tikzstyle{printersafe}=[decoration={snake,amplitude=0pt}]

\newcommand{\supp}{\operatorname{Supp}}

\newcommand{\Spec}{\operatorname{Spec}}

\newcommand{\pp}{\mathbb{P}}

\newcommand{\qq}{\mathbb{Q}}
\newcommand{\zz}{\mathbb{Z}}
\newcommand{\nn}{\mathbb{N}}
\newcommand{\rr}{\mathbb{R}}
\newcommand{\cc}{\mathbb{C}}
\newcommand{\kk}{\mathbb{K}}

\newcommand{\Af}{\mathbb{A}}

\def\O#1.{\mathcal {O}_{#1}}			
\def\pr #1.{\mathbb P^{#1}}				
\def\af #1.{\mathbb A^{#1}}			
\def\ses#1.#2.#3.{0\to #1\to #2\to #3 \to 0}	
\def\xrar#1.{\xrightarrow{#1}}			
\def\K#1.{K_{#1}}						
\def\bA#1.{\mathbf{A}_{#1}}			
\def\bM#1.{\mathbf{M}_{#1}}				
\def\bN#1.{\mathbf{N}_{#1}}
\def\bL#1.{\mathbf{L}_{#1}}				
\def\bB#1.{\mathbf{B}_{#1}}				
\def\bK#1.{\mathbf{K}_{#1}}		
\def\bP#1.{\mathbf{P}_{#1}}		\def\bT#1.{\mathbf{T}_{#1}}	
\def\subs#1.{_{#1}}					
\def\sups#1.{^{#1}}						

\DeclareMathOperator{\coeff}{coeff}	

\DeclareMathOperator{\Supp}{Supp}

\newcommand{\rar}{\rightarrow}	
\newcommand{\drar}{\dashrightarrow}	

\usepackage{tikz}
\usetikzlibrary{matrix,arrows,decorations.pathmorphing}

\newtheorem{introcon}{Conjecture}

  \newtheorem{introthm}{Theorem}

  \newtheorem{theorem}{Theorem}[section]
  \newtheorem{lemma}[theorem]{Lemma}
  \newtheorem{proposition}[theorem]{Proposition}
  \newtheorem{corollary}[theorem]{Corollary}

  \newtheorem{notation}[theorem]{Notation}
  \newtheorem{definition}[theorem]{Definition}

\theoremstyle{remark}

\numberwithin{equation}{section}

\usepackage[all]{xy}

\usepackage{soul}

\usepackage{mathrsfs}

\begin{document}

\title[Complements and coregularity of Fano varieties]{Complements and coregularity of Fano varieties}

\author[F.~Figueroa]{Fernando Figueroa}
\address{Department of Mathematics, Princeton University, Fine Hall, Washington Road, Princeton, NJ 08544-1000, USA}
\email{fzamora@princeton.edu}

\author[S.~Filipazzi]{Stefano Filipazzi}
\address{
EPFL SB MATH CAG
MA C3 625 (B\^atiment MA)
Station 8
CH-1015 Lausanne, Switzerland
}
\email{stefano.filipazzi@epfl.ch}

\author[J.~Moraga]{Joaqu\'in Moraga}
\address{UCLA Mathematics Department, Box 951555, Los Angeles, CA 90095-1555, USA
}
\email{jmoraga@math.ucla.edu}

\author[J.~Peng]{Junyao Peng}
\address{Department of Mathematics, Princeton University, Fine Hall, Washington Road, Princeton, NJ 08544-1000, USA
}
\email{junyaop@princeton.edu}

\thanks{
SF was partially supported by ERC starting grant \#804334. FF received partial financial support by the NSF under J\'anos Koll\'ar's grant number DMS-1901855.
}

\maketitle

\begin{abstract}
We study the relation 
between the coregularity, 
the index of log Calabi--Yau pairs, 
and the complements of Fano varieties.
We show that the index of a log Calabi--Yau pair $(X,B)$ of coregularity $1$ is at most 
$120\lambda^2$, where $\lambda$ is the Weil index of $K_X+B$.
This extends a recent result due to Filipazzi, Mauri, and Moraga.
We prove that a Fano variety
of absolute coregularity $0$ admits either
a $1$-complement or a $2$-complement. 
In the case of Fano varieties
of absolute coregularity $1$, we show that they
admit an $N$-complement with $N$ at most 6.
Applying the previous results, we prove that a klt singularity
of absolute coregularity $0$ admits
either a $1$-complement or $2$-complement.
Furthermore, a klt singularity of absolute coregularity $1$ admits an $N$-complement with $N$ at most 6.
This extends the classic classification
of $A,D,E$-type klt surface singularities
to arbitrary dimensions.
Similar results are proved in the case of coregularity $2$.
In the course of the proof, we prove 
a novel canonical bundle formula for 
pairs with bounded relative coregularity.
In the case of coregularity at least $3$,
we establish analogous statements under the assumption of the index conjecture
and the boundedness of B-representations.
\end{abstract}

\setcounter{tocdepth}{1} 
\tableofcontents

\section{Introduction}

Fano varieties and 
Calabi--Yau varieties are two of the three
building blocks of algebraic varieties.
In the former case,
the canonical divisor is anti-ample,
while in the latter case it is numerically trivial.
In this article, we study the coregularity
of Fano and Calabi--Yau varieties.
This invariant measures the dimension of the dual complexes corresponding to log Calabi--Yau structures on the variety.
We show that if the coregularity is at most two, we can control the index of a Calabi--Yau variety
and the complements of a Fano variety.
In~\cite{Mor24}, the third named author relates various problems about Fano varieties
with the concept of coregularity.

\subsection{Log Calabi--Yau pairs}
A log Calabi--Yau pair
$(X,B)$ is a projective log canonical pair
for which $K_X+B\equiv 0$.
By the abundance conjecture, which is known in this special case~\cite{Gon13}, it is also known that $K_X+B\sim_\qq 0$.
The index of $(X,B)$ is the smallest positive integer $I$ for which $I(K_X+B)\sim 0$.
It is conjectured that the index $I$ 
of $(X,B)$
admits an upper bound depending on the dimension of $X$
and the set of coefficients of $B$.
This is known as the index conjecture.
For instance, if $(X,B)$ is two dimensional
and the coefficients of $B$ are standard (i.e., of the form $1-\frac{1}{m}$ for some $m\in \zz_{>0}$), then $I(K_X+B)\sim 0$ for some $I\leq 66$ (see, e.g., \cite[Theorem 4.11]{Ish00}).
The bound $66$ is optimal and it can be obtained by considering non-symplectic finite actions on K3 surfaces~\cite{MO98}.
In~\cite{ETW22}, the authors exhibit a sequence of klt Calabi--Yau varieties $X_d$ with index $i_d$ that grows doubly exponentially with the dimension $d$.
The {\em coregularity} of a log Calabi--Yau pair
$(X,B)$ is defined to be $\dim X - \dim \mathcal{D}(X,B)-1$.
Here, $\mathcal{D}(X,B)$ is the {\em dual complex} of $(X,B)$. This is a pseudo-manifold that encodes the combinatorial data of log canonical centers of a dlt modification of $(X,B)$. The dimension $\dim \mathcal{D}(X,B)$ is independent of the chosen dlt modification, so it is an intrinsic invariant of $(X,B)$.
In the following subsection, we present theorems regarding the index of log Calabi--Yau pairs of coregularity $0$ and $1$.

\subsection{Index and coregularity of log Calabi--Yau pairs} Firstly, we study 
log Calabi--Yau pairs of coregularity 0 and 1.
We use the language of generalized pairs
as in~\cite{Bir19,Fil20}. 
This gives a larger scope for the theorems
and also facilitates inductive arguments. 
In the case of generalized log Calabi--Yau pairs of coregularity $0$, the following theorem is proved in~\cite{FMM22}.

\begin{introthm}\label{introthm:index-coreg-0}
Let $(X,B,\bM.)$ be a projective generalized log Calabi--Yau pair\footnote{A generalized log Calabi--Yau pair is a generalized lc pair $(X,B,\bM.)$ with $K_X+B+\bM X.\sim_\qq 0$.} of coregularity $0$ and let be $\lambda$ a positive integer.
Assume that $\lambda(K_X+B+\bM X.)$ is Weil.
Then, we have that $2\lambda(K_X+B+\bM X.)\sim 0$.
\end{introthm}

Note that in order to have a linear equivalence
$D\sim 0$, the divisor $D$ must be Weil.
Hence, multiplying $K_X+B+\bM X.$ by $\lambda$ is
indeed needed to compute its index.
In~\cite[Example 7.4]{FMM22}, the authors give an example for which 
$(X,B)$ is log Calabi--Yau of coregularity $0$,
$B$ is a Weil divisor, 
$2(K_X+B)\sim 0$,
and $K_X+B$ is not linearly equivalent to 0.
Hence, the factor $2\lambda$ in the previous theorem is optimal. 
Indeed, the factor $2$ is often related to the orientability of the pseudo-manifold $\mathcal{D}(X,B)$ (see~\cite[\S~5]{FMM22}).
In~\cite{FMM22}, the authors use topological methods
and birational geometry to prove the previous theorem.
In this article, we recover this statement using birational geometry and the theory of complements.

Our next theorem deals with the index
of log Calabi--Yau pairs of coregularity $1$.
This is a generalization of the previous statement
to the case of coregularity $1$.

\begin{introthm}\label{introthm:index-coreg-1}
Let $(X,B,\bM.)$ be a projective generalized log Calabi--Yau pair of coregularity $1$ and $\lambda$ be a positive integer.
Assume the two following conditions hold:
\begin{itemize}
    \item the generalized pair $(X,B,\bM.)$ has Weil index $\lambda$; and 
    \item the variety $X$ is rationally connected or $\bM.=0$.
\end{itemize}
Then, we have that
$I(K_X+B+\bM X.)\sim 0$
for $I=m\lambda$ with $m\leq 120\lambda$.
\end{introthm}

We emphasize that the previous theorem does not hold if $X$ is not rationally connected
and $\bM.$ is non-trivial.
For instance, we can let $X$ be an elliptic curve
and $\bM.$ be an $I$-torsion point in $\mathrm{Pic}^0(X)$.
Then, we have that 
$I(K_X+\bM X.)\sim I\bM X. \sim 0$ is minimal
and $(X,\bM .)$ is a generalized log Calabi--Yau pair of coregularity $1$.
Note that this is not an issue if $X$ is rationally connected.
In this case,
the torsion of components of the b-nef divisor is controlled by their Weil index.
Theorem~\ref{introthm:index-coreg-0}
and Theorem~\ref{introthm:index-coreg-1}
are still valid if the coefficients of $B$ belong to a set of rational numbers satisfying the DCC condition.
This follows from the global ascending chain condition for generalized log Calabi--Yau pairs
with bounded coregularity (see~\cite[Theorem 2]{FMP22}).
Finally, in the case of coregularity $1$, we obtain the following statement.

\begin{introthm}\label{introthm:index-coreg-1-standard}
Let $(X,B,\bM.)$ be a projective generalized log Calabi--Yau pair of coregularity $1$.
Assume that the following conditions hold:
\begin{itemize}
    \item the coefficients of $B$ are standard;
    \item the divisor $2\bM .$ is b-Cartier; and
    \item the variety $X$ is rationally connected
    or $\bM.=0$.
\end{itemize}
Then, we have that $I(K_X+B+\bM X.)\sim 0$ for some $I\in \{1,2,3,4,6\}$.
\end{introthm}

\subsection{Fano varieties} 
Given a klt Fano variety $X$,
the anti-canonical divisor $-K_X$ is ample.
Hence, the linear system $|-mK_X|$ is basepoint free for $m$ sufficiently large and divisible. 
In particular, we can find 
an effective divisor $B \in |-mK_X|$ such that the pair $(X,B/m)$ is klt and log Calabi--Yau. 
This means that every Fano variety
admits a log Calabi--Yau structure. 
In~\cite{Bir19}, Birkar showed that an $n$-dimensional Fano variety $X$ admits an $N(n)$-complement, i.e., a boundary $B$ for which
$(X,B)$ is log canonical and $N(n)(K_X+B)\sim 0$.
This can be thought of as an effective log Calabi--Yau structure on $X$.
In~\cite{FMX19}, Filipazzi, Moraga, and Xu proved that a 3-fold that admits a $\qq$-complement\footnote{A $\qq$-complement is a boundary $B$ for which $(X,B)$ is lc and $K_X+B\sim_\qq 0$.}
also admits a $N_3$-complement.

Let $X$ be an $n$-dimensional Fano variety.
We can define the {\em absolute coregularity} to be:
\[ 
\hat{\rm coreg}(X) \coloneqq \min \{ {\rm coreg}(X,B) \mid \text{$(X,B)$ is log Calabi--Yau}
\}.
\] 
By definition, the absolute coregularity of $X$ is at most $n$.
It is expected that a Fano variety of absolute coregularity $c$ admits an $N(c)$-complement (see, e.g.,~\cite[Conjecture 4.1]{Mor24}).
Indeed, following the philosophy of Kawamata's X-method, one expects to lift complements from minimal log canonical centers.
In the case of a Fano variety of absolute coregularity $c$, we can produce minimal log canonical centers having dimension $c$ (on a suitable dlt modification). 
In the following subsection, we present some theorems regarding complements of Fano type varieties
of absolute coregularity $0$ and $1$.

\subsection{Complements and coregularity
of Fano type varieties}

Our main theorem in this direction
states that a Fano type variety of absolute coregularity 0 admits a $1$-complement or a $2$-complement.

\begin{introthm}\label{introthm:Fano-coreg-0}
Let $(X,B,\bM .)$ be a projective generalized Fano type pair 
of absolute coregularity $0$.
Assume that the following conditions hold: 
\begin{itemize}
    \item the coefficients of $B$ are standard; and 
    \item the b-nef divisor $2\bM.$ is b-Cartier.
\end{itemize}
Then, there exists a boundary $B^+ \geq B$ satisfying the following conditions:
\begin{itemize}
    \item the generalized pair $(X,B^+,\bM.)$ is generalized log canonical;
    \item we have that $2(K_X+B^{+}+\bM X.)\sim 0$; and 
    \item the equality ${\rm coreg}(X,B^+,\bM.)=0$ holds.
\end{itemize}
\end{introthm}

In the case that $B=\bM.=0$, the previous theorem says that, for a Fano variety $X$ of absolute coregularity $0$, the linear system $|-2K_X|$ contains an element with nice singularities. 
In particular, the linear system $|-2K_X|$ is non-empty.
In~\cite[\S~8]{Tot24}, Totaro investigates Fano varieties 
with large {\em bottom weight}, which is
the smallest positive integer $m$ for which $H^0(X,-mK_X)\neq 0$.
In particular,~\cite[Theorem 8.1]{Tot24} 
implies the existence of a Fano $4$-fold that does not admit an $m$-complement for $m\leq 1799233$.
This shows that the constant $N(4)$ obtained by Birkar in~\cite{Bir19} is at least $1799233$.
More generally,~\cite[Theorem 8.1]{Tot24}
shows that $N(d)$ grows at least doubly exponentially with $d$.
In contrast to this, our statement shows that a Fano variety of absolute coregularity $0$ either admits a $1$-complement
or a $2$-complement.
In~\cite[Example 3.15]{Mor24}, the third author gives an example of a Fano surface of absolute coregularity $0$ for which there is no $1$-complement.
Thus, the previous theorem is sharp.
In the case of absolute coregularity $1$, we obtain a similar result.

\begin{introthm}\label{introthm:Fano-coreg-1}
Let $(X,B,\bM.)$ be a projective generalized Fano type pair of absolute coregularity $1$.
Assume that the following conditions hold:
\begin{itemize}
    \item the coefficients of $B$ are standard; and 
    \item the b-nef divisor $2\bM.$ is b-Cartier.
\end{itemize}
Then, there exists a boundary $B^+\geq B$ satisfying the following conditions:
\begin{itemize}
    \item the generalized pair $(X,B^+,\bM .)$ is generalized log canonical;
    \item we have that $N(K_X+B^{+}+\bM.)\sim 0$, where $N\in \{1,2,3,4,6\}$; and 
    \item the equality ${\rm coreg}(X,B^+,\bM.)=1$ holds.
\end{itemize}
\end{introthm}

In Table~\ref{table:dcc}, we summarize the 
unconditional theorems regarding complements of Fano varieties.
The entry $(d,c)$ in the table corresponds to the minimum integer $N_{d,c}$ for which
every $d$-dimensional Fano variety of coregularity $c$ admits at most a $N_{d,c}$-complement. In the blank spots, there set of such Fano varieties is empty.
Due to the work of Liu~\cite{Liu23}, we know that $N_2\leq {10^{10}}^{11}$.
However, it is expected that we can take $N_2=66$.
By the work of Totaro~\cite{Tot24}, we know that $N_d$ grows at least doubly exponentially with $d$.
\begin{table}[ht]
\begin{tabular}{
>{\columncolor[HTML]{FFFFFF}}c 
>{\columncolor[HTML]{FFFFFF}}l 
>{\columncolor[HTML]{FFFFFF}}l lllll
>{\columncolor[HTML]{FFFFFF}}l }
\multicolumn{1}{l}{\cellcolor[HTML]{FFFFFF}Coregularity} & \cellcolor[HTML]{FFFFFF}  & \cellcolor[HTML]{FFFFFF}      & \cellcolor[HTML]{FFFFFF}      & \cellcolor[HTML]{FFFFFF}                             & \cellcolor[HTML]{FFFFFF}      & \cellcolor[HTML]{FFFFFF}      & \cellcolor[HTML]{FFFFFF}{\color[HTML]{FFFFFF} } &                         \\
7                                                        &                           &                               & \cellcolor[HTML]{FFFFFF}      & \cellcolor[HTML]{FFFFFF}                             & \cellcolor[HTML]{FFFFFF}      & \cellcolor[HTML]{FFFFFF}      & \cellcolor[HTML]{C0C0C0}$N_7$                   &                         \\
6                                                        &                           &                               & \cellcolor[HTML]{FFFFFF}      & \cellcolor[HTML]{FFFFFF}                             & \cellcolor[HTML]{FFFFFF}      & \cellcolor[HTML]{C0C0C0}$N_6$ & \cellcolor[HTML]{9B9B9B}$N_7$                   &                         \\
5                                                        &                           &                               & \cellcolor[HTML]{FFFFFF}      & \cellcolor[HTML]{FFFFFF}                             & \cellcolor[HTML]{C0C0C0}$N_5$ & \cellcolor[HTML]{9B9B9B}$N_6$ & \cellcolor[HTML]{C0C0C0}$N_7$                   & {\color[HTML]{FFFFFF} } \\
4                                                        &                           &                               & \cellcolor[HTML]{FFFFFF}      & \cellcolor[HTML]{C0C0C0}$N_4$                        & \cellcolor[HTML]{9B9B9B}$N_5$ & \cellcolor[HTML]{C0C0C0}$N_6$ & \cellcolor[HTML]{9B9B9B}$N_7$                   & {\color[HTML]{FFFFFF} } \\
3                                                        &                           &                               & \cellcolor[HTML]{C0C0C0}$N_3$ & \cellcolor[HTML]{9B9B9B}$N_4$                        & \cellcolor[HTML]{C0C0C0}$N_5$ & \cellcolor[HTML]{9B9B9B}$N_6$ & \cellcolor[HTML]{C0C0C0}$N_7$                   &                         \\
2                                                        &                           & \cellcolor[HTML]{C0C0C0}$N_2$ & \cellcolor[HTML]{9B9B9B}$N_2$ & \cellcolor[HTML]{C0C0C0}{\color[HTML]{000000} $N_2$} & \cellcolor[HTML]{9B9B9B}$N_2$ & \cellcolor[HTML]{C0C0C0}$N_2$ & \cellcolor[HTML]{9B9B9B}$N_2$                   &                         \\
1                                                        &                           & \cellcolor[HTML]{9B9B9B}6     & \cellcolor[HTML]{C0C0C0}6     & \cellcolor[HTML]{9B9B9B}6                            & \cellcolor[HTML]{C0C0C0}6     & \cellcolor[HTML]{9B9B9B}6     & \cellcolor[HTML]{C0C0C0}6                       &                         \\
0                                                        & \cellcolor[HTML]{9B9B9B}1 & \cellcolor[HTML]{C0C0C0}2     & \cellcolor[HTML]{9B9B9B}2     & \cellcolor[HTML]{C0C0C0}2                            & \cellcolor[HTML]{9B9B9B}2     & \cellcolor[HTML]{C0C0C0}2     & \cellcolor[HTML]{9B9B9B}2                       &                         \\
\multicolumn{1}{r}{\cellcolor[HTML]{FFFFFF}}             & \cellcolor[HTML]{FFFFFF}1 & \cellcolor[HTML]{FFFFFF}2     & \cellcolor[HTML]{FFFFFF}3     & \cellcolor[HTML]{FFFFFF}4                            & \cellcolor[HTML]{FFFFFF}5     & \cellcolor[HTML]{FFFFFF}6     & \cellcolor[HTML]{FFFFFF}7                       & Dimension 
\end{tabular}
\caption{Dimension, coregularity, and complements.}
\label{table:dcc}
\end{table}

\subsection{Calabi--Yau pairs of higher coregularity} 
In the case of higher coregularity,
we need to deal with klt Calabi--Yau varieties
of higher dimensions.
We show that
controlling the index of log Calabi--Yau
pairs of coregularity $c$
can be reduced to a problem
about $c$-dimensional klt Calabi--Yau pairs. 
In order to state our next theorem, we need
to introduce two conjectures 
about Calabi--Yau pairs.
The first one is the boundedness of the index
for klt Calabi--Yau pairs.

\begin{introcon}\label{conj:index}
Let $d$ be a positive integer and 
let $\Lambda$ be a set of rational numbers
satisfying the descending chain condition.
There exists a constant $I \coloneqq I(\Lambda,d)$,
satisfying the following property.
For every projective $d$-dimensional klt log Calabi--Yau pair $(X,B)$
such that $B$ has coefficients in $\Lambda$,
we have that 
\[
I(\Lambda,d)(K_X+B)\sim 0.
\]
\end{introcon}

The previous conjecture is stated in~\cite[Conjecture 2.33]{Bir23}.
The second conjecture is known as the boundedness of B-representations.
It predicts that the birational automorphism group of a log Calabi--Yau pair
acts on the sections of $I(K_X+B)\sim 0$
with bounded order (see, e.g.,~\cite[Conjecture 3.2]{Fuj01}). 

\begin{introcon}\label{conj:b-rep}
Let $d$ and $I$ be two positive integers.
There is a constant $b\coloneqq b(d,I)$
satisfying the following property.
For every projective $d$-dimensional klt log Calabi--Yau pair $(X,B)$
with $I(K_X+B)\sim 0$,
the image of ${\rm Bir}(X,B)$
in $GL(H^0(I(K_X+B)))\simeq \kk^*$ 
is finite and has order at most $b$.
\end{introcon}

Now, we can state our main theorem
about the index of log Calabi--Yau pairs.
It shows that the boundedness of the index 
of generalized log Calabi--Yau pairs of 
coregularity $c$
can be reduced to the previous two conjectures in dimension $c$.

\begin{introthm}\label{introthm:index-higher-coreg}
Let $c$ and $p$ be positive integers
and $\Lambda\subset \qq$ be a set satisfying
the descending chain condition.
Assume that Conjecture~\ref{conj:index}
and Conjecture~\ref{conj:b-rep} 
hold in dimension $c$.
There is a constant $I \coloneqq I(\Lambda,c,p)$ satisfying the following property.
Let $(X,B,\bM.)$ be a projective generalized log Calabi--Yau pair of coregularity $c$ for which:
\begin{itemize}
    \item either $X$ is rationally connected or $\bM.=0$;
    \item the coefficients of $B$ are contained in $\Lambda$; and
    \item the b-nef divisor $p\bM.$ is b-Cartier.
\end{itemize}
Then, we have that
$I(K_X+B+\bM X.)\sim 0$.
\end{introthm}

\subsection{Fano varieties of higher coregularity.}
In the case of Fano varieties of higher absolute coregularity, we get the boundedness of complements 
of Fano type varieties with bounded absolute coregularity
subject to the previous conjectures.

\begin{introthm}\label{introthm:Fano-coreg-c}
Let $c$ and $p$ be positive integers
and $\Lambda\subset \qq$ be a closed set
satisfying the descending chain condition.
Assume that Conjecture~\ref{conj:index}
and Conjecture~\ref{conj:b-rep} hold in dimension $c$. 
There is a constant $N \coloneqq N(\Lambda,c,p)$ satisfying the following.
Let $(X,B,\bM.)$ be a projective generalized Fano type pair 
of absolute coregularity $c$ for which:
\begin{itemize}
    \item the divisor $B$ has coefficients in $\Lambda$; and
    \item the b-nef divisor $p\bM.$ is Cartier where it descends. 
\end{itemize}
Then, there exists a boundary 
$B^+\geq B$ satisfying the following conditions:
\begin{itemize}
    \item the generalized pair $(X,B^+,\bM.)$ is generalized log canonical;
    \item we have that $N(K_X+B^{+}+\bM X.)\sim 0$; and 
    \item the equality
    ${\rm coreg}(X,B^+,\bM.)=c$ holds.
\end{itemize}
\end{introthm}

We stress that Conjecture~\ref{conj:index} is known up to dimension $3$.
On the other hand, Conjecture~\ref{conj:b-rep} is known up to dimension $2$.
In particular, both
Theorem~\ref{introthm:index-higher-coreg} and Theorem~\ref{introthm:Fano-coreg-c}
hold unconditionally in the case of coregularity 2.

\subsection{Coregularity and the canonical bundle formula}
The canonical bundle formula plays a fundamental role in the theory of complements. 
In many cases, we need to lift complements from the base of a log Calabi--Yau fibration.
We will prove the following statement
that relates the canonical bundle formula with the coregularity.

\begin{introthm}\label{introthm:cbf-and-coreg}
Let $c$ and $p$ be nonnegative integers
and $\Lambda \subset \qq$ be a set 
satisfying the descending chain condition.
Assume that Conjecture~\ref{conj:index} and Conjecture~\ref{conj:b-rep}
hold in dimension at most $c-1$.
There exists a set $\Omega \coloneqq \Omega(\Lambda,c,p) \subset \qq$ 
satisfying the descending chain condition 
and a positive integer $q \coloneqq q(\Lambda,c,p)$,
satisfying the following property.
Let $\pi \colon X\rightarrow Z$ be a Fano type morphism between projective varieties.
Let $(X,B,\bM.)$ be a projective generalized pair of coregularity $c$ for which:
\begin{itemize}
    \item the generalized pair $(X,B,\bM.)$ is log Calabi--Yau over $Z$;
    \item the coefficients of $B$ belong to $\Lambda$;
    \item the b-nef divisor $p\bM.$ is Cartier where it descends;
    \item the b-nef divisor $\bM.$ is $\qq$-trivial on the general fiber of $X\rightarrow Z$;
    \item every generalized log canonical center of $(X,B,\bM.)$ is a log canonical center of $(X,B)$; and
    \item every log canonical center of $(X,B)$ dominates $Z$.
\end{itemize}
Then, we can write
\[
q(K_X+B+\bM X.) \sim q\pi^*(K_Z+B_Z+\bN Z.),
\]
where the following conditions hold:
\begin{itemize}
    \item $B_Z$ is the discriminant part of the adjunction for $(X,B,\bM.)$ over $Z$;
    \item the coefficients of $B_Z$ belong to $\Omega$; and
    \item the b-nef divisor $q\bN.$ is b-Cartier.
\end{itemize}
\end{introthm}
 
\subsection{Kawamata log terminal singularities} 
Finally, we show some applications
of the previous theorems of this article
to the study of klt singularities. 
We obtain the following result
about klt singularities
of absolute coregularity $0$.

\begin{introthm}
\label{introthm:klt-0}
Let $(X;x)$ be a klt singularity of
absolute coregularity $0$.
Then, there exists a boundary 
$\Gamma$ through $x$
satisfying the following conditions:
\begin{itemize}
\item we have that $2(K_X+\Gamma)\sim 0$
on a neighborhood of $x$; and
\item the coregularity of $(X,\Gamma)$
at $x$ is equal to $0$.
\end{itemize}
In particular, the pair $(X,\Gamma;x)$ is strictly log canonical at $x$.
\end{introthm}

Analogously, we obtain a similar result in the context of klt singularities
of absolute coregularity $1$.

\begin{introthm}
\label{introthm:klt-1}
Let $(X;x)$ be a klt singularity
of absolute coregularity $1$. 
Then, there exists a boundary
$\Gamma$ through $x$ 
satisfying the following conditions:
\begin{itemize}
\item we have that $N(K_X+\Gamma)\sim 0$
for some $N\in \{1,2,3,4,6\}$; and 
\item the coregularity of $(X,\Gamma)$ at $x$ is equal to $1$.
\end{itemize} 
In particular, the pair $(X,\Gamma;x)$ is strictly log canonical at $x$.
\end{introthm} 

The two previous theorems generalize the $A,D,E$-type classification 
of klt surface singularities to higher dimensional klt singularities.
The $A$-type klt surface singularities are the toric surface singularities.
In the Gorenstein case, these are the $A_n$-singularities.
The $A$-type singularities are the klt surface singularities
of absolute coregularity $0$ 
that admit a $1$-complement.
The $D$-type klt surface singularities
are quotients of toric singularities via an involution.
In the Gorenstein case, these are $D_n$-singularities.
The $D$-type singularities are 
the surface singularities of absolute coregularity $0$ that admit a $2$-complement but no $1$-complement.
The $E$-type klt surface singularities
are the {\em exceptional surface singularities}.
In the Gorenstein case, these are exactly
the $E_6$, $E_7,$ and $E_8$ singularities.
These are the klt surface singularities
that have absolute coregularity $1$.
They admit a $3$, $4$ or $6$-complement but not a $1$-complement or $2$-complement.
The aforementioned results about complements and coregularity of $2$-dimensional klt singularities are proved in~\cite[Section 3.2]{Mor24}.

\subsection{On the techniques of the article}
The theory of complements was introduced by Shokurov in the early 2000s (see, e.g.,~\cite{Sho00}),
although these objects already appeared in the work of Keel and M\textsuperscript{c}Kernan on quasi-projective surfaces~\cite{KM98}.
In this work, complements were called {\em tigers}.
Since then, it has been understood that vanishing theorems, the canonical bundle formula, and the minimal model program
are indispensable tools to produce complements on a variety (see, e.g.,~\cite{PS09,PS11,KM98}).
Using the aforementioned techniques,
the language of generalized pairs, 
and the boundedness of exceptional Fano varieties,
Birkar proved the boundedness of complements
for Fano varieties~\cite{Bir19}.
Since then, the theory of complements has been expanded to Fano pairs with more general coefficients~\cite{FM20,HLS19,Sho20}, 
to log canonical Fano varieties~\cite{Xu20}, and to
log Calabi--Yau $3$-folds~\cite{FMX19}.
In this article, we study the theory of complements through the lens of the coregularity.
The techniques are similar to the ones in the aforementioned papers.
However, in order to obtain novel results,
we need to re-prove several parts of this theory
keeping track of this new invariant.
The fact that our results are independent of dimension imposes an extra difficulty.
At the same time, we will need to use some recent results regarding the coregularity and its connections to singularities~\cite{FMP22}
and Calabi--Yau pairs~\cite{FMM22}.
\color{black}

\subsection*{Acknowledgements}
This project was initiated in the \href{https://web.math.princeton.edu/~jmoraga/Learning-Seminar-MMP}{Minimal Model Program Learning Seminar}.
The authors would like to thank Mirko Mauri for many discussions that led to some of the ideas of this article.

\section*{Strategy of the proof}

In this section, we give a sketch of the proof of the main theorems of this article, namely
Theorem~\ref{introthm:index-higher-coreg}, 
Theorem~\ref{introthm:Fano-coreg-c}, and
Theorem~\ref{introthm:cbf-and-coreg}. 
The other theorems will be obtained using the same strategy and an analysis of the coefficients throughout the proof. 
We write Theorem $X(c)$ for Theorem $X$ in coregularity at most $c$.
Theorem~\ref{introthm:index-higher-coreg}$(0)$ follows from~\cite[Theorem 1]{FMM22},
while Theorem~\ref{introthm:cbf-and-coreg}$(0)$ is trivial.
We will prove the following four statements:

\begin{itemize} 
\item[(i)] Theorem~\ref{introthm:Fano-coreg-c}$(0)$ holds;
\item[(ii)] Theorem~\ref{introthm:Fano-coreg-c}$(c-1)$ implies 
Theorem~\ref{introthm:cbf-and-coreg}$(c)$;
\item[(iii)] Theorem~\ref{introthm:cbf-and-coreg}$(c)$ implies
Theorem~\ref{introthm:index-higher-coreg}$(c)$; and
\item[(iv)] Theorem~\ref{introthm:index-higher-coreg}$(c)$ and
Theorem~\ref{introthm:cbf-and-coreg}$(c)$ imply Theorem~\ref{introthm:Fano-coreg-c}$(c)$.
\end{itemize} 

We write Theorem $X(d,c)$ for Theorem $X$ in dimension at most $d$ and coregularity at most $c$. 
For instance, Theorem~\ref{introthm:Fano-coreg-c}$(d,c)$ is known by~\cite[Theorem 1.7]{Bir19}. 
Thus, we may write $N(\Lambda,d,c,p)$ for the positive integer provided by Theorem~\ref{introthm:Fano-coreg-c}$(d,c)$. We may suppress $\Lambda$ and $p$ from the notation whenever they are clear from the context.
Our aim is to show that, once we fix $c$,
there is an upper bound $N(c)$ for all $N(d,c)$.
Similarly, Theorem~\ref{introthm:cbf-and-coreg}$(d,c)$ is known due to~\cite[Proposition 6.3]{Bir19}. 
We write $q(d,c)$ for the constant provided by Theorem~\ref{introthm:cbf-and-coreg}$(d,c)$ and we show that $q(d,c)$ is bounded above by a constant only depending on $c$.
Theorem~\ref{introthm:index-higher-coreg}$(d,c)$ is not known even if we fix the dimension.
In this case, the aim is twofold: to prove the existence of an upper bound $I(d,c)$ for fixed dimension $d$, and to show that all the $I(d,c)$ are bounded above in terms of $c$.
The proof of implication (i) is similar to that of (iv). 
In the following three subsections, we sketch the proofs of (ii), (iii), and (iv).

\subsection*{A canonical bundle formula}
Let $(X,B,\bM.)\rightarrow Z$ be as in the setting of Theorem~\ref{introthm:cbf-and-coreg}$(d,c)$.

First, we show that for every $z\in Z$ closed, 
we may find a relative $N(c-1)$-complement for $(X,B,\bM.)$ over $z$.
We pick an effective Cartier divisor $E$ on $Z$ through $z$.
We let $t$ be the largest positive number for which $(X,B+t\pi^*E,\bM.)$
has generalized log canonical singularities around $z$.
By the connectedness theorem, the coregularity of $(X,B+t\pi^*E,\bM.)$ is at most $c-1$.
Indeed, since all the generalized log canonical centers of $(X,B,\bM.)$ are horizontal over $Z$, introducing a vertical generalized log canonical center will strictly decrease the coregularity. 
Taking a dlt modification of $(X,B+t\pi^*E,\bM.)$, we can produce a new generalized pair
$(X',B',\bM.)$ such that $z$ is contained in the image of a component $S$ of $\lfloor B'\rfloor$.
By perturbing the coefficients, we may assume that the coefficients of $B'$ belong to $\Lambda$.
We replace $(X,B,\bM.)$ by $(X',B',\bM.)$ and assume there is a vertical divisorial log canonical center $S$.
Notice that this replacement changes the crepant birational class of the original generalized pair $(X,B,\bM.)$ in order to create a new log canonical center.
Running a suitable MMP over $Z$, we reduce to the case in
which $S\rightarrow \pi(S)\ni z$ is a Fano type morphism.
The generalized pair $(S,B_S,\bM S.)$ obtained by adjunction
of $(X,B,\bM.)$ to $S$ has dimension at most $d-1$
and coregularity at most $c-1$.
If $q(S)=z$, then we may apply Theorem~\ref{introthm:Fano-coreg-c}$(d-1,c-1)$
to conclude that $(S,B_S,\bM S.)$ admits an $N(d-1,c-1)$-complement.
Since we are assuming Theorem~\ref{introthm:Fano-coreg-c}$(c-1)$, this is also an $N(c-1)$-complement.
If $\dim \pi(S)\geq 1$, then we construct an $N(c-1)$-complement by induction on the dimension.
In any case, we obtain an $N(c-1)$-complement for $(S,B_S,\bM S.)$ around $z$.
Using Kawamata--Viehweg vanishing, we lift such complement
to an $N(c-1)$-complement for $(X,B,\bM.)$ around the fiber of $z\in Z$.
The details of this proof can be found in \S~\ref{sec:rel}, where we discuss relative complements.
In \S~\ref{subsec:lifting}, we explain how to lift complements from divisors.

Now, we can assume the existence of bounded relative $N(c-1)$-complements for $(X,B,\bM.)\rightarrow Z$.
The existence of bounded relative complements 
allows us to find $q$ in the statement of
Theorem~\ref{introthm:cbf-and-coreg}$(d,c)$.
Indeed, we can take $q(d,c)=N(c-1)$.
The main difficulty is to control 
the coefficients of $\bN Z.$ in the model where it descends.
In order to do so, we will cut the base with hypersurfaces to reduce to the case in which the base is a curve.
Once the base is a curve $C$, we will study the coefficients of a relative complement over a closed point $c\in C$.
Analyzing the coefficients of this relative complement will
show that $q\bN Z.$ is integral. 
A similar argument on a suitable resolution 
$Z'\rightarrow Z$ proves that 
$q\bN Z'.$ is integral, where $Z'$ is a model on which $\bN Z.$ descends.
This finishes the proof of Theorem~\ref{introthm:cbf-and-coreg}$(c)$ using
Theorem~\ref{introthm:Fano-coreg-c}$(c-1)$.
The details of this proof are given in \S~\ref{sec:cbf}.

\subsection*{Index of log Calabi--Yau pairs} Let $(X,B,\bM.)$ be a generalized log Calabi--Yau pair as in Theorem~\ref{introthm:index-higher-coreg}$(d,c)$. 
By~\cite[Theorem 2]{FMP22}, 
we may assume that the set $\Lambda$ in the statement of the theorem is finite.
By~\cite[Theorem 4.2]{FS23}, we can replace $(X,B,\bM.)$ by a Koll\'ar--Xu model (see \S~\ref{subsec:kollar-xu-models}).
We have a Fano type contraction
$q\colon X\rightarrow Z$
such that all the generalized log canonical centers of $(X,B,\bM.)$ dominate the base $Z$.
Both the index and Weil index of $K_X+B+\bM X.$ are preserved by the Koll\'ar--Xu model.
We will proceed with the proof in three different cases, 
depending on the dimension of the base of the Koll\'ar--Xu model
and the coefficient sets of $B$ and $\bM.$.
We argue by induction on the dimension $d$ of $X$.
The base of the induction is the klt case which follows by Conjecture~\ref{conj:index} (see, e.g., Lemma~\ref{lem:index-gen-klt}).\\

\noindent\textit{Case 1}: The moduli part $\bM.=0$.

In this case, we know that $K_X+B\sim_\qq 0$. 
We do not assume that $X$ is rationally connected.
We choose a component $S$ of $\lfloor B\rfloor$ 
and run a $(K_X+B-\epsilon S)$-MMP.
This minimal model program terminates with a Mori fiber
space on which $S$ is ample over the base.
Observe that the variety $S$ may not be normal.
However, the pair obtained by adjunction
$(S,B_S)$ is semi-log canonical.
In~\S~\ref{sec:slc}, we show that the statement of
Theorem~\ref{introthm:index-higher-coreg}$(d-1,c)$ holds for semi-log canonical pairs provided
it holds for log canonical pairs.
To do so, we use Conjecture~\ref{conj:index} and Conjecture~\ref{conj:b-rep}
in dimension $c$.
Here, it is crucial that we work with pairs instead of generalized pairs.
Indeed, Conjecture~\ref{conj:b-rep} is not known for generalized pairs, even in dimension $2$.
Hence, we conclude that $I(\Lambda,d-1,c,0)(K_S+B_S)\sim 0$.

Thus, in this case, we conclude that the index of $(X,B)$ is at most $I(\Lambda,d-1,c,0)$.\\

\noindent\textit{Case 2:} The base $Z$ of the Koll\'ar--Xu model
is positive dimensional,
the divisor $\{B\}+\bM X.$ is trivial on the general fiber of $X\to Z$, 
and the b-nef divisor $\bM.$ is non-trivial.

In this case, we apply Theorem~\ref{introthm:cbf-and-coreg}$(c)$.
We can write 
\begin{equation}\label{eq:cbf-intro}
q(K_X+B+\bM X.) \sim q \pi^*(K_Z+B_Z+\bN Z.).
\end{equation} 
The variety $Z$ has dimension at most $c$.
The integer $q$ only depends on $\Lambda,c$ and $p$.
The coefficients of $B_Z$ belong to a DCC set that only depends on $\Lambda,c$ and $p$. 
The b-nef divisor $q\bN.$ is b-Cartier.
The variety $X$ is rationally connected, as 
we are assuming that the b-nef divisor $\bM.$ is non-trivial.
Hence, $Z$ is also rationally connected.
Let $Z'\rightarrow Z$ be the model where $\bN Z.$ descends.
In particular, $Z'$ is rationally connected.
Note that in general, $\bN Z'.$ may have torsion components.
However, since $Z'$ is rationally connected,
the $q$-th multiple of such torsion components
are linearly equivalent to zero (see~\cite[Corollary 3.9]{FMM22}).
Using Conjecture~\ref{conj:index}, we will show that
the index of $K_Z+B_Z+\bN Z.$ only depends on $\Lambda,c$ and $p$.
Thus, by the linear equivalence~\eqref{eq:cbf-intro},
we conclude that the index of $(X,B,\bM.)$ is bounded above
by a constant $I_0(\Lambda,c,p)$.\\

\noindent\textit{Case 3:} The divisor
$\{B\}+\bM X.$ is non-trivial on the general fiber of $X\rightarrow Z$ and the b-nef divisor $\bM.$ is non-trivial.

We run a $(K_X+\lfloor B\rfloor)$-MMP over $Z$.
Since $K_X+\lfloor B\rfloor$ is not pseudo-effective over $Z$,
this minimal model program terminates with a Mori fiber space
$p\colon X'\rightarrow W$ over $Z$.
We denote by $B'$ the push-forward of $B$ on $X'$.
the divisor $K_{X'}+\lfloor B'\rfloor$ is anti-ample over $W$.
Since $\lfloor B\rfloor$ is big over $Z$, the divisor $\lfloor B'\rfloor$ has a component $S$ that dominates $W$.
By construction, the general fibers of $S\rightarrow W$ are of Fano type.
In this case, $X$ and $X'$ are rationally connected,
as we are assuming that the b-nef divisor $\bM.$ is non-trivial.
Hence, the image $W$ of $X'$ is rationally connected.
Since a general fiber of $S\rightarrow W$ is of Fano type, they are rationally connected.
Thus, $S$ is rationally connected, being the base
and general fibers of $S\rightarrow W$ rationally connected.
In particular, if $(S,B_S+\bM S.)$ is the generalized pair obtained by adjunction, then we know that $I(\Lambda,d-1,c,p)(K_S+B_S+\bM S.)\sim 0$.
Here, we argued by induction on the dimension and used Theorem~\ref{introthm:index-higher-coreg}$(d-1,c)$.
Depending on the dimension of $W$, we either use
Kawamata--Viehweg vanishing or Koll\'ar's torsion-free theorem to conclude that
$I(\Lambda,d-1,c,p)(K_{X'}+B'+\bM X'.)\sim 0$.
Hence, the index of $(X,B,\bM.)$ is at most $I(\Lambda,d-1,c,p)$.
These lifting arguments are explained in 
\S~\ref{subsec:lifting}.\\

In summary, a generalized log Calabi--Yau pair $(X,B,\bM.)$ as in Theorem~\ref{introthm:index-higher-coreg}$(d,c)$
has index at most
\[
\max\{I_0(\Lambda,c,p),I(\Lambda,d-1,c,p),I(\Lambda,d-1,c,0)\}.
\]
Thus, we have that
\[
I(\Lambda,d,c,p) 
\leq
\max\{I_0(\Lambda,c,p),I(\Lambda,d-1,c,p),I(\Lambda,d-1,c,0)\}.
\]
Hence, there is an upper bound for $I(\Lambda,d,c,p)$ which only depends on $\Lambda,c$, and $p$.
This finishes the sketch of the proof 
of Theorem~\ref{introthm:index-higher-coreg}$(c)$ using
Theorem~\ref{introthm:cbf-and-coreg}$(c)$.

\subsection*{Complements on Fano varieties}
Let $(X,B,\bM.)$ be a Fano type pair as in Theorem~\ref{introthm:Fano-coreg-c}$(d,c)$.
In \S~\ref{sec:reduct-finite-coeff}, we reduce to the case in which $\Lambda$ finite.
This is crucial for lifting complements from divisors (see \S~\ref{subsec:lifting}).
By the assumption on the absolute coregularity of $(X,B,\bM.)$, we may find a generalized log Calabi--Yau structure
$(X,B+\Gamma,\bM.)$ of coregularity $c$.
By dimensional reasons and the assumption on the absolute coregularity of $(X,B,\bM.)$, $(X,B+\Gamma,\bM.)$ may be generalized klt only if $d = c$ and $(X,B,\bM.)$ is exceptional;
this case is settled by \cite[Theorem 1.7]{Bir19} in dimension $c$.
Therefore, in the rest of this sketch, we may assume that $(X,B+\Gamma,\bM.)$ is not generalized klt.
Let $(Y,B_Y+\Gamma_Y+E,\bM.)$ be a dlt modification of $(X,B+\Gamma,\bM.)$.
Here, $B_Y$ (resp. $\Gamma_Y$) is the strict transform 
of the fractional part of $B$ (resp. $\Gamma$),
while we set $E=\lfloor B_Y+\Gamma_Y+E \rfloor$.
Since $(X,B+\Gamma,\bM.)$ is not generalized klt, we have $E \neq 0$.
Since $X$ is of Fano type, it easily follows that so is $Y$.
In particular, $Y$ is a Mori dream space.
We run a $-(K_Y+B_Y+E+\bM Y.)$-MMP. 
Note that $-(K_Y+B_Y+E+\bM Y.)$ is a pseudo-effective divisor.
Hence, this minimal model program
must terminate with a good minimal model $Z$.
We let $B_Z$ and $E_Z$ be the push-forwards to $Z$ of $B_Y$ and $E$, respectively.
In order to produce a complement 
for $(X,B)$,
it suffices to produce a complement for $(Z,B_Z+E_Z,\bM Z.)$.
Replacing $(X,B,\bM.)$ by $(Z,B_Z+E_Z,\bM Z.)$, 
we may assume that $-(K_X+B+\bM X.)$ is semi-ample
and ${\rm coreg}(X,B,\bM.)=c$.
Notice that this reduction does not alter the coefficients set for the boundary part of $(X,B,\bM.)$, since the only divisors that may have been introduced in the boundary have coefficient 1.
Furthermore, by the choice of the MMP run, it follows that $E$ cannot be contracted.
In particular, after this reduction, we may assume that $\lfloor B \rfloor \neq 0$.
We will proceed in three different cases depending on the dimension of the ample model $W$ of the divisor $-(K_X+B+\bM X.)$.\\

\noindent\textit{Case 1}: The dimension of $W$ is $0$.

In this case, we have that 
$K_X+B+\bM X.\sim_\qq 0$. 
Hence, producing a complement for $(X,B,\bM.)$
is the same as controlling the index of the generalized pair. 
Thus, the statement follows from 
Theorem~\ref{introthm:index-higher-coreg}$(c)$.\\

\noindent\textit{Case 2}: The dimension of $W$ is $d$.

In this case, we have that $-(K_X+B+\bM X.)$ is semi-ample and big. 
Furthermore, the round-down $\lfloor B\rfloor$
is non-trivial.
We pass to a suitable birational model of $(X,B,\bM.)$ where a component $S$ of $\lfloor B\rfloor$ is of Fano type.
Performing adjunction to $S$, we obtain a log Fano pair of dimension $d-1$ and coregularity $c$.
Using Theorem~\ref{introthm:Fano-coreg-c}$(d-1,c)$, we produce an $N(\Lambda,d-1,c,p)$-complement on $S$ that can be lifted to an $N(\Lambda,d-1,c,p)$-complement of $(X,B)$.\\

\noindent\textit{Case 3}: The dimension of $W$ is positive and strictly less than $d$.

The fibration $\pi\colon (X,B,\bM.)\rightarrow W$ is a log Calabi--Yau fibration for $(X,B,\bM.)$. 
If $\{B\}+\bM.$ is big over $W$, then by perturbing the coefficients of $B$ we reduce to Case 2.
Otherwise, we may replace $W$ with the ample model
of $\{B\}+\bM.$ over $W$.
Doing so, we may assume $\{B\}+\bM.$ is trivial on the general fiber of $X\rightarrow W$.
If all the generalized log canonical centers of $(X,B)$ dominate $W$, then we are in the situation of Theorem~\ref{introthm:cbf-and-coreg}$(c)$.
The generalized pair $(W,B_W,\bN W.)$  induced on the base
is of Fano type and exceptional. 
By~\cite[Theorem 1.7]{Bir19} in dimension $c$ or less, we can find an $N(\Omega,c)$-complement for $(W,B_W,\bN W.)$. 
Here, $\Omega$ only depends on $\Lambda,c$, and $p$.
Then, we can pull the complement back via $\pi$ to obtain an $N(\Omega,c)$-complement for $(X,B,\bM.)$.
Finally, we may assume that $\{B\}+\bM.$ is trivial on the general fiber of $X\rightarrow W$ and there is some component $S\subset \lfloor B\rfloor$ that is vertical over $W$.
In this case, $B_{\rm hor}$ is big over $W$.
Here $B_{\rm hor}$ stands for the sum of the components of $B$ which are horizontal over $W$.
Again, we can perturb the coefficients of $B$ to reduce to Case 2.\\

In summary, a generalized pair $(X,B,\bM.)$ as in Theorem~\ref{introthm:Fano-coreg-c}$(d,c)$ admits an $N$-complement, 
where $N\leq \max\{N(\Omega,c),N(\Lambda,d-1,c,p)\}$.
Thus, we have 
\[
N(\Lambda,d,c,p)\leq \max\{N(\Omega,c),N(\Lambda,d-1,c,p)\}. 
\]
Hence, there is an upper bound for 
$N(\Lambda,d,c,p)$ which only depends on $\Lambda,c$, and $p$.
This finishes the proof of Theorem~\ref{introthm:Fano-coreg-c}$(c)$ using
Theorem~\ref{introthm:index-higher-coreg}$(c)$
and Theorem~\ref{introthm:cbf-and-coreg}$(c)$.

\section{Preliminaries}

We work over
an algebraically closed field $\kk$ of characteristic zero.
Our varieties are connected and quasi-projective
unless otherwise stated.
In this section, we introduce some preliminaries regarding
singularities,
Fano varieties,
Calabi--Yau pairs,
and coregularity.

\subsection{Divisors, b-divisors, and generalized pairs}
In this subsection, we recall some basics about b-divisors and generalized pairs.

\begin{definition}
{\rm Let $X$ be a normal variety. A {\em b-divisor} $\bM.$ on $X$ is a function which associates any birational map $X'\dashrightarrow X$ with an $\rr$-divisor $\bM X'.$ on $X'$. The set of divisors $\{\bM X'. \colon  X'\dashrightarrow X\}$ satisfies the following compatibility condition: if $g \colon  X_1\to X_2$ is a birational morphism over $X$, then $g_*\bM X_1. = \bM X_2.$.
We say that a b-divisor $\bM.$ on $X$ {\em descends} on some birational model $X'$ of $X$ if $\bM X'.$ is $\rr$-Cartier and $\bM.$ is equivalent to $( X'\to X, \bM X'.)$. In other words, for any birational map $h \colon Y\to X'$ over $X$, we have $h^*\bM X'. = \bM Y.$.
In the previous case, we say that $\bM.$ is a {\em b-$\rr$-Cartier divisor}.

Let $X\to Z$ be a projective morphism. The b-divisor $\bM.$ is said to be {\em b-Cartier} (resp. {\em b-nef, b-nef/$Z$}) if $\bM X'.$ is Cartier (resp. nef, relatively nef over $Z$) on some birational model $X'$ over $X$ where $\bM.$ descends.

The {\em b-Cartier closure} of an $\mathbb{R}$-Cartier divisor $M$ is a b-divisor $\bM .$ whose trace on every birational model $f:Y\rightarrow X$ is $f^*M$.}
\end{definition}

\begin{definition}
{\rm Let $X$ be a normal variety and $\pi \colon  X\to Z$ be a projective morphism. A {\em generalized pair on $X$ over $Z$} is a triple $(X,B,\bM.)$ where 
\begin{itemize}
    \item $B$ is an effective $\rr$-divisor on $X$;
    \item $\bM.$ is a b-nef/$Z$ b-$\rr$-Cartier on $X$; and
    \item $K_X+B+\bM X.$ is $\rr$-Cartier.
\end{itemize}
When $Z$ is a point, we simply call $(X,B,\bM.)$ a {\em generalized pair}. }
\end{definition}

\subsection{Singularities of generalized pairs}

In this subsection, we define the notions of singularities for generalized pairs.

\begin{definition}
{\rm Let $X$ be a normal variety and $(X,B,\bM.)$ be a generalized pair on $X$. Let $D$ be a divisor over $X$. Pick a log resolution $f \colon X'\to X$ of $(X,B)$ such that $D$ is a divisor on $X'$ and $\bM.$ descends on $X'$. We can write
\[
K_{X'} + B' + \bM X'. = f^*(K_X+B+\bM X.)
\]
for some uniquely determined $B'$. Define the {\em generalized log discrepancy} $a_D(X,B,\bM.)$ to be $1-\coeff_D (B')$. 

We say that $(X,B,\bM.)$ is {\em generalized log canonical} (resp. {\em generalized klt}) if $a_D(X,B,\bM.)$ is nonnegative (resp. positive) for any divisor $D$ over $X$. A {\em generalized non-klt place} (resp. {\em generalized log canonical place}) of  $(X,B,\bM.)$ is a prime divisor $D$ over $X$ with $a_D(X,B,\bM.)\leq 0$ (resp. $a_D(X,B,\bM.)=0$).
A {\em generalized non-klt center} of $(X,B,\bM.)$ is the image of a generalized non-klt place.
We denote the set of generalized non-klt centers of $(X,B,\bM.)$ by ${\rm Nklt}(X,B,\bM.)$. 
A {\em generalized log canonical center} of $(X,B,\bM.)$ is the image $Z$ of a generalized non-klt place such that every generalized non-klt place whose image on $X$ contains $Z$ is a generalized log canonical place.

We say that $(X,B,\bM.)$ is {\em generalized dlt} if it is generalized log canonical and satisfies the following condition: for any generalized log canonical center $V$ of $(X,B,\bM.)$, the pair $(X,B)$ is log smooth around the generic point of $V$ and $\bM.$ descends on $X$ in a neighborhood of the generic point of $V$. We say that $(X,B,\bM.)$ is {\em generalized plt} if it is generalized dlt and every connected component of $\lfloor B\rfloor$ is irreducible.

Let $(X,B,\bM.)$ be a generalized log canonical pair over a base $Z$.
Let $f \colon Y\to X$ be a birational morphism and write
\[
K_{Y} + B_Y + \bM Y. = f^*(K_X+B+\bM X.).
\]We say that $(Y,B_Y,\bM.)$ is a {\em $\qq$-factorial generalized dlt modification of $(X,B,\bM.)$} if the variety $Y$ is $\qq$-factorial, $(Y,B_Y,\bM.)$ is generalized dlt, and every $f$-exceptional divisor appears in $B_Y$ with coefficient 1. }
\end{definition}

\begin{lemma}[{\cite[Theorem 2.9]{FS23}}]
Every generalized log canonical pair over a base $Z$ has a $\qq$-factorial generalized dlt modification. 
\end{lemma}

The following lemma states that the singularities of the pair $(X,B)$ are milder
than the singularities of $(X,B,\bM.)$.

\begin{lemma}[{\cite[Remark 4.2.(3)]{BZ16}}]\label{lem:glc-implies-lc}
Let $(X,B,\bM.)$ be a generalized log canonical pair over $Z$.
Suppose $K_X+B$ is $\rr$-Cartier. Then for any divisor $D$ over $X$, the log discrepancies satisfy
\[a_D(X,B,\bM.) \leq a_D(X,B).\]
In particular, the pair $(X,B)$ is log canonical.
\end{lemma}

\subsection{Crepant birational maps}

In this subsection, we recall the notion of crepant birational map and group of crepant birational automorphisms.

\begin{definition}{\rm
    Let $(X_1,B_1,\bM.)$ and $(X_2,B_2,\bM.)$ be generalized pairs over $Z$.
    We say that they are \emph{crepant} if there exists a common resolution $\alpha_1 \colon X' \rar X_1$ and $\alpha_2 \colon X' \rar X_2$, where each $\alpha_i$ is proper, such that
    $$
    \K X'. + B'_1 + \bM X'. = \K X'. + B'_2 + \bM X'.,
    $$
    holds,
    where we have $\K X'. + B'_i + \bM X'.= \alpha_i^*(\K X_i. + B_i + \bM X_i.)$ for $i=1,2$. 
}
\end{definition}

In the case of pairs, we recall the notion of $B$-birational map, originally due to Fujino \cite[Definition 1.5]{Fuj00}.
Observe that, for our purposes in later sections, it is important to deal with possibly reducible varieties.

\begin{definition}\label{def_b-bir}
    {\rm
    Let $(X,\Delta) = \sqcup (X_i,\Delta_i)$ and $(X',\Delta') = \sqcup (X'_i,\Delta'_i)$ be possibly reducible normal pairs.
    We say that $f \colon X \dashrightarrow X'$ is a {\it $B$-birational map} if $(X,\Delta)$ and $(X',\Delta')$ are crepant.
    That is, $X$ and $X'$ have the same number of irreducible components, and there exists a permutation $\sigma$ of the index set of the irreducible component such that, for every $i$, the restriction $f_i \colon X_i \dashrightarrow X'_{\sigma(i)}$ is birational and $(X_i,\Delta_i)$ is crepant to $(X'_{\sigma(i)},\Delta'_{\sigma(i)})$.
    }
\end{definition}

\begin{definition}\label{def_bir}
{\rm
    Given a pair $(X,\Delta)=\sqcup (X_i,\Delta_i)$ as in Definition~\ref{def_b-bir}, we define
    $$
    \mathrm{Bir}(X,\Delta) \coloneqq \{f|f \colon (X,\Delta) \dashrightarrow (X,\Delta) \;  \text{is $B$-birational} \}.
    $$
    The set $\mathrm{Bir}(X,\Delta)$ forms a group under composition.
}   
\end{definition}
    
We observe that Definition~\ref{def_b-bir} and Definition~\ref{def_bir} naturally extend to the case of generalized pairs.

\subsection{Complements}
In this subsection, we introduce the notion of relative complements.

\begin{definition}
{\rm A {\em contraction} is a projective morphism of quasi-projective varieties $f \colon X \rar Z$ such that $f_* \O X.=\O Z.$.
Notice that, if $X$ is normal, then so is $Z$.
A {\em fibration} is a contraction $X \rar Z$ such that $\dim Z < \dim X$.}
\end{definition}

\begin{definition}
{\rm Let $(X,B)$ be a pair and $X\to Z$ a contraction. We say that a pair $(X,B)$ is {\em log Fano} (resp. {\em weak log Fano} or {\em log Calabi--Yau}) {\em over $Z$} if it is log canonical and $-(K_X+B)$ is ample over $Z$ (resp. $-(K_X+B)$ is nef and big over $Z$ or $K_X+B$ is $\rr$-trivial over $Z$). 

We say that $(X,B)$ is of {\em Fano type} (resp. {\em log Calabi--Yau type}) {\em over $Z$} if $(X,B+\Delta)$ is klt and weak log Fano (resp. log Calabi--Yau) for some choice of $\Delta\geq 0$. 

If $(X,0)$ is of Fano type (resp. log Calabi--Yau type) over $Z$, we say that $X\to Z$ is a {\em Fano type morphism} (resp. {\em log Calabi--Yau type morphism}). If $(X,B)$ is log Fano (resp. log Calabi--Yau, Fano type, Calabi--Yau type) over a point, we simply say that $(X,B)$ is {\em log Fano} (resp. {\em log Calabi--Yau, Fano type, Calabi--Yau type}). }
\end{definition}

\begin{definition}
{\rm Let $X\to Z$ be a contraction and $(X,B,\bM.)$ be a generalized pair over $Z$.
Let $N$ be a positive integer. An {\em $N$-complement of $K_X+B+\bM X.$ over a point $z\in Z$} is a divisor $K_X+B^++\bM X.$ such that over some neighborhood of $z$, we have:
\begin{itemize}
    \item $(X,B^+,\bM.)$ is generalized log canonical;
    \item $N(K_X+B^++\bM X.)\sim_Z 0$;
    \item $N\bM.$ is b-Cartier; and
    \item $B^+\geq B$.
\end{itemize}
If the above conditions hold for $K_X+B^++\bM X.$ over every $z\in Z$, we say that $K_X+B^++\bM X.$ is an {\em $N$-complement of $K_X+B+\bM X.$ over $Z$}. We say that $K_X+B^+ + \bM X.$ is a {\em $\qq$-complement of $K_X+B+\bM X.$ over $z\in Z$} (resp. {\em $\qq$-complement of $K_X+B+\bM X.$ over $Z$}) it is a $q$-complement for some $q\in \zz_{>0}$. }
\end{definition}

The following lemma states that complements can be pulled back via $K_X$-positive 
birational contractions (see~\cite[6.1.(3)]{Bir19}).

\begin{lemma}
\label{lem:complements-and-K-positive}
Let $(X,B,\bM.)$ be a generalized log canonical pair over a base $Z$. Suppose $f \colon X\dashrightarrow X'$ is a $(K_X+B+\bM X.)$-non-negative birational contraction over $Z$. Let $B'=f_*B$ and $N$ be a positive integer. If  $K_{X'}+B'+\bM X'.$ has an $N$-complement over $z\in Z$, then $K_X+B+\bM X.$ also has an $N$-complement over $z\in Z$.
\end{lemma}

The following lemma says that extracting divisors with small log discrepancy from a Fano type variety
preserves the Fano type property (see~\cite[6.13.(7)]{Bir19}).

\begin{lemma}
\label{lem:complements-and-dlt-mod}
Let $X\rightarrow Z$ be a contraction.
Let $X$ be a Fano type variety over $Z$
and $(X,B)$ a log Calabi--Yau pair over $Z$.
Let $f\colon Y\rightarrow X$ be a birational morphism. Suppose that every $f$-exceptional divisor $E$ satisfies $a_E(X,B)<1$. Then $Y$ is of Fano type over $Z$.

Furthermore, let $K_Y+B_Y$ be the log pull-back of $K_X+B$. If $K_Y+B_Y$ has an $N$-complement over $z\in Z$, then $(X,B)$ also has an $N$-complement over $z\in Z$.
\end{lemma}

\subsection{Coefficients under adjunction}
\label{subsec:coreg-under-adjunction}

In this subsection, we study the coefficients
of a pair under adjunction.

\begin{definition}\label{def:index-of-a-set}
{\rm Let $\mathcal{R}$ be a set of rational numbers. We define $I_\mathcal{R}$ to be the minimal integer $I$ such that for any $r\in R$ and $n \in \nn$, we have that 
$$\lfloor  nI r \rfloor \geq n(I-1) r.$$
If there does not exist such an integer, we define $I_\mathcal{R}$ to be $0$.

Note that when the set $\mathcal{R}$ is finite $I_\mathcal{R}$ exists as the least common multiple of the denominators will satisfy the previous inequality. When $\mathcal{R}$ is the set of standard coefficients $I_\mathcal{R}$=1. If $\mathcal{R}$ is finite, then $I_\mathcal{R}$ is bounded above by the least common multiple of the rational numbers in $\mathcal{R}$ that are not standard.}
\end{definition}

\begin{definition}
{\rm 
Let $\Lambda$ be a set of real numbers in $[0,1]$. Define the {\em derived set of $\Lambda$} as
\[
D(\Lambda) \coloneqq  \left\{a\in [0,1] \mid a = \frac{m-1+\lambda_1+\cdots+\lambda_n}{m},
\text{ where } n\in \zz_{\geq 0}, m\in \zz_{>0}
\text{ and }\lambda_1,\ldots,\lambda_n\in \Lambda\cup\{0,1\} \right\}.
\]
We also define $D_{\lambda_0}(\Lambda)\subset D(\Lambda)$ to be the subset in which, in the definition of $a$, at least one $\lambda_i$ is equal to $\lambda_0$.
The set $\Lambda$ is said to be {\em derived} if $\Lambda=D(\Lambda)$.
If $\lambda$ is a positive integer, then we set
$D_\lambda \coloneqq D\left(\zz\left
[\frac{1}{\lambda}\right]\cap [0,1]\right)$.
}
\end{definition}

For instance, the set of standard coefficients $\mathcal{S} \coloneqq \{1-1/m\mid m\in \zz_{>0}\}\cup \{1\}$ is derived.
The following lemmata describe some properties of derived sets.

\begin{lemma}[{\cite[Proposition 3.4.1]{HMX14}}]\label{lem:derived-closure}
Let $\Lambda$ be a set of real numbers in $[0,1]$. Then $D(\Lambda) = D(D(\Lambda))$, i.e., $D(\Lambda)$ is a derived set.
\end{lemma}

The following lemma allows us to control the coefficients of the generalized pairs
obtained by divisorial adjunction.
The lemma is a special case of \cite[Lemma 3.3]{Bir19}; we refer to the proof of \cite[Lemma 3.8]{FMM22} for the details of this adaptation.

\begin{lemma}[{\cite[Lemma 3.3]{Bir19}}]\label{lemma-coeff-adj}
Let $(X,B,\bM.)$ be a generalized log canonical pair over $Z$ and $\Lambda$ be a set of rational numbers in $[0,1]$. Suppose the coefficients of $B$ and $\bM X'.$ belong to $\Lambda$ for some model $X'$ where $\bM.$ descends. Let $S$ be the normalization of a component of $\lfloor B\rfloor$. Write
\[(K_X+B+\bM X.)|_S\sim K_S+B_S+\bN S.\] for the generalized adjunction on $S$, where $B_S$ is the boundary part and $\bN.$ the moduli part. Then the coefficients of $B_S$ and $\bN S'.$ belong to the derived set $D(\Lambda)$ for some model $S'$ where $\bN.$ descends.
\end{lemma}

The following lemma is used in the proof of Theorem \ref{thm:cbf-and-coreg-induction} to control the coefficients of the discriminant part of a log Calabi--Yau fibration over a curve.

\begin{lemma}\label{lem:coeff-dcc-and-rational-acc-points}
Let $q$ be a positive integer. Let $\Lambda$ be a set of nonnegative rational numbers. Suppose $\Lambda$ satisfies the DCC and has rational accumulation points. Then the set
\[
\Sigma_q  \coloneqq \left\{\frac{b^+-b}{m}\geq 0\mid qb^+\in \zz_{>0}, b^+\leq 1 ,m\in \zz_{>0}, b\in \Lambda\right\} \subseteq \qq
\]
satisfies the ascending chain condition and has rational accumulation points.
\end{lemma}

\begin{proof}
We first show that $\Sigma_q$ satisfies the ascending chain condition. Suppose in $\Sigma_q$ we can find an increasing sequence 
\[
\frac{b_1^+-b_1}{m_1} < \frac{b_2^+-b_2}{m_2} < \cdots < \frac{b_k^+-b_k}{m_k} < \cdots
\]
Since $b_k^+ \in \{i/q \colon  0\leq i\leq q\}$ has only finitely many choices, we may assume, by passing to a subsequence, that all $b_k^+$ are the same and equal to the number $b^+$. Furthermore, note that
\[
m_k < \frac{m_1(b_k^+-b_k)}{b_1^+-b_1} \leq \frac{m_1}{b^+-b_1}.
\]
The second inequality holds as $1\geq b_k^+ \geq b_k^+-b_k$.
Hence, the sequence $m_k$ is bounded above. Thus, by passing to a subsequence we may assume that $m_k = m$ for all $k$. Now, we obtain a decreasing sequence
\[
b_1 > b_2 > \cdots > b_k > \cdots,
\]
which violates the descending chain condition of $\Lambda$. Thus, $\Sigma_q$ satisfies the ascending chain condition.

Let $a\neq 0$ be an accumulation point of $\Sigma_q$. Since $\Sigma_q$ satisfies the ascending chain condition, we may find a sequence
\[
\frac{b_1^+-b_1}{m_1} \geq \frac{b_2^+-b_2}{m_2} \geq \cdots \geq \frac{b_k^+-b_k}{m_k} \geq \cdots 
\]
whose limit is $a$. We may assume that $b_k^+ = b^+$ for all $k$. Then
\[m_k \leq \frac{b_k^+-b_k}{a} \leq \frac{1}{a}.
\]
The second inequality holds as $1\geq b_k^+ \geq b_k^+-b_k$.
Hence, the sequence $m_k$ is bounded above. By passing to a subsequence we may assume that $m_k = m$ for all $k$. Since $\Lambda$ has rational accumulation points,
\[
a = \lim_{k\to \infty} \frac{b^+-b_k}{m} = \frac{b^+}{m} - \frac{1}{m} \lim_{k\to\infty} b_k \in \qq,
\]
as desired.
\end{proof}

\subsection{Coregularity of pairs}
In this subsection, we define the coregularity of a generalized pair
and prove some of its properties.

\begin{definition}\label{def:dual-complex}
{\em 
Let $(X,B,\bM.)$ be a generalized log canonical pair.
Let $f \colon  Y \to X$ be a generalized dlt modification and write
\[
K_Y + B_Y + \bM Y. = f^*(K_X+B+\bM X.).
\]
Let \[\lfloor B_Y\rfloor = E_1 + E_2 + \cdots + E_r\] be a simple normal crossing divisor on $Y$. 

The {\it dual complex} $\mathcal{D}(Y,B_Y +\bM Y.)$ is a simplicial complex constructed as follows:
\begin{itemize}
    \item For every $1\leq i\leq r$, there is a vertex $v_i$ in $\mathcal{D}(Y,B_Y+\bM Y.)$ corresponding to the divisor $E_i$. For every subset $I\subseteq \{1,2,\ldots,r\}$ and every irreducible component $Z$ of $\bigcap_{i\in I} E_i$, there is a simplex $v_Z$ of dimension $\# I - 1$ corresponding to $Z$;
    \item For every $I\subseteq \{1,2,\ldots,r\}$ and $j\in I$, there is a gluing map constructed as follows. Let $Z\subseteq \bigcap_{i\in I} E_i$ be any irreducible component. Let $W$ be the unique component of $\bigcap_{i\in I\setminus\{j\}}E_i$ containing $Z$.
    Them, the gluing map is the inclusion of $v_W$ into $v_Z$ as the face of $v_Z$ that does not contain the vertex $v_i$.
\end{itemize}
Define the dimension of $\mathcal{D}(Y,B_Y+\bM Y.)$ to be the smallest dimension
of the maximal simplex, with respect to the inclusion,
of $\mathcal{D}(Y,B_Y+\bM Y.)$. When $\mathcal{D}(Y,B_Y+\bM Y.) = \varnothing$, set its dimension to be $-1$.

The dual complex $\mathcal{D}(Y,B_Y,\bM Y.)$ depends on the dlt modification $Y$. However, its PL-homeomorphism type is independent of the dlt modification (see, e.g.,~\cite[Theorem 1.6]{FS23}).

Define the \textit{dual complex} $\mathcal{D}(X,B,\bM.)$ associated to the generalized pair $(X,B,\bM.)$ as the homeomorphism type of the complex $\mathcal{D}(Y,B_Y+\bM Y.)$. Thus, for any dlt modification we have:

\[ \dim \mathcal{D}(X,B,\bM .)= \dim \mathcal{D}(Y,B_Y,\bM Y.)\]

When $\bM. = 0$, we write $\mathcal{D}(X,B)$ instead of $\mathcal{D}(X,B,0)$ for simplicity.
}
\end{definition}

\begin{definition}
{\rm Let $(X,B,\bM.)$ be a generalized log canonical pair over $Z$.
We define its {\em coregularity} to be
\[
{\rm coreg}(X,B,\bM.)  \coloneqq  \dim X - 1 - \dim \mathcal{D}(X,B,\bM.).
\]}
\end{definition}

\begin{definition}
{\rm Let $(X,B,\bM.)$ be a generalized log canonical pair over $Z$.
We define the {\em absolute coregularity over $Z$} of $(X,B,\bM.)$, denoted by $\hat{\rm coreg}(X/Z,B,\bM.)$, as follows:
\begin{itemize}
    \item if $(X,B^+,\bM.)$ is not a generalized log Calabi--Yau pair over $Z$ for every divisor $B^+\geq B$, we set $\hat{\rm coreg}(X/Z,B,\bM.)$ to be $\infty$;
    \item otherwise, we set $\hat{\rm coreg}(X/Z,B,\bM.)$ to be the smallest value of ${\rm coreg}(X,B^+,\bM.)$, over all divisors $B^+\geq B$ for which $(X,B^+,\bM.)$ is generalized log Calabi--Yau over $Z$.
\end{itemize}
Let $z\in Z$ be a point.
We define the {\em absolute coregularity}
of $(X,B,\bM.)$ over $z\in Z$, 
denoted by $\hat{\rm coreg}_z(X,B,\bM.)$ 
to be the minimum
of $\hat{\rm coreg}(\pi^{-1}(U)/U,B,\bM.)$
where $U$ runs over all neighborhoods of $z\in Z$.

By definition, we have that
\[
\hat{\rm coreg}_z(X,B,\bM.) 
\in \{0,\dots, \dim X, \infty\}.
\]
If $X\rightarrow Z$ is the structure morphism of $X$
and ${\rm coreg}(X,B,\bM.)=\dim X$, 
then we say that $(X,B,\bM.)$ is an
{\em exceptional} generalized pair.
}
\end{definition}

By the negativity lemma, 
the coregularity is preserved
under certain MMP.

\begin{lemma}\label{lem:coreg-nonnegative-contraction}
Let $(X,B,\bM.)$ be a generalized log canonical pair over $Z$.
Let $z\in Z$ be a point.
Suppose $f \colon  X\dashrightarrow Y$ is a $(K_X+B+\bM X.)$-nonnegative birational contraction over $Z$.
Write $B_Y = f_*B$.
Then
\[
\hat{\rm coreg}_z(Y,B_Y,\bM.) = \hat{\rm coreg}_z(X,B,\bM.).
\]
\end{lemma}

\begin{proof}
Up to shrinking $Z$ around $z$, 
we can find a generalized log Calabi--Yau pair
$(X,B+\Gamma,\bM.)$ over $Z$ that
computes the absolute coregularity
of $(X,B,\bM.)$ over $z$. 
Let $\Gamma_Y=f_*\Gamma$.
Since $(X,B+\Gamma,\bM.)$ is generalized log Calabi--Yau over $Z$, we conclude that
$(Y,B_Y+\Gamma_Y,\bM.)$ is generalized log canonical.
As $(X,B+\Gamma,\bM.)$ and $(Y,B_Y+\Gamma_Y,\bM.)$ are crepant equivalent,
we conclude that they have the same coregularity.
Hence, we deduce that 
\[
\hat{\rm coreg}_z(Y,B_Y,\bM.) \leq \hat{\rm coreg}_z(X,B,\bM.).
\]
On the other hand, 
up to shrinking $Z$ around $z$, 
we can find an effective divisor $D_Y$ on $Y$
that computes the absolute
coregularity of $(Y,B_Y,\bM.)$ over $z$. 
Let $p\colon W \rightarrow X$ and $q\colon W\rightarrow Y$ be a common resolution.
Write
\[
K_X+B+D+\bM X. = q_*p^*(K_Y+B_Y+D_Y+\bM Y.).
\]
By the negativity lemma, we know that $D$ is an effective divisor.
Hence, $(X,B+D,\bM.)$ is a generalized log Calabi--Yau pair over $Z$.
As above, we conclude that the absolute coregularity of
$(X,B,\bM.)$ over $z\in Z$ is at most the absolute coregularity of $(Y,B_Y,\bM.)$ over $z\in Z$.
This finishes the proof.
\end{proof}

The following lemma states the coregularity behaves well under adjunction
for generalized log Calabi--Yau pairs. 

\begin{lemma}\label{lem:coreg-adjunction}
Let $(X,B,\bM.)$ be a generalized log Calabi--Yau pair over $Z$.
Let $z\in Z$ be a point.
Let $S$ be the normalization of a component
of $\lfloor B\rfloor$ whose image on $Z$ contains $z$.
Let $B_S$ 
and $\bN.$ be the boundary and moduli parts defined by generalized adjunction, so that
$(K_X+B+\bM X.)|_S\sim K_S+B_S+\bN S..$
Then, we have that 
\[{\rm coreg}(S,B_S, \bN.)={\rm coreg}(X,B, \bM.).\]
holds after possibly shrinking around $z\in Z$.
\end{lemma}

\begin{proof}
By passing to a generalized dlt modification, we may assume that both generalized
pairs $(X,B,\bM.)$ and $(S,B_S,\bN.)$
are generalized dlt.
Since any minimal generalized log canonical center of $(S,B_S,\bN.)$ is a minimal generalized log canonical center of $(X,B,\bM.)$, we have that ${\rm coreg}(S,B_S,\bN.)\geq {\rm coreg}(X,B,\bM.)$.
On the other hand, let $W$ be a minimal generalized log canonical center
of $(X,B,\bM.)$ whose image on $Z$ contains $z$. 
By~\cite[Theorem 1.4]{FS23},  $(S,B_S,\bN.)$ admits
a generalized log canonical center $W_S$ that is $\mathbb{P}^1$-linked to $W$, thus  $\dim W_S =\dim W$.
This implies that ${\rm coreg}(X,B,\bM.)\geq {\rm coreg}(S,B_S,\bN.)$.
\end{proof}

See~\cite[\S~3]{Bir19} for the construction of generalized adjunction. By altering the pairs, we can get the same result for pairs with nef anticanonical class.

\begin{lemma}\label{lem:fano-coreg-adjunction}
Let $(X,B,\bM.)$ be a generalized log canonical pair over $Z$. 
Assume that $-(K_X+B+\bM X.)$ nef over $Z$.
Let $S$ be the normalization of a component
of $\lfloor B\rfloor$.
Let $B_S$ and $\bN.$ be the boundary and moduli parts defined by generalized adjunction, so that
$(K_X+B+\bM X.)|_S\sim K_S+B_S+\bN S.$.
Then, we have that 
\[{\rm coreg}(S,B_S, \bN.)={\rm coreg}(X,B, \bM.).\]
\end{lemma}

\begin{proof}
Define $P \coloneqq -(K_X+B+\bM X.)$ and let $\mathbf{P}$ denote its b-Cartier closure.
Then, $\mathbf{P}$ is a b-nef $\qq$-Cartier divisor. We can apply Lemma~\ref{lem:coreg-adjunction} to the
generalized log Calabi--Yau pair $(X,B,\bM .+ \mathbf{P})$. 
Therefore ${\rm coreg} (S,B_S, \bN .+ \mathbf{P}|_{S} )={\rm coreg} (X,B, \bM . + \mathbf{P})$.
Since $\mathbf{P}$ descends on $S$, we conclude that 
${\rm coreg}(S,B_S,\bN.+\mathbf{P}|_{S})={\rm coreg}(S,B_S,\bN.)$.
Hence ${\rm coreg}(S,B_S, \bN.)={\rm coreg}(X,B, \bM.)$.
\end{proof}

The following lemma will be used to cut down the dimension of the base $Z$ in a fibration $X\to Z$.
\begin{lemma}\label{lem:hyperplane-section-inductive-step}
Let $(X,B)$ be a log canonical pair over $Z$ and $\pi \colon X\to Z$ be a fibration with $\dim Z\geq 2$. Suppose 
\begin{itemize}
    \item the pair $(X,B)$ is log Calabi--Yau over $Z$;
    \item $\phi$ is of Fano type over an open set $U$ of $Z$;
    \item every log canonical center of $(X,B)$ dominates $Z$; and
    \item the coregularity of $(X,B)$ is at most $c$.
\end{itemize}
Let $H$ be a general hyperplane section of $Z$ and $G$ be the pull-back of $H$ to $X$. Write 
\[
(K_X+B+G)|_G = K_G+B_G.
\]
Then we have
\begin{itemize}
    \item the pair $(G,B_G)$ is log canonical;
    \item the pair $(G,B_G)$ is log Calabi--Yau over $H$;
    \item the induced map $G\to H$ is of Fano type over $U\cap H$;
    \item every log canonical center of $(G,B_G)$ dominates $H$; and
    \item the coregularity of $(G,B_G)$ is at most $c$.
\end{itemize}
Furthermore, let $B_Z$ and $B_H$ denote the discriminant parts of the adjunction for $(X,B)$ over $Z$ and $(G,B_G)$ over $H$, respectively. Let $D$ be a prime divisor on $Z$ and $C$ a component of $D\cap H$. Then
\[
\coeff_D(B_Z) = \coeff_C(B_H).
\]
\end{lemma}
\begin{proof}
We follow the proof of \cite[Lemma 3.2]{Bir16}.

Since $G$ is the pull-back of a general hyperplane section on $Z$, $(X,B+G)$ is log canonical. Thus, by adjunction, $(G,B_G)$ is log canonical and log Calabi--Yau over $H$. Moreover, every log canonical center of $(G,B_G)$ is a component of the intersection of a log canonical center of $(X,B+G)$ and $G$, and hence must dominate $H$. By \ref{lem:coreg-adjunction}, we have an equality
\[
\text{coreg}(G,B_G) = \text{coreg}(X,B) \leq c.
\]
Denote the map $G\to H$ by $\psi$. Let $t$ be the log canonical threshold of $\pi^*D$ with respect to $(X,B)$ over the generic point of $D$. Then there is a non-klt center $W$ of $(X,B+t\pi^*D)$ which dominates $D$ and the pair $(X,B+t\pi^*D)$ is lc over the generic point of $D$. Since $G$ is a general pull-back, $(X,B+G+t\pi^*D)$ is also lc over the generic point of $D$.
By inversion of adjunction \cite{Ka07}, there exists a component of $G\cap W$ which is a non-klt center of $(G,B_G+t\pi^*C)$ and $(G,B_G+t\pi^*C)$ is lc near the generic point of $C$. Thus, $t$ is the log canonical threshold of $\psi^*C$ with respect to $(G,B_G)$. In particular, we have
\[
\coeff_D(B_Z) = 1-t = \coeff_C(B_H).
\]
\end{proof}

\begin{definition}\label{def:coreg-lct}
{\em
We say that a log canonical threshold
$t={\rm lct}((X,B);\Gamma)$ has coregularity $c$
if ${\rm coreg}(X,B+t\Gamma)=c$ and
the support of $\Gamma$ contains the image on $X$ of a $c$-dimensional log canonical center on a dlt modification.
}
\end{definition}

\subsection{Koll\'ar--Xu models for log Calabi--Yau pairs}
\label{subsec:kollar-xu-models}

In this subsection, we introduce the concept of Koll\'ar--Xu models.
Using a theorem due to Filipazzi and Svaldi, we conclude that 
every generalized log Calabi--Yau pair
admits a Koll\'ar--Xu model.

\begin{definition}\label{def:kx-model}
{\em 
Let $(X,B,\bM.)$ be a projective generalized log Calabi--Yau pair.
We say that $(X,B,\bM.)$ 
is a {\em Koll\'ar--Xu generalized pair}
if there exists a projective contraction
$\pi\colon X \rightarrow Z$
for which the following conditions are satisfied:
\begin{itemize} 
    \item[(1)] the generalized pair
    $(X,B,\bM.)$ is generalized dlt;
    \item[(2)] every generalized log canonical center of $(X,B,\bM.)$ dominates $Z$; and
    \item[(3)] the divisor $\lfloor B\rfloor$
    fully supports a $\pi$-big and $\pi$-semi-ample divisor.
\end{itemize} 
In particular, the morphism 
$\pi\colon Y\rightarrow Z$ 
is of Fano type.

Let $(X,B,\bM.)$ be a generalized log Calabi--Yau pair. 
Let $\pi\colon Y\dashrightarrow X$ be a birational map.
Assume that $\pi$ only extracts log canonical places of $(X,B,\bM.)$
and is an isomorphism over
$X\setminus \Supp \lfloor B\rfloor$.
If $(Y,B_Y,\bM.)$ is a Koll\'ar--Xu generalized pair, then
we say that $Y\dashrightarrow X$
is a {\em Koll\'ar--Xu model}
for $(X,B,\bM.)$.
We may also say that $(Y,B_Y,\bM.)$,
together with $\pi$, defines a Koll\'ar--Xu model for $(X,B,\bM.)$.
}
\end{definition}

The following theorem is a generalization of \cite[Theorem 49]{KX16}
to the context of generalized pairs.
We refer the reader to~\cite[Theorem 4.2]{FS23}.
It gives a first approximation
for the existence of Koll\'ar--Xu models in the following theorem.

\begin{theorem}
\label{thm_fs23}
Let $(X,B,\mathbf{M})$ be a projective $\qq$-factorial generalized dlt log Calabi--Yau pair.
Then, there exists a crepant birational map $\phi \colon X \drar \overline{X}$, a generalized pair $(\overline{X}, \overline{B}, \mathbf M)$, and a morphism $\pi \colon \overline{X} \rar Z$ such that:
\begin{itemize}
    \item[(1)] $\lfloor \overline{B} \rfloor$ fully supports a $\pi$-ample divisor;
    \item[(2)] every generalized log canonical center of $(\overline{X},\overline{B},\mathbf{M})$ dominates $Z$;
    \item[(3)] $\overline{E} \subset \Supp\lfloor\overline{B}\rfloor$ for every $\phi \sups -1.$-exceptional divisor $\overline{E} \subset \overline{X}$; and
    \item[(4)] $\phi \sups -1.$ is an isomorphism over $\overline{X} \setminus \Supp \lfloor \overline{B}\rfloor$.
\end{itemize}
\end{theorem}

We observe that the model $\overline{X}$ in Theorem~\ref{thm_fs23} is not necessarily $\qq$-factorial.
However, using $\qq$-factorial dlt modifications, we construct a Koll\'ar--Xu model.

\begin{theorem}\label{thm:fully-support-big}
Let $(X,B,\bM.)$ be a projective generalized log Calabi--Yau pair.
Then, it admits a Koll\'ar--Xu model
$Y\dashrightarrow X$.
Furthermore, if $(X,B,\bM.)$ has coregularity $c$, then so does
$(Y,B_Y,\bM.)$.
\end{theorem}

In the context of Theorem \ref{thm:fully-support-big}, if $c=0$, then $Z=\Spec (\mathbb{K})$ and $Y$ is a variety of Fano type.

\begin{proof}
Let $(X,B,\bM.)$ be a generalized pair as in the statement.
Then, we first consider a $\qq$-factorial generalized dlt modification and then apply Theorem~\ref{thm_fs23} to such modification.
Call $(\overline{X},\overline{B},\bM.)$ the resulting model and $\pi \colon \overline{X} \rar Z$ the morphism claimed in Theorem~\ref{thm_fs23}.
Since $(X,B,\bM.)$ has coregularity $c$, then so does $(\overline{X},\overline{B},\bM.)$.
In particular, $(\overline{X},\overline{B},\bM.)$ has a $c$-dimensional generalized log canonical center.
Then, by (2) in Theorem~\ref{thm_fs23}, it follows that $Z$ has dimension at most $c$.
By (1) in Theorem~\ref{thm_fs23}, $\overline{B}$ fully supports a $\pi$-ample divisor, which we will denote by $\overline{H}$.
Now, let $(Y,B_Y,\bM.)$ be a $\qq$-factorial generalized dlt modification of $(\overline{X},\overline{B},\bM.)$, with morphism $p \colon Y \rar \overline{X}$.
We denote by $\pi_Y$ the induced morphism $\pi_Y \colon Y \rar Z$.
Note that every generalized log canonical center of $(Y,B_Y,\bM.)$ dominates $Z$.
Then, (2) in Definition~\ref{def:kx-model} holds.
By \cite[Remark 4.3]{FS23}, we have $\mathrm{Nklt}(\overline{X},\overline{B},\bM.)=\Supp\lfloor \overline{B}\rfloor$.
Since $p$ only extracts generalized log canonical places, it then follows that every $p$-exceptional divisor has positive coefficients in $H_Y \coloneqq \pi^*\overline{H}$ and that $\Supp H_Y= \Supp \lfloor B_Y \rfloor$.
Then, it follows that $\lfloor B_Y \rfloor$ fully supports a $\pi_Y$-big and $\pi_Y$-semi-ample divisor.
Thus, (3) in Definition~\ref{def:kx-model} holds.
The statements (1) and (3) in Definition~\ref{def:kx-model} hold by construction.
We conclude that $(Y,B_Y,\bM.)$, together with $\pi_Y$, are a Koll\'ar--Xu model of $(X,B,\bM.)$.
Lastly, $(Y,B_Y,\bM.)$ has coregularity $c$, since it is crepant to a generalized pair of coregularity $c$.
\end{proof}

\subsection{Index of generalized klt pairs} 

In this subsection, we reduce the index conjecture for generalized klt pairs to the standard index conjecture. 

\begin{lemma}\label{lem:index-gen-klt}
Let $d$ and $p$ be two positive integers.
Let $\Lambda$ be a set of rational numbers satisfying the descending chain condition.
Assume Conjecture~\ref{conj:index}$(d)$ holds.
Then, there exists a positive integer
$I \coloneqq I(\Lambda,d,p)$,
satisfying the following.
Let $(X,B,\bM.)$ be a generalized klt 
Calabi--Yau pair for which:
\begin{itemize}
    \item the variety $X$ has dimension $d$;
    \item the coefficients of $B$ belong to $\Lambda$; 
    \item the divisor $p\bM.$ is Cartier where it descends. 
\end{itemize}
Then, we have that $I(K_X+B+\bM X.)\sim 0$.
\end{lemma}

\begin{proof}
By the global ACC \cite[Theorem 1.6]{BZ16}, we may assume that $\Lambda$ is a finite set of rational numbers.
The statement is clear in dimension $1$.
We proceed by induction on the dimension.
If $\bM.=0$, then the statement follows from the conjecture.
Since $(X,B,\bM.)$ is generalized klt, it admits a small $\qq$-factorialization.
Therefore, we may assume that $X$ is $\qq$-factorial.
By the ACC for generalized log canonical thresholds, we may assume $(X,B,\bM.)$ is $\epsilon$-log canonical for some $\epsilon$ that only depends on $d$, $p$, and $\Lambda$. 
Then, it follows that $X$ is itself $\epsilon$-log canonical.
We run a $K_X$-MMP which terminates with a Mori fiber space $X\dashrightarrow X'\rightarrow Z$.
If $Z$ is zero-dimensional, then $X'$ belongs to a bounded family by~\cite{Bir21}.
By~\cite[Theorem 1.2]{FM20} $(K_{X'}+B'+\bM X'.)$ admits an $I$-complement for some $I$ that only depends on $\Lambda,d$ and $p$. Since $K_{X'}+B'+\bM X'.\sim_\mathbb{Q} 0$, we conclude that $I(K_{X'}+B'+\bM X'.)\sim 0$, so the statement follows for $X$ as well.
Now, assume that $Z$ is positive-dimensional.
We write $\pi'\colon X'\rightarrow Z$ for the corresponding contraction.
By~\cite[Lemma 5.4]{FM20}, we can write
\[
q(K_{X'}+B'+\bM X'.) \sim q
\pi^*(K_Z+B_Z+\bN Z.),
\] 
where the coefficients of $B_Z$ belong to $\Omega$ which satisfies the DCC
and only depends on $\Lambda,d$ and $p$. 
Furthermore, $q$ only depends on $\Lambda,d,$ and $p$
and $q\bN.$ is Cartier where it descends.
The generalized pair $(Z,B_Z, \bN .)$ is generalized log canonical, since it comes from the generalized canonical bundle formula \cite[Theorem 1.4]{Fil20}.
By induction on the dimension, we know that
$I_0(K_Z+B_Z+\bN Z.)\sim 0$ for some $I_0$ that only depends on $\Lambda,d$ and $p$.
Thus, we conclude that $I(K_X+B+\bM X.)\sim 0$ where $I={\rm lcm}(q,I_0)$. 
\end{proof}

\subsection{Lifting sections using fibrations}
In this subsection, we give some lemmata regarding the lifting of sections using fibrations.

\begin{theorem}\label{thm:lifting-fibration-gen-pairs}
Let $\lambda,d$, and $c$ be nonnegative integers.
Assume that Theorem~\ref{introthm:index-higher-coreg}$(d-1,c)$ holds
and let $I \coloneqq I(D_\lambda,d-1,c,\lambda)$ be the integer provided by this theorem.
Let $(X,B,\bM.)$ be a $d$-dimensional rationally connected generalized log Calabi--Yau pair.
Assume that the following conditions hold:
\begin{itemize}
    \item $X$ is $\qq$-factorial and $(X,B,\bM.)$ is generalized dlt;
    \item there is a fibration $X\rightarrow W$, which is a $(K_X + \lfloor B \rfloor)$-Mori fiber space;
    \item a component $S\subset \Supp\lfloor B\rfloor$ is rationally connected and ample over $W$;
    \item the coefficients of $B$ belong to $D_\lambda$;
    \item we have that $\lambda\bM.$ is b-Cartier; and
    \item the generalized pair $(X,B,\bM.)$ has coregularity $c$.
\end{itemize}
Then, we have that $I(K_X+B+\bM X.)\sim 0$.
\end{theorem}

\begin{proof}
Let $(X,B,\bM.)$, $S$, $f \colon X \rar W$, and $I \coloneqq I(D_\lambda,d-1,c,\lambda)$ be as in the statement.
First, we show that we can apply the inductive hypothesis to $S$.

Since $(X,\lfloor B \rfloor)$ is dlt, $S$ is normal.
Furthermore, we have $\mathrm{Nklt}(X,S)=S$.
Since $-(\K X. + S)$ is $f$-ample, by the connectedness principle \cite{FS23}, it follows that $S \rar W$ has connected fibers.
Now, let $(S,B_S,\bN.)$ be the generalized pair induced on $S$ by generalized divisorial adjunction.
By Lemma~\ref{lem:coreg-adjunction}, $(S,B_S,\bN.)$ has coregularity $c$.
Furthermore, by Lemma~\ref{lemma-coeff-adj} and Lemma~\ref{lem:derived-closure}, it satisfies the assumptions of Theorem~\ref{introthm:index-higher-coreg}$(d-1,c)$ with constant $I$.
Thus, we have
\begin{equation} \label{eq_dim_one_less}
I (\K S. + B_S + \bN S.) \sim 0.
\end{equation}

By \cite[Theorem 3.1]{FMP22}, the coefficients of $B_S$ belong to a finite set only depending on $\lambda$ and $c$.
In particular, they are divisible by $I$, as so are the coefficients of $\bN S.$.
Then, these coefficients control the coefficients of $\mathrm{Diff}_S(0)$, as we explain in what follows.
By \cite[3.35]{K13}, along the codimension 2 points of $X$ contained in $S$, $X$ has cyclic singularities.

Then, given a prime divisor $P$ in $S$, an \'{e}tale local neighbourhood of a general point $p \in P$ is isomorphic to
\[
(p \in (X, B, \bM.))\simeq (0 \in (\Af^2=(x,y), (x=0) + c(y=0)))/(\zz/m\zz)\times \Af^{\dim X -2},
\]
where $Z \simeq (\Af^2=(x,y))/(\zz/m\zz)$, $S=(x=0)$ and $S'=(y=0)$. Since the class group of $Z$ is generated by $S'$, there exists an integer $\mu$ such that
\begin{equation}\label{eq:linear-equiv-toric}
I(K_X+B+\bM X.) \sim \mu S'.
\end{equation}
By adjunction, $S'|_S \sim_{\mathbb{Q}} \frac{1}{m}\{0\}$.
We also have that $I(K_X+B+\bM X.)|_S\sim 0$, as the denominators of the coefficients of $B_S$ and $\bN S.$ divide $I$ and hence it is a Cartier divisor on a smooth germ.

Then, we can write 
\[
0 \sim I(K_X+B+M)|_S \sim 
\mu S'|_S \sim \frac{\mu}{m}\{0\}. 
\] 
We conclude that $m$ divides $\mu$.
In particular, we have that $\mu S$ is a Cartier divisor, as the Cartier index of any Weil divisor of $X$ through $\{0\}$ divides $m$.
By the linear equivalence~\eqref{eq:linear-equiv-toric}, we conclude that the divisor
$I(K_X+B+M)$ is Cartier at the generic point of $P$. 
Note that this argument is independent of $P$, so
we conclude that $I(K_X+B+M)$ is Cartier at the generic point of every divisor on $S$. 
Thus, $I(K_X+B+M)$ is Cartier along a set $U$ that contains the generic point of every divisor of $S$.

Then, by \cite[2.41 and Lemma 2.42]{Bir19}, we have the following short exact sequence
\begin{equation}\label{ses_fiber_space2}
0 \rar \O X.(I(K_{X}+B + \bM X.)-S) \rar \O X.(I(K_{X}+B+ \bM X.)) \rar \O S.(I(K_{S}+B_{S}+ \bN S.)) \rar 0.
\end{equation}
Since $I(K_{X}+B + \bM X.)-S\sim_{\qq,f} -S$, the divisor $-S$ is $f$-ample, and $\dim W < \dim X$,
we have
\[
f_*\O X.(I(K_{X}+B + \bM X.)-S)=0.
\]
Similarly, we write 
\[
I(K_{X}+B + \bM X.)-S\sim_{\qq,f} -S \sim_{\qq,f} K_X +(B-S+\bM X.).
\]
Note that $X$ is klt and $B-S+\bM X. = B^{<1}+\bM X. + (B^{=1} - S)$ is $f$-ample, since $f$ is a Mori fiber space and the divisor $B^{<1}+\bM X.$ is $f$-ample.
Thus, by the relative version of Kawamata--Viehweg vanishing, we have
\[
R^1f_*\O X.(I(K_{X}+B + \bM X.)-S)=0.
\]
Therefore, by pushing forward \eqref{ses_fiber_space2} via $f$, we obtain 
\[
f_* \O X.(I(K_{X}+B+ \bM X.)) \simeq f_* \O S.(I(K_{S}+B_{S}+ \bN S.)).
\]
Now, taking global sections, we have
\begin{equation} \label{eq:sections}
H^0(X, \O X.(I(K_{X}+B+ \bM X.))) = H^0(S, \O S.(I(K_{S}+B_{S}+ \bN S.)))=H^0(S, \O S.) \neq 0.
\end{equation}
By Lemma \cite[Lemma 3.1]{FMM22}, \eqref{eq:sections} implies that $I(K_{X}+B+ \bM X.) \sim 0$.
\end{proof}

\begin{theorem}\label{thm:lifting-fibration-pairs-2}
Let $\lambda,d$, and $c$ be nonnegative integers.
Assume that Theorem~\ref{introthm:index-higher-coreg}$(d-1,c)$ holds
and let $I \coloneqq I(D_\lambda,d-1,c,0)$ be the integer provided by this theorem.
Assume that $I$ is divisible by $2\lambda$.
Let $(X,B)$ be a $d$-dimensional log Calabi--Yau pair.
Assume that the following conditions hold:
\begin{itemize}
    \item $X$ is $\qq$-factorial and klt;
    \item there is a fibration $X\rightarrow W$;
    \item $S$ is a prime component of $\lfloor B \rfloor$ that is ample over $W$;
    \item $(X,B-S)$ is dlt;
    \item the morphism $S \rar W$ does not have connected fibers;
    \item the coefficients of $B$ belong to $D_\lambda$; and 
    \item the pair $(X,B)$ has coregularity $c$.
\end{itemize}
Then, we have that $I(K_X+B)\sim 0$.    
\end{theorem}

\begin{proof}
Let $(X,B)$, $S$, $f \colon X \rar W$, and $I \coloneqq I(D_\lambda,d-1,c,\lambda)$ be as in the statement.
We will proceed in several steps.\\

\noindent{\it Step 1:} In this step, we observe that $\dim X - \dim W=1$, $f$ is generically a $\pr 1.$-fibration, and $B_{\rm hor}=S$.\\

Since $(X,B)$ is log canonical and $X$ is $\qq$-factorial, then $(X,S)$ is log canonical.
By considering a general fiber of $f$, the restriction of $S$ induces a disconnected ample divisor.
Therefore, by \cite[Exercise III.11.3]{Har77}, it follows that the general fiber of $f$ is a curve.
By the log Calabi--Yau condition and the fact that $0 \neq B_{\rm hor} \geq S$, it follows that $f$ is generically a $\pr 1.$-fibration.
Since $S \rar W$ does not have connected fibers and $\deg K_{\pr 1.}=-2$, it follows that $S$ is the only component of $\Supp B$ that dominates $W$, i.e., we have $B_{\rm hor} =S$.
In particular, we may find a non-empty open subset $U \subseteq W$ such that $K_X+S \sim_\qq 0 /U$.\\

\noindent{\it Step 2:} In this step, we show that $\dim W > 0$.\\

By Step 1, $X \rar W$ is generically a $\pr 1.$-fibration, and $S$ determines two distinct points on the geometric generic fiber of $f$.
Thus, if $\dim W =0$, it would follow that $(X,B) \simeq (\pr 1.,\{ 0 \} + \{\infty \})$, with $S$ identified with $\{ 0 \} + \{\infty \}$.
Since we assumed that $S$ is a prime divisor, this leads to the sought contradiction.\\

\noindent{\it Step 3:} In this step, we show that we may replace $X$ birationally so that $(X,S)$ is plt, $K_X +S \sim_\qq 0/W$, and $f$ is a Mori fiber space.\\

Since $(X,B)$ is log canonical and $X$ is $\qq$-factorial, then $(X,S)$ is log canonical.
Let $X'$ be a $\qq$-factorial dlt modification of $(X,S)$, and let $S'$ denote the strict transform of $S$.
In particular, $(X',S')$ is plt.
By \cite[Theorem 1.1]{HX13}, $(X',S')$ has a relatively good minimal model over $W$, which we denote by $(\tilde X, \tilde S)$.
Since $S'$ dominates $W$ and $\dim X - \dim W =1$, $S'$ is relatively big over $W$.
Therefore, $S'$ cannot be contracted on $\tilde X$, and therefore $\tilde S$ is a divisor.
Since $(X',S')$ is plt, then so is $(\tilde X,\tilde S)$.
We denote by $\tilde W$ the relatively ample model, which is birational to $W$.
Now, we run a $\K \tilde X.$-MMP with scaling over $\tilde W$, which terminates with a Mori fiber space $\hat f \colon \hat X \rar \hat W$ over $\tilde W$.
Since $\dim X -\dim W =1$, it follows that $\hat W \rar \tilde W$ is birational.
As before, since $\tilde S$ dominates $\tilde W$ and is relatively big over $\tilde W$ as $\dim X - \dim W =1$, it follows that $\tilde S$ cannot be contracted on $\hat X$.
Let $\hat S$ be its strict transform on $\hat X$.
Lastly, as $\K \tilde X. + \tilde S \sim_\qq 0 /\tilde W$ and $\tilde S$ is the only log canonical place of $(\tilde X,\tilde S)$, $\hat S$ is the only log canonical place of $(\hat X,\hat S)$.
But then, since $\hat S$ is a divisor on $\hat X$, it follows that $(\hat X,\hat S)$ is plt.
In particular, $\hat S$ is normal.
Thus, by \cite[Corollary 3.3]{FMM22}, up to replacing $X$, $S$, and $Z$ with $\hat X$, $\hat S$, and $\hat Z$, respectively, in the following of the proof, we may further assume that $(X,S)$ is plt, $K_X +S \sim_\qq 0/W$, and $f$ is a Mori fiber space.\\

\noindent{\it Step 4:} In this step, we introduce a suitable pair structure on the base $W$.\\

Let $(S,B_S)$ denote the pair induced by adjunction from $(X,B)$ to $S$.
By Lemma~\ref{lem:coreg-adjunction}, $(S,B_S)$ has coregularity $c$.
Then, by Lemma \ref{lemma-coeff-adj}, the inductive hypothesis applies to $(S,B_{S})$ for the same value of $I$.
We also let $(S,\mathrm{Diff}_S(0))$ be the pair structure induced from $(X,S)$ to $S$.
By Step 1, $f_S \colon S \rar W$ is generically 2:1 and hence Galois.
By~\cite[Proposition 4.37.(3)]{K13}, $(S,\mathrm{Diff}_S(0))$ is invariant under the rational Galois involution.
Then, since $\K X. + S \sim_\qq 0 /W$ and $f$ is a Mori fiber space, it follows that $B-S$ is the pull-back of a $\qq$-divisor on $W$.
Then, it follows that also $(S,B_S)$ is Galois invariant.
Then, by considering the Stein factorization of $S \rar W$ and descending the pair structure thanks to the fact that $\K S. +B_S \sim_\qq 0$, it follows that we can induce a pair structure $(W,B_W)$ such that $f_S^*(\K W. + B_W)=\K S. + B_S$.
Since $(S,B_S)$ has coregularity $c$, by \cite[Proposition 3.11]{FMM22}, then so does $(W,B_W)$.
Furthermore, since $(S,B_S)$ has coefficients in $D_\lambda$, then so does $(W,B_W)$.
Indeed, at the codimension 1 points of $W$ where $S \rar W$ is \'etale, we will have the same coefficients on $B_W$ and $B_S$ for the corresponding divisors.
Then, we can consider a prime divisor $Q \subset W$ such that $S \rar W$ ramifies of order 2 at the generic point of $Q$.
Then, over the generic point of $Q$, by the Riemann--Hurwitz formula, we have
\[
\K S.+ cP=f_S^*\left(K_W+\frac{1}{2}Q+\frac{c}{2}Q \right),
\]
where $P$ is the unique prime divisor in $S$ dominating $Q$, in other words $\coeff_P(B_S)= c \in D_\lambda$ and $\coeff_Q(B_W)=\frac{1+c}{2}$.
By the definition of $D_\lambda$, we must have $c=\frac{m-1+a\lambda^{-1}}{m}$ for some $m \in \zz_{>0}$ and $a \in \zz_{\geq 0}$, and it follows that
\[
\frac{1}{2}+\frac{m-1+a\lambda^{-1}}{m}=\frac{2m-1+a\lambda^{-1}}{2m} \in D_\lambda.
\]
Thus, by the inductive hypothesis, we have 
\begin{equation}\label{eq_base_p1_link}
I(K_W+B_W)\sim 0.
\end{equation}

\noindent{\it Step 5:} In this step, we introduce a suitable generalized pair structure on $W$, and we compare it with $(W,B_W)$.\\

By the canonical bundle formula, the lc-trivial fibration $f \colon (X,B) \rar W$ induces a generalized pair structure $(W,\Delta_W,\bN.)$ on $W$.
By \cite[\S~7.5, (7.5.5)]{PS09} and the fact that the generic fiber of $f \colon (X,B) \rar W$ is a conic with two points, we have
\begin{equation} \label{eq_cbf}
\K X. + B \sim f^*(\K W. + \Delta_W+ \bN W.).
\end{equation}
Furthermore, the representatives of the b-divisor $\bN.$ are determined up to $\zz$-linear equivalence.

By~\cite[proof of Theorem 1.1]{FG14b}, in an lc-trivial fibration $(Y,\Gamma) \rar C$, the total space of the fibration and a minimal (with respect to inclusion) log canonical center $\Xi$ dominating the base $C$ induce the same generalized pair on the base space.
More precisely, they induce the same boundary divisor as $\qq$-Weil divisor and the same moduli divisor as $\qq$-b-divisor class.
Yet, this comparison is possible after the base change induced by the Stein factorization of $\Xi \rar C$, as $\Xi \rar C$ may not have connected fibers.
We observe that the identification between the boundary divisors can also be obtained by inversion of adjunction together with the connectedness principle.

In our situation, this implies that $(X,B)$ and $(S,B_S)$ induce the same generalized pair on $W$, up to pull-back to the Stein factorization on $S \rar W$.
In this case, $S$ is generically a 2:1 cover of $W$.
Thus, it follows that $(W,\Delta_W,\bN.)$ and $(W,B_W)$ agree once pulled back to $S$.
By construction, we have $f_S'^*(\K W. + B_W)=\K S. + B_S$, and the moduli b-divisor is trivial.
In turn, this implies that $\Delta_W=B_W$ and $f_S^* \bN W. \sim 0$.
As for the moduli b-divisor, we only claim $\zz$-linear equivalence, as a representative of the b-divisorial class can be replaced in the $\zz$-linear equivalence class.

Let $S'$ denote the Stein factorization of $S \to W$, with induced morphism $f_{S'} \colon S' \rar W$.
Then, $f_{S'}$ is a finite Galois morphism of degree 2, $f_{S'}^* \bN W. \sim 0$.
By construction, $f_{S'}^* \bN W.$ is Galois invariant, since it is the $\qq$-Cartier pull-back of a $\qq$-divisor on $W$ via the finite morphism $f_{S'}$.
We observe that this implies that $2 \bN W.$ is a $\zz$-divisor, and that $2f_{S'}^* \bN W.$ is the integral pull-back of the integral divisor $2 \bN W.$.

Now, let $s$ be a trivializing section of $f_{S'}^* \bN W.$, and let $\tau$ be the non-trivial element in the Galois group of $S' \rar W$.
By the invariance of $f_{S'}^* \bN W.$, we have that $f_{S'}^* \bN W. + \tau^*f_{S'}^* \bN W.=2f_{S'}^* \bN W.$.
But then, $s \otimes \tau^*s$ is a Galois invariant trivializing section of $2f_{S'}^* \bN W.$.
Then, this section descends to $W$, thus implying that
\begin{equation}\label{2N_eq_0}
    2\bN W. \sim 0.
\end{equation}

\noindent{\it Step 6:} In this step, we conclude the proof.

Combining the previous steps and using the fact that $2|I$, we have
\begin{equation*}
\begin{split}
I(\K X. + B) &\sim If^*(\K W. + B_W + \bN W.)\\
&\sim If^*(\K W.+B_W) + If^* \bN W.\\
&\sim f^*(I(\K W. + B_W)) + f^* I\bN W.\\
&\sim f^*0 + f^*0 \sim 0,
\end{split}
\end{equation*}
where the first linear equivalence follows from \eqref{eq_cbf}, the second one follows from the fact that $K_W+B_W$ is $\qq$-Cartier, the third one follows from the definition of pull-back of $\qq$-divisors, while the last line follows from \eqref{eq_base_p1_link} and \eqref{2N_eq_0}.
This concludes the proof.
\end{proof}

\section{Finite coefficients}
\label{sec:reduct-finite-coeff}

In this section, we explain how to reduce the problem of boundedness of complements from pairs with DCC coefficients
to pairs with finite coefficients. 
First, we prove two lemmata that will be used
in the proof of the main proposition of this section.

\begin{lemma}\label{lem:exchange-fiber}
Let $\phi\colon X\rightarrow Z$ be a contraction from a projective $\qq$-factorial variety $X$.
Let $(X,B,\bM.)$ be a generalized dlt pair over $Z$.
Assume $(X,B,\bM.)$ is generalized log Calabi--Yau over $Z$.
Assume there is a component $S\subseteq \lfloor B\rfloor$ that is vertical over $Z$.
Then, there exists a birational contraction
$X\dashrightarrow X'$ over $Z$
\begin{align*}
\xymatrix{
(X,B,\bM.)\ar@{-->}[rr] \ar[dr]_{\phi}& & (X',B',\bM.) \ar[dl]^{\phi'} \\
& Z &
}
\end{align*} 
satisfying the following conditions:
\begin{itemize}
    \item[(i)] the generalized pair
    $(X',B',\bM.)$ is generalized log canonical;
    \item[(ii)] the strict transform $S'$ of $S$ in $X'$ is a divisorial generalized log canonical center of $(X',B',\bM.)$;
    \item[(iii)] we have that  ${\phi'}^{-1}(\phi(S))=S'$ holds set-theoretically; and 
    \item[(iv)] the generalized pair obtained by adjunction of $(X',B',\bM.)$ to $S'$ is generalized semi-log canonical.
\end{itemize}
\end{lemma}

\begin{proof}
By~\cite[Lemma 3.5]{MS21}, we may run an MMP
for $(X,B-S,\bM.)$ over $Z$
with scaling of an ample divisor $A$.
By the negativity lemma, the divisor $S$ is not contracted by this MMP.
Furthermore, this MMP is $(K_X+B+\bM X.)$-trivial.
Hence, conditions (i) and (ii) hold for any model in this minimal model program.

We argue that after finitely many steps condition (iii) holds.
Let $X_i\dashrightarrow X_{i+1}$ be the $i$-th step of this MMP
and $\phi_i\colon X_i\rightarrow Z$ be the induced projective morphism.
We let $\lambda_i$ the positive real number for which the birational map 
$X_i\dashrightarrow X_{i+1}$ is
$(K_{X_i}+B_i-S_i+\bM X_i.+\lambda_i A_i)$-trivial. Let $\lambda_\infty=\lim_i \lambda_i$. 
If $\lambda_\infty >0$, then the previous MMP is also an MMP for $(X,B-S+\lambda_\infty A,\bM.)$. 
By~\cite[Lemma 3.7]{MS21}, this is also an MMP for a klt pair with big boundary over $Z$ which must terminate by~\cite{BCHM10}.
Let $X'$ be the model where this MMP terminates.
In $X'$, we have that $-S'$ is nef over $Z$. So $S'$ must be the set-theoretic pre-image of $\phi'(S')$.

From now on, we assume that $\lambda_\infty=0$.
Let $W_1,\dots,W_k$ be the irreducible components
of $\phi^{-1}(\phi(S))$ different than $S$.
Note that every step of the MMP is $S$-positive.
Thus, if the strict transform of any component $W_j$ is contained in the exceptional locus of $X_i\dashrightarrow X_{i+1}$, then the number of components of $\phi_i^{-1}(\phi(S))$ drops.
Henceforth, it suffices to show that each such component is eventually contained in the exceptional locus of a step of the MMP.
Assume $\phi(W_1)\subseteq \phi(S)$ is maximal among the sets $\phi(W_j)$'s with respect to the inclusion.
Let $z\in \phi(W_1)$ be a general point.
Up to re-ordering the $W_j$'s, since $X \rar Z$ has connected fibers, we may assume that 
$\phi^{-1}(z)\cap W_1\cap S$ is non-empty.
Hence, for a general point 
$w \in \phi^{-1}(z)\cap W_1$, we can find a curve $C$
such that $w \in C$, $C \nsubseteq S$,
and $C$ intersects $S$ non-trivially.
In particular, we have that 
$C\subseteq {\rm Bs}_{-}(K_X+B-S+\bM X./Z)$.
In particular, since we have $w \in C$, it follows that $w \in {\rm Bs}_{-}(K_X+B-S+\bM X./Z)$.
Since $w$ is a general point in $\phi^{-1}(z)\cap W_1$, we also get that
$\phi^{-1}(z)\cap W_1 \subseteq {\rm Bs}_{-}(K_X+B-S+\bM X./Z)$.
Since $z$ is general, we conclude that
$W_1 \subset {\rm Bs}_{-}(K_X+B-S+\bM X./Z)$.
For $\lambda_1>0$ small enough, we have that
\[
W_1 \subset {\rm Bs}(K_X+B-S+\lambda_1A+\bM X./Z).
\]
Since $\lambda_\infty=0$, we conclude that for some $i$ the birational map $X\dashrightarrow X_i$ is a minimal model for
$(X,B-S+\lambda_1A,\bM./Z)$.
In particular, $W_1$ must be contained in the exceptional locus of $X\dashrightarrow X_i$. Hence, after finitely many steps of this MMP, condition (iii) is satisfied.

Let $X'$ be a model where condition (iii) holds. 
By construction, the generalized pair
$(X',B',\bM.)$ is obtained by a partial run $X\dashrightarrow X'$ of the MMP for $(X,B-S,\bM.)$.
In particular, $(X',B'-S',\bM.)$ is generalized dlt and $\qq$-factorial.  
Hence, $(X',B'-\epsilon \lfloor B'\rfloor)$ is klt, where $0 < \epsilon \ll 1$.
By~\cite[Example 2.6]{FG14}, the pair obtained by adjunction of $(X',B'-\epsilon \lfloor B'\rfloor+\epsilon S')$ to $S'$ is semi-log canonical.
In turn, by letting $\epsilon \to 0$, it follows that the pair obtained by adjunction of $(X',B')$ to $S'$ is semi-log canonical.
Hence, the generalized pair obtained by adjunction of $(X',B',\bM.)$ to $S'$ is generalized semi-log canonical.
\end{proof}

\begin{lemma}
\label{lem:lcy-horizontal-coeff}
Let $c$ and $p$ be nonnegative integers and $\Lambda \subset \qq_{>0}$ be a closed set satisfying the DCC.
There exists a finite subset $\mathcal{R} \coloneqq \mathcal{R}(c,p,\Lambda)\subseteq \Lambda$ satisfying the following. 
Let $(X,B,\bM.)$ be a generalized log canonical pair over $Z$ and $X\rightarrow Z$ be a fibration for which the following conditions hold:
\begin{itemize}
    \item the generalized pair $(X,B,\bM.)$ is log Calabi--Yau over $Z$;
    \item the generalized pair $(X,B,\bM.)$ has coregularity $c$; 
    \item $p\bM.$ is b-Cartier; and
    \item the coefficients of $B$ belong to $\Lambda$.
\end{itemize}
Then, the coefficients of $B_{\rm hor}$ belong to $\mathcal{R}$.
\end{lemma}

\begin{proof}
Let $(X_i,B_i,\bM i.)$ be a sequence of generalized pairs as in the statement
and $\phi_i\colon X_i\rightarrow Z_i$ be the corresponding contractions.
Assume there exist prime divisors $P_i\subset X_i$ for which 
$d_i \coloneqq {\rm coeff}_{P_i}(B_i)$ is strictly increasing
and $P_i$ dominates $Z_i$.
Assume that some generalized log canonical center of 
$(X_i,B_i,\bM i.)$ is vertical over $Z_i$.
We may replace $(X_i,B_i,\bM i.)$ with a generalized dlt modification
and assume there is a component $S_i\subseteq \lfloor B_i\rfloor$ that is vertical over $Z_i$.
Furthermore, up to choosing a different vertical component possibly dominating a different subset of $Z_i$, we may assume that there is a generalized log canonical center of $(X_i,B_i,\bM i.)$ dimension $c$ contained in $S_i$.
By Lemma~\ref{lem:exchange-fiber}, up to losing the dlt property for $(X_i,B_i,\bM i.)$,
we may assume that $S_i$ is the set-theoretic pre-image of $\phi$.
Let $W_i$ be the normalization of $S_i$ 
and let $W_i\rightarrow Z_{W_i}$ be the fibration obtained by the Stein factorization of 
$W_i\rightarrow \phi_i(S_i)$.
Let $(W_i,B_i,\bN i.)$ be the generalized pair obtained by generalized adjunction
of $(X_i,B_i,\bM i.)$ to $W_i$.
Note that $P_i\cap S_i$ dominates $\phi(S_i)$.
Hence, there is a component of $B_{W_i}$ with coefficient in $D_{d_i}(\Lambda)$ which is horizontal over $Z_{W_i}$ (see Lemma~\ref{lemma-coeff-adj}).
Observe that the following conditions hold:
\begin{itemize}
    \item the generalized pair
    $(W_i,B_i,\bN i.)$ is log Calabi--Yau over $Z_{W_i}$;
    \item the generalized pair $(W_i,B_i,\bN i.)$ has coregularity $c$;
    \item $p\bN i.$ is b-Cartier;
    \item the coefficients of $B_{W_i}$ belong to $D(\Lambda)$; and 
    \item there is a component $Q_i$ of $B_{W_i}$ horizontal over $Z_i$ whose coefficient belong to $D_{d_i}(\Lambda)$.
\end{itemize}
We replace $(X_i,B_i,\bM i.)$ 
with $(W_i,B_{W_i},\bN i.)$
and $P_i$ with $Q_i$.
After finitely many replacements, 
we may assume that for every $i$
the generalized log canonical centers of $(X_i,B_i,\bM i.)$
are horizontal over $Z_i$.~\cite[Theorem 2]{FMP22} applied
to the general fiber of $X_i\rightarrow Z_i$
implies that the coefficients of $B_{\rm hor}$ belong to an ACC set.
Thus, we conclude that the coefficients
of $B_{\rm hor}$ belong to a finite set $\mathcal{R}$
which only depends on $c$, $p$ and $\Lambda$.
\end{proof}

The proof of the following corollary is verbatim from the previous proof by replacing~\cite[Theorem 2]{FMP22}
with~\cite[Corollary 3]{FMM22}.

\begin{corollary}\label{cor:pseff-treshold-coreg-0}
Let $(X,B,\bM.)$ be a generalized log canonical pair over $Z$ and $X\rightarrow Z$ be a fibration for which the following conditions hold:
\begin{itemize}
    \item the generalized pair $(X,B,\bM.)$ is log Calabi--Yau over $Z$;
    \item the generalized pair $(X,B,\bM.)$ has coregularity $0$; 
    \item $2\bM.$ is b-Cartier; and
    \item the coefficients of $B$ belong to $\mathcal{S}$.
\end{itemize}
Then, the coefficients of $B_{\rm hor}$ belong to $\{1,\frac{1}{2}\}$.
\end{corollary}

\begin{notation}
\label{not:approxiatiom}
{\em 
Let $\Lambda \subset \mathbb{Q}_{>0}$ be a closed set of rational numbers satisfying the DCC.
Given a natural number $m\in \zz_{>0}$, we consider the partition
\[
\mathcal{P}_m \coloneqq 
\left\{ 
\left(0,\frac{1}{m}\right],
\left(\frac{1}{m},\frac{2}{m}\right], 
\dots,
\left(\frac{m-1}{m},1\right]  
\right\} 
\] 
of the interval $[0,1]$.
Denote by $I(b,m)$ the interval of $\mathcal{P}_m$
containing $b\in \Lambda$.
Define the number
\[
b_m \coloneqq \sup\{x\mid x\in I(b,m)\cap \Lambda\} \in \Lambda. 
\] 
For every positive integer $m$ and every $b \in \Lambda$, 
we have that $b\leq b_m$
as $b\in \Lambda \cap I(b,m)$.
If $b\in \Lambda$ is fixed and $m$ divisible enough, we have that $b=b_m$.
The set 
$\mathcal{C}_m \coloneqq \{b_m \mid b\in \Lambda\}$ is finite and we have that the set
\[
\Lambda = \bigcup_{m\in \zz_{>0}} \mathcal{C}_m
\] 
satisfies the DCC. 
Given a boundary divisor $B\geq 0$ on a quasi-projective variety $X$, we can write
$B=\sum_j b^{(j)}B^{(j)}$ in a unique way such that
the $B^{(j)}$'s are pairwise different prime divisors on $X$.
If the coefficients of $B$ belong to $\Lambda$,
we define
\[
B_m \coloneqq \sum_j b_m^{(j)}B^{(j)}.
\]
It follows that $B\leq B_m$.
} 
\end{notation} 

\begin{theorem}\label{thm:reduction-of-coefficients}
Let $c$ and $p$ be nonnegative integers and $\Lambda \subset \qq_{>0}$ be a set satisfying the DCC with rational accumulation points.
There exists a finite subset $\mathcal{R} \coloneqq \mathcal{R}(c,p,\Lambda) \subset \bar{\Lambda} \subset \qq_{>0}$ satisfying the following.
Let $(X,B,\bM.)$ be a generalized log canonical pair over $Z$, $X\rightarrow Z$ be a contraction,
and $z\in Z$ be a point. 
Assume the following conditions are satisfied:
\begin{itemize}
    \item the variety $X$ is of Fano type over $Z$;
    \item the divisor $B$ has coefficients in $\Lambda$;
    \item $p\bM.$ is b-Cartier;
    \item the generalized pair $(X,B,\bM.)$ has coregularity $c$ around $z$; and 
    \item the divisor $-(K_X+B+\bM X.)$ is nef over $Z$.
\end{itemize}
There exists a birational transformation
$X\dashrightarrow X'$ over $Z$ and a generalized pair
$(X',\Gamma',\bM.)$ satisfying the following:
\begin{itemize}
    \item the coefficients of $\Gamma'$ belong to $\mathcal{R}$;
    \item the pair $(X',\Gamma',\bM.)$ has coregularity $c$ over $z$; 
    \item the divisor $-(K_{X'}+\Gamma'+\bM X'.)$ is nef over a neighborhood of $z\in Z$; and 
    \item if $(X',\Gamma',\bM.)$ is $N$-complemented over $z\in Z$, then
    $(X,B,\bM.)$ is $N$-complemented over $z\in Z$.
\end{itemize}
\end{theorem}

\begin{proof}
Let $(X,B,\bM.)$ be a generalized pair as in the conditions of the theorem. 
By passing to a $\qq$-factorial generalized dlt modification,
we may assume the considered generalized pairs are gdlt and $\qq$-factorial.
We denote by $m(X,B,\bM.)$ the minimal $m$ for which $\mathcal{R}=\mathcal{C}_{m}$
satisfies the statement of the theorem for $(X,B,\bM.)$.
Since $B_m=B$ for $m$ large enough, 
then $m(X,b,\bM.)$ is finite. 
It suffices to show that 
$m(X,B,\bM.)$ is bounded above by a constant that only depends on $c$, $p$ and $\Lambda$. 
Assume that this is not the case.
Then, we may find a sequence of generalized pairs
$(X_i,B_i, \bM i.)$, 
contractions $X_i\rightarrow Z_i$,
and closed points $z_i\in Z_i$, 
satisfying the conditions of the theorem,
for which $m_i \coloneqq m(X_i,B_i,\bM i.)-1$ is strictly increasing. 
In particular, we have that 
$B_{i,m_i}-B_i$ is a non-trivial effective divisor.
Let $P_i$ be a prime component of $B_{i,m_i}-B_i$ that intersects the fiber over $z$.
We study how the singularities of $(X_i,B_i,\bM i.)$ over $z_i\in Z_i$ and the nefness of $-(K_{X_i}+B_i+\bM i.)$ over $Z_i$ change 
as we increase the coefficient at $P_i$.
\\

\textit{Step 1:} For the generalized pair
$(X_i,B_i,\bM i.)$, we will produce a positive real number $t_i$ which either computes a log canonical threshold
or a pseudo-effective threshold.\\

For each generalized pair 
$(X_i,B_i,\bM i.)$, we will define a real number $t_i$ as follows.
We consider the generalized pairs
\begin{equation}\label{eq:pair-parameter}
(X_i,B_{i,t},\bM i.) \coloneqq 
(X_i,B_i + t({\rm coeff}_{P_i}(B_{i,m_i}-B_i))P_i,\bM i.).
\end{equation} 
Let $t_{i,0}$ be the largest real number for which the generalized pair~\eqref{eq:pair-parameter}
is generalized log canonical over $z_i\in Z_i$ and 
\[
-(K_{X_i}+ B_i + t({\rm coeff}_{P_i}(B_{i,m_i}-B_i))P_i+\bM i,X_i.)
\] 
is nef over a neighborhood of $z_i\in Z_i$.
Assume that $t_{i,0}<1$.
Then, for $t>t_{i,0}$ close enough to $t_{i,0}$ one of the following conditions hold:
\begin{itemize}
    \item[(i)] the generalized pair
    $(X_i,B_{i,t},\bM i.)$ is not generalized log canonical over $z_i\in Z_i$; or
    \item[(ii)] the divisor $-(K_{X_i}+B_{i,t}+\bM i,X_i.)$ is not pseudo-effective over a neighborhood of $z_i\in Z_i$ and $(X_i,B_{i,t},\bM i.)$ is generalized log canonical over a neighborhood of $z_i \in Z_i$; or
    \item[(iii)] the divisor $-(K_{X_i}+B_{i,t}+\bM i,X_i.)$ is pseudo-effective but it is not nef over every neighborhood of $z_i\in Z_i$ and $(X_i,B_{i,t},\bM i.)$ is generalized log canonical over a neighborhood of $z_i \in Z_i$.
\end{itemize}

Assume that case (i) holds.
Then, we set 
\[
t_i \coloneqq t_{i,0}{\rm coeff}_{P_i}(B_{i,m_i}-B_i)+{\rm coeff}_{P_i}(B_i) 
\in 
[{\rm coeff}_{P_i}(B_i),{\rm coeff}_{P_i}(B_{i,m_i})).
\]
We show that $t_i$ computes a generalized
log canonical threshold of coregularity at most $c$.
Indeed, we have that 
\[
t_i = {\rm glct}((X_i,B_{i}-{\rm coeff}_{P_i}(B_i)P_i,\bM i.);P_i),
\]
so $t_i$ is a generalized log canonical threshold over $z_i\in Z_i$.
By construction, the support of $P_i$ contains a generalized log canonical center of the generalized pair $(X_i,B_i+t_iP_i,\bM i.)$.
Set $\bN i. \coloneqq -(K_{X_i}+B_i+t_iP_i+\bM i.)$ as a nef b-divisor over $Z_i$, i.e., we set $\bN i.$ to be the b-Cartier closure of $-(K_{X_i}+B_i+t_iP_i+\bM i,X_i.)$.
Then, we have that 
\begin{equation}\label{eq:new-pair-nef-part}
(X_i,B_i+t_iP_i,\bM i. + \bN i.)
\end{equation} 
is a generalized log Calabi--Yau pair over $Z_i$.
Furthermore, the generalized log canonical centers of~\eqref{eq:new-pair-nef-part}
are the same as the generalized log canonical centers of
$(X_i,B_i+t_iP_i,\bM i.)$.
By~\cite[Theorem 1.4]{FS23},
up to replacing $(X_i,B_i+t_iP_i,\bM i.)$ with a generalized dlt modification,
the support of $P_i$ contains a generalized log canonical center of $(X_i,B_i+t_iP_i,\bM i.)$ of dimension at most $c$.
Hence, $t_i$ is a generalized log canonical threshold of coregularity at most $c$.

Assume that case (ii) holds. 
Then, we set 
\[
t_i \coloneqq t_{i,0}{\rm coeff}(B_{i,m_i}-B_i)+{\rm coeff}_{P_i}(B_i)\in 
[{\rm coeff}_{P_i}(B_i),{\rm coeff}_{P_i}(B_{i,m_i})).
\] 
In this case, we can find a Mori fiber space structure $X'_i\rightarrow W_i$ over $Z_i$ such that the following conditions are satisfied:
\begin{itemize}
    \item the generalized pair
    $(X'_i,B'_i-{\rm coeff}_{P'_i}(B'_i)P'_i+t_iP'_i,\bM i.)$ is generalized log Calabi--Yau over $W_i$; and 
    \item the prime divisor $P'_i$ is ample over $W_i$.
\end{itemize}
Note that $(X'_i,B'_i-{\rm coeff}_{P'_i}(B'_i)P'_i+t_iB'_i,\bM i.)$ has coregularity at most $c$.
We have that $t_i$ is the coefficient 
of a component of $B'_i-{\rm coeff}_{P'_i}(B'_i)+t_iP'_i$ which is horizontal over $W_i$.

From now on, we assume that (i) and (ii) do not happen. 
Assume that (iii) holds. 
Then, there exists a birational contraction 
$X_i\dashrightarrow X'_i$ which is 
$(K_{X_i}+B_{i,t_{i,0}}+\bM i,X_i.)$-trivial.
Indeed, this contraction is defined by the partial 
$-(K_{X_i}+B_{i,t'}+\bM i,X_i.)$-MMP with scaling of $P_i$, for $t'$ close enough to $t_{i,0}$ as in (iii).
By construction, the first scaling factor is $t'-t_{i,0}$, and since $-(K_{X_i}+B_{i,t'}+\bM i,X_i.)$ is pseudo-effective over $Z_i$ and $X_i \rar Z_i$ is of Fano type, this MMP terminates with a good minimal model.
In particular, at the last step of this MMP, the scaling factor is 0.
Then, $X_i'$ is the outcome of the last step where the scaling factor is $t'-t_{i,0}$.
In particular, we have that
$-(K_{X'_i}+B'_{i,t}+\bM i, X_i'.)$
is nef over $Z_i$ for $t>t_{i,0}$ close enough to $t_{i,0}$.
Note that an $N$-complement of $(X_i',B'_{i,t_{i,0}},\bM X_i'.)$ 
induces an $N$-complement of $(X_i,B_{i,t_{i,0}},\bM i.)$ by pulling back, 
and so an $N$-complement of $(X_i,B_i,\bM i.)$.
Henceforth, we may replace $X_i$ with $X'_i$ and keep increasing $t$. 
Since $X_i$ is of Fano type over $Z_i$,
there are only finitely many birational contractions $X_i\dashrightarrow X'_i$.
Therefore, we can replace $X_i$ with $X_i'$ only finitely many times.
Thus, after finitely many birational contractions, we either have that $t_{i,0}=1$,
that $t_i<1$ is a log canonical threshold, or that
$t_i<1$ is a pseudo-effective threshold.

We assume that $t_{i,0}=1$.
Then there exists a birational contraction
$X_i \dashrightarrow X'_i$ 
and a generalized log canonical pair
\begin{equation}\label{eq:gen-pair-increased-coeff} 
(X'_i,B'_i-({\rm coeff}_{P_i}B'_i +{\rm coeff}_{P_i}B'_{i,m_i})P_i, \bM i.)
\end{equation}  
for which the divisor 
\[ 
-(K_{X'_i}+B'_i-({\rm coeff}_{P_i}B'_i +{\rm coeff}_{P_i}B'_{i,m_i})P_i+\bM i,X'_i.)
\] 
is nef over $Z_i$. 
Note that the coefficients of the boundary of the generalized pair~\eqref{eq:gen-pair-increased-coeff} 
belong to $\Lambda$ and the variety $X'_i$ is of Fano type over $Z_i$.
The b-nef divisor $p\bM X'_i.$ is b-Cartier.
By construction, if the generalized pair~\eqref{eq:gen-pair-increased-coeff} admits an $N$-complement, then so does $(X_i,B_i,\bM i.)$. 
We can replace $(X_i,B_i,\bM i.)$ with the generalized pair~\eqref{eq:gen-pair-increased-coeff}. 
By doing so, we decrease the number of components
of $B_{i,m_i}-B_i$. 
By the choice of $m_i$, the divisor $B_{i,m_i}-B_i$ cannot be
zero after this replacement. 
Thus, we may pick a new component and start increasing its coefficient (to this end, notice that $X_i'$ is $\qq$-factorial by construction).
Note that this process must terminate either with $t_i<1$ a log canonical threshold
or $t_i<1$ a pseudo-effective threshold.
Otherwise, we contradict the definition of $m_i$.\\

\textit{Step 2:} We show that a subsequence of the $t_i$'s is strictly increasing.\\

Up to passing to a subsequence, we may assume that $t_i$ is either strictly increasing, strictly decreasing, or it stabilizes.
The condition $t_i\in [{\rm coeff}_P(B_i),{\rm coeff}_P(B_{i,m_i}))$ implies that $t_i$ must be strictly increasing.\\

\textit{Step 3:} We finish the proof of the statement.\\

If case (i) happens infinitely many times, then we get a contradiction to the ACC for generalized log canonical thresholds with bounded coregularity~\cite[Theorem 1]{FMP22}. If case (ii) happens infinitely many times, 
then we get a contradiction to Lemma~\ref{lem:lcy-horizontal-coeff}. 
In any case, we get a contradiction.
Hence, the sequence $m_i$ has an upper bound.
\end{proof}

\section{Semi-log canonical pairs}
\label{sec:slc}

In this section, we discuss
the index of semi-log canonical pairs.
We show that to control
the index of a semi-log canonical log Calabi--Yau pair of coregularity $c$
it suffices to control
the index of dlt log Calabi--Yau pairs of coregularity $c$.
To prove the main
statement of this section, we will need to use the language of 
admissible
and pre-admissible sections. 
The preliminary results for this section are taken from~\cite{FG14,Fuj00,Xu20}.

The following definition is due to Fujino~\cite[Definition 4.1]{Fuj00}.

\begin{definition}
{\em 
Let $(X,B)$ be a possibly disconnected projective semi-dlt pair
of dimension $n$
and assume that 
$N(K_X+B)$ is Cartier.
Let $(X',B')$ be its normalization
and $D^n\subset X'$ be the normalization
of 
%the double locus $D\subset X$.
$\lfloor B' \rfloor$.
As usual, we denote by
$(D^n,B_{D^n})$ the dlt pair obtained by adjunction of
$(X',B')$ to $D^n$.
We define the concept of 
{\em pre-admissible}
and 
{\em admissible}
sections in $H^0(X,\O X. (I(K_X+B)))$ by induction
on the dimension using the two following rules:
\begin{enumerate}
    \item we say that a section
    \[
    s\in H^0(X,\mathcal{O}_X(I(K_X+B)))
    \] 
    is {\em pre-admissible} if 
    $s|_{D^n} \in H^0(D^n,\mathcal{O}_{D^n}(I(K_{D^n}+B_{D^n})))$ is admissible.
    This set is denoted by
    $PA(X,I(K_X+B))$; and
    \item we say that 
    \[
    s\in H^0(X,\mathcal{O}_X(I(K_X+B))
    \]
    is {\em admissible} if 
    $s$ is pre-admissible
    and 
    $g^*(s|_{X'_i}) = s|_{X'_j}$ holds for every crepant birational map
    $g\colon (X'_i,B_{X'_i})\dashrightarrow 
    (X'_j,B_{X'_j})$,
    where $X'=\sqcup X'_i$
    are the irreducible components of $X'$.
    The set of admissible sections
    is denoted by
    $A(X,I(K_X+B))$.
\end{enumerate}
}
\end{definition}

The following lemma is due to Gongyo~\cite[Remark 5.2]{Gon13}.

\begin{lemma}\label{lem:pull-back-admissible}
Let $(X,B)$ be a projective semi-dlt pair
for which $I(K_X+B)\sim 0$. 
Let $\pi\colon (X',B')\rightarrow (X,B)$ be its normalization.
Then, a section
$s\in H^0(X,\O X.(I(K_X+B)))$
is pre-admissible
(resp. admissible)
if and only if 
$\pi^*s \in 
H^0(X',\O X.(I(K_{X'}+B')))$ is 
pre-admissible
(resp. admissible). 
\end{lemma}

The following lemma allows us 
to descend linear equivalence
from normal varieties
to semi-log canonical varieties
(see, e.g.,~\cite[Lemma 4.2]{Fuj00}).

\begin{lemma}\label{lem:descending-nonormal}
Let $(X,B)$ be a projective semi-log canonical pair
for which $I(K_X+B)$ is an integral divisor.
Let $(X',B')\rightarrow (X,B)$ be its normalization
and $(Y,B_Y)$ a $\qq$-factorial dlt modification of $(X',B')$.
Assume that $I(K_Y+B_Y)$ is Cartier.
Then, a section
$s\in PA(I(K_Y+B_Y))$ descends
to $H^0(X,\O X. (I(K_X+B)))$.
In particular, if we have that
$I(K_Y+B_Y)\sim 0$, and there exists a nowhere vanishing section $0\neq s\in PA(I(K_Y+B_Y))$, then we have that
$I(K_X+B)\sim 0$.
\end{lemma}

\begin{proof}
The first part of the statement is \cite[Lemma 4.2]{Fuj00}.
Now, $s \in PA(I(K_Y+B_Y))$ is nowhere vanishing, it descends to a nowhere vanishing section of $\O X.(I(K_X+B))$, thus showing that $I(K_X+B)\sim 0$. 
\end{proof}

In the context of connected dlt pairs,
the set of admissible sections
is the same as the set of pre-admissible sections (see, e.g.,~\cite[Proposition 4.7]{Fuj00}).

\begin{lemma}\label{lem:PA=A}
Let $(X,B)$ be a connected projective dlt pair 
with $\lfloor B \rfloor \neq 0$.
Assume that $I(K_X+B)\sim 0$
and $I$ is even.
Then, we have that 
\[
PA(I(K_X+B)) = A(I(K_X+B)).
\] 
\end{lemma}

On the other hand, in the dlt setting, we can lift
admissible sections from the boundary
to pre-admissible sections
on the whole pair (see, e.g.,~\cite[Lemma 3.2.14]{Xu20}).

\begin{lemma}\label{lem:lifting-A-to-PA}
Assume that $(X,B)$ is a possibly disconnected
projective dlt pair.
Assume that $I(K_X+B)\sim 0$
and $I$ is even.
Assume that
\[
0\neq s \in A(\lfloor B\rfloor, I(K_X+B)|_{\lfloor B\rfloor}).
\] 
Then, there exists 
\[
0\neq t \in PA(X,I(K_X+B)) 
\] 
for which $t|_{\lfloor B\rfloor}=s$.
\end{lemma}

The following lemma states
that the boundedness of indices
for klt Calabi--Yau pairs
together with the 
boundedness of B-representations
imply the existence of admissible sections (see, e.g.,~\cite[Proposition 3.2.7]{Xu20}).

\begin{lemma}\label{lem:admissible-klt-case}
Let $c$ be a nonnegative integer
and $\Lambda$ be a set of rational numbers satisfying the descending chain condition.
Assume Conjecture~\ref{conj:index} and Conjecture~\ref{conj:b-rep} in dimension $c$.
There is a constant $I(\Lambda,c)$, only depending on $\Lambda$ and $c$,
satisfying the following.
Let $(X,B)$ be a projective klt log Calabi--Yau pair
with coefficients in $\Lambda$
and dimension $c$.
Then, there is a section
\[
0\neq s \in A(X,I(\Lambda,c)(K_X+B)).
\]
\end{lemma}

Note that for a klt log Calabi--Yau pair
an admissible section is nothing else than a section
which is invariant under the pull-back via crepant birational
transformations.

Finally, we prove the following lemma
that allows us to produce
admissible sections
on possibly disconnected dlt pairs, 
once we know the existence
of admissible sections on connected dlt pairs.
The proof is similar to that of~\cite[Proposition 3.2.8]{Xu20}.

\begin{lemma}\label{lem:from-conn-to-disc}
Let $d$ be a positive integer.
Let $(X,B)$ be a possibly disconnected 
projective dlt log Calabi--Yau pair.
Assume that for every component $(X_i,B_i)$ of $(X,B)$, 
we have a non-trivial section on 
$A(X_i,I(K_{X_i}+B_i))$. 
Then, we have that
$A(X,I(K_X+B))$ admits a nowhere vanishing section.
\end{lemma}

\begin{proof}
Let $(X,B)$ be a possibly disconnected
projective dlt log Calabi--Yau pair.
We write
$(X_i,B_i)$ for its components
for $i\in \{1,\dots,k\}$.
By assumption, for each $i$, we have 
\[
0\neq s_i \in A(X_i,I(K_{X_i}+B_i)).
\] 
For $\lambda_i\in \cc$, we define 
\[
s\coloneqq(\lambda_1s_1,\dots,\lambda_ks_k)\in 
H^0(X,N(K_X+B)).
\] 
Let $G={\rm Bir}(X,B)$.
We claim that the image
of $G$ in
$GL(H^0(X,\O X. (I(K_X+B))))$ is finite.
We denote by
\[
\rho_I \colon {\rm Bir}(X,B)\rightarrow
GL(H^0(X,\O X.(I(K_X+B))))
\]
the usual map
induced by pulling back sections.
Thus, we want to show that
$\rho_I(G)$ is finite. By~\cite{FG14}, the finiteness of $\rho_I(G)$ is known if $X$ is connected.
Thus, we need to reduce the disconnected case to the connected one.
Note that for every $g\in G$, we have that
$\rho_I(g)^{k!}$ has finite order
by~\cite[Theorem 3.15]{FG14} and the fact that the order of any permutation in $S_k$ divides $k!$.
Hence, we conclude that $\rho_I(G)$
is a finitely generated subgroup
of finite exponent
of a general linear group, where the bound is determined by $k!$ and the least common multiple of the orders of the pluricanonical representations of each irreducible component.
Indeed, notice that $\rho_I(G)$ is finitely generated, as it is the extension of two finite groups: the image via $\rho_I$ of the subgroup fixing the irreducible components of $X$, which is isomorphic to the product of the pluricanonical representations of each irreducible component (hence, a finite group by \cite{FG14}), and a subgroup of $S_k$.
By a theorem due to Burnside, known as the bounded Burnside problem for linear groups \cite[Theorem 6.13]{CD21}, we conclude that $\rho_I(G)$ is finite.

Consider the section
\[
t\coloneqq\sum_{\sigma \in G}\sigma(s).
\]
By construction, we have that
$t\in A(X,I(K_X+B))$.
Indeed, $t$ is invariant under the action of any birational transformation of $(X,B)$. 
Thus, the restriction of $t$ to every log canonical center is also invariant. 
It suffices to show that 
$t$ is non-trivial on each component of $X$.
By considering orbits of the action,
we may assume that ${\rm Bir}(X,B)$
acts transitively
on the components of $X$.
Consider
the basis
$((0,\dots,s_i,\dots,0))_{1 \leq i \leq k}$
of $H^0(X,\O X. (I(K_X+B)))$.
Since the sections $s_i$ are admissible,
in this basis, 
the action
of $\rho_I(g)$ is represented
by a matrix
whose diagonal entries are either 0 (if $g(X_i)\neq X_i$) or 1 (if $g(X_i)=X_i$).
Hence, by observing that the matrix associated to the identity element of $G$ is the identity matrix, it follows that the 
action of $\sum_{\sigma \in G}\sigma$
in this basis
is given by a non-trivial matrix whose diagonal entries are all integers greater than or equal to 1.
By the transitivity of the action
and the fact that the action of $\sum_{\sigma \in G}\sigma$ is
given by a matrix whose diagonal are positive integers, 
we deduce that we can find $\lambda_i\in \cc$ for which $t$ is non-zero on all components.
\end{proof}

Now, we are ready to prove the main theorem of this section.

\begin{theorem}\label{dlt-admissible}
Let $c$ be a nonnegative integer
and $\Lambda$ be a set of rational numbers
satisfying the descending chain condition.
Assume Conjecture~\ref{conj:index}
and Conjecture~\ref{conj:b-rep}
in dimension $c$.
There is a constant $I(\Lambda,c)$, only 
depending on $\Lambda$ and $c$, satisfying the following.
Let $(X,B)$ be a projective dlt pair
with coefficients in $\Lambda$
and coregularity $c$.
Assume that $I(\Lambda,c)(K_X+B)\sim 0$.
Then, there is a nowhere vanishing admissible section
\[
0\neq s \in A(X,I(\Lambda,c)(K_X+B)).
\] 
\end{theorem}

\begin{proof}
Let $I(\Lambda,c)$ be the positive integer
given by Lemma~\ref{lem:admissible-klt-case}.
Without loss of generality, we may assume that $I(\Lambda,c)$ is even.

By induction on $i$, 
we prove that every $i$-dimensional 
log canonical center
$V$ of $(X,B)$ satisfies that
\begin{equation}\label{eq:admissible-section}
0\neq s_V \in A(V,I(\Lambda,c)(K_V+B_V)),
\end{equation}
where $(V,B_V)$ is the pair
obtained by dlt adjunction
of $(X,B)$ to $V$.
If $i=c$, then the pair is klt
and the statement follows from Lemma~\ref{lem:admissible-klt-case}.

Now, assume that the statement
holds for every irreducible $i$-dimensional
dlt center of $(X,B)$.
Let $W$ be a log canonical center
of $(X,B)$ of dimension $i+1$. 
The pair
$(W,B_W)$ obtained from adjunction
is dlt of dimension $i+1$
and it holds that
$I(\Lambda,c)(K_W+B_W)\sim 0$.
Let $W_0$ be the union of all the log canonical centers of $(W,B_W)$.
Let $(W_0,B_{W_0})$ be the pair
obtained by performing adjunction
of $(W,B_W)$ to $(W_0,B_{W_0})$.
Hence, $(W_0,B_{W_0})$ is an $i$-dimensional
semi-dlt pair
with $I(\Lambda,c)(K_{W_0}+B_{W_0})\sim 0$.
Observe that $W_0$ may have multiple irreducible components.
Let $n_U\colon U\rightarrow W_0$ be the normalization
of $W_0$.
Let $(U,B_{U})$ be the pair
obtained by log pull-back
of $(W_0,B_{W_0})$ to $U$.
Then, we have that
$(U,B_{U})$ is a possibly
disconnected
projective dlt pair of dimension $i$
and coregularity $c$.
By~\cite[Theorem 1.4]{FS23}, we know that every component has coregularity $c$.
Furthermore, we have that
$I(\Lambda,c)(K_{U}+B_{U})\sim 0$.
By~\eqref{eq:admissible-section} in dimension $i$, we have that
each irreducible component 
$U_j$ of $U$ satisfies that
\[
0\neq s_{U,j} \in A(U_j,I(\Lambda,c)(K_{U_j}+B_{U_j})).
\]
By Lemma~\ref{lem:from-conn-to-disc}, we conclude that there exists a nowhere vanishing section
\[
0\neq s_{U} \in A(U,I(\Lambda,c)(K_U+B_U)).
\] 
By Lemma~\ref{lem:descending-nonormal}, we conclude that this section descends to
$s_{W_0} \in H^0(W_{0},\O W_0.(I(\Lambda,c)(K_{W_0}+B_{W_0})))$.
Note that we have 
$n_U^* s_{W_0}=s_U$.
By Lemma~\ref{lem:pull-back-admissible},
we conclude that
$s_{W_0} \in A(W_0,I(\Lambda,c)(K_{W_0}+B_{W_0}))$.
By Lemma~\ref{lem:lifting-A-to-PA}, we conclude that there exists 
\[
0\neq t_{W} \in PA(W,I(\Lambda,c)(K_W+B_W)).
\] 
Finally, since $W$ is connected, we conclude by Lemma~\ref{lem:PA=A}, that 
there is a section 
\[
0\neq t_W \in A(W,I(\Lambda,c)(K_W+B_W)). 
\]
This finishes the inductive step.

We conclude, that for every $i \in \{c,\dots,\dim X\}$, 
every $i$-dimensional log canonical center
of $(X,B)$ admits an admissible section. 
In particular, we get a nowhere vanishing section
$0\neq s_X \in A(X,I(\Lambda,c)(K_X+B))$ as claimed.
\end{proof}

The previous theorem allows controlling
the index of semi-log canonical pairs
once we can control the index 
of their normalization. 

\begin{theorem}\label{thm:slc-case}
Let $c$ be a nonnegative integer
and $\Lambda$ be a set of rational numbers
satisfying the descending chain condition.
Assume Conjecture~\ref{conj:index}
and Conjecture~\ref{conj:b-rep}
in dimension $c$.
There is a constant $I(\Lambda,c)$, only
depending on $\Lambda$ and $c$, satisfying the following.
Let $(X,B)$ be a projective semi-log canonical pair with 
coefficients in $\Lambda$ and coregularity $c$.
Let $(Y,B_Y)$ be a $\qq$-factorial dlt modification of a normalization of $(X,B)$. Assume that $I(\Lambda,c)(K_Y+B_Y)\sim 0$.
Then, we have that
$I(\Lambda,c)(K_X+B)\sim 0$.
\end{theorem}

\begin{proof}
We can consider $I(\Lambda,c)$ as in Theorem~\ref{dlt-admissible}. 
By Theorem~\ref{dlt-admissible} and Lemma~\ref{lem:from-conn-to-disc}, we know that there exists a nowhere vanishing pre-admissible section
\[
0\neq s_Y \in PA(Y,I(\Lambda,c)(K_Y+B_Y)).
\] 
By Lemma~\ref{lem:descending-nonormal}, we conclude that the linear equivalence
$I(\Lambda,c)(K_X+B)\sim 0$ holds.
\end{proof}

In the case of dimension 0,  Conjecture~\ref{conj:index}
and Conjecture~\ref{conj:b-rep}
are trivial.
Indeed, the only variety of interest is $\mathrm{Spec}(\mathbb K)$, no boundary is allowed for dimensional reasons, and $\mathrm{Bir}(\mathrm{Spec}(\mathbb K))$ is trivial.
Thus, we get the following statement.

\begin{theorem}
Let $(X,B)$ be a projective semi-log canonical 
Calabi--Yau pair 
of coregularity 0
and $\lambda$ be its Weil index.
Let $(Y,B_Y)$ be a $\qq$-factorial
dlt modification
of a normalization of $(X,B)$.
If $2\lambda(K_Y+B_Y)\sim 0$, 
then $2\lambda(K_X+B)\sim 0$.
\end{theorem}

Finally, Conjecture~\ref{conj:index}
and Conjecture~\ref{conj:b-rep}
are known in the case of klt pairs of dimension 1 or 2 (see, e.g.,~\cite{Xu20}). 
We get the following statement. 

\begin{theorem}\label{thm:slc-case-coreg=2}
Let $\Lambda$ be a set of rational numbers
satisfying the descending chain condition.
There exists a constant $I(\Lambda)$,
only depending on $\Lambda$, satisfying the following.
Let $(X,B)$ be a projective semi-log canonical 
Calabi--Yau pair
of coregularity 1 (resp. 2)
such that $B$ has coefficients in $\Lambda$.
Let $(Y,B_Y)$ be a $\qq$-factorial dlt modification of a normalization
of $(X,B)$.
If $I(\Lambda)(K_Y+B_Y)\sim 0$, then 
$I(\Lambda)(K_X+B)\sim 0$.
\end{theorem}

Let us note that Conjecture~\ref{conj:index} 
is known for klt $3$-folds (see, e.g.,~\cite{Xu20}).
However, Conjecture~\ref{conj:b-rep}
is still unknown 
in the case of klt Calabi--Yau $3$-folds.

\subsection{Lifting complements
from non-normal divisors in fibrations}
In this subsection, we prove a statement about lifting complements from non-normal divisors in fibrations.

\begin{theorem}\label{thm:lifting-fibration-pairs}
Let $\lambda$, $d$, and $c$ be nonnegative integers.
Assume that Conjecture~\ref{conj:index}$(c)$
and
Conjecture~\ref{conj:b-rep}$(c)$
hold.
Let $I \coloneqq I(D_\lambda,d-1,c,0)$
be the integer provided by Theorem~\ref{introthm:index-higher-coreg}$(d-1,c)$.
Up to replacing $I$ with a bounded multiple, further assume that $I$ is divisible by the integer provided by Theorem~\ref{thm:slc-case}$(D_\lambda,c)$.
Let $(X,B)$ be a projective $d$-dimensional log Calabi--Yau pair.
Assume that the following conditions hold:
\begin{itemize}
    \item $X$ is $\qq$-factorial and klt;
    \item there is a fibration $X\rightarrow W$, which is a $(K_X +  B - S)$-Mori fiber space;
    \item a component $S\subset \lfloor B\rfloor$ which is ample over the base
    and $(X,B-S)$ is dlt;
    \item the morphism $S \rar W$ has connected fibers;
    \item the coefficients of $B$ belong to $D_\lambda$; and 
    \item the pair $(X,B)$ has coregularity $c$.
\end{itemize}
Then, we have that $I(K_X+B)\sim 0$.    
\end{theorem}

\begin{proof}
The proof is formally identical to the proof of Theorem \ref{thm:lifting-fibration-gen-pairs}, with the only difference that we need to appeal to the results in \S~\ref{sec:slc}, since $S$ may not be normal.
For completeness, we include a full proof of the statement.

Let $(X,B)$, $S$, $f \colon X \rar W$, and $I \coloneqq I(D_\lambda,d-1,c,\lambda)$ be as in the statement.
First, we show that we can apply the inductive hypothesis to $S$.

By~\cite[Example 2.6]{FG14}, the pair obtained by adjunction of $(X,B-\epsilon \lfloor B\rfloor+\epsilon S)$ to $S$ is semi-log canonical.
In particular, $S$ is $S_2$.
In turn, by letting $\epsilon \to 0$, it follows that the pair obtained by adjunction of $(X,B)$ to $S$ is semi-log canonical.
In particular, let $(S,B_S)$ denote the pair obtained by adjunction from $(X,B)$, and let $(S^\nu, B \subs S^\nu.)$ denote its normalization.
By \cite[Lemma 2.28]{FMP22}, $(S^\nu,B \subs S^\nu.)$ has coregularity $c$.
Then, by Lemma \ref{lemma-coeff-adj}, $(S^\nu,B \subs S^\nu.)$ satisfies the assumptions of Theorem~\ref{introthm:index-higher-coreg}$(d-1,c)$ with constant $I$.
Then, by Theorem \ref{thm:slc-case}, we have
\begin{equation} \label{eq_dim_one_less_slc}
I (\K S. + B_S) \sim 0.
\end{equation}
By \cite[Proposition 4.32]{K13}, $S$ is semi-normal.
Then, by \cite[Lemma 2.3]{FI24} and the fact that $S \rar W$ has connected fibers, we have $f_* \O S. = \O W.$.
Lastly, we observe that, if $\dim X - \dim W =1$, since $S \rar W$ has connected fibers, it follows that $(B-S)_{\rm hor} \neq 0$.

Now, consider the short exact sequence
\begin{equation}\label{ses_fiber_space2_slc}
0 \rar \O X.(I(K_{X}+B)-S) \rar \O X.(I(K_{X}+B)) \rar \O S.(I(K_{S}+B_{S})) \rar 0.
\end{equation}
The exactness of \eqref{ses_fiber_space2_slc} follows verbatim as the exactness of \eqref{ses_fiber_space2}.
Since $I(K_{X}+B)-S\sim_{\qq,f} -S$, the divisor $-S$ is $f$-ample, and $\dim W < \dim X$,
we have
\[
f_*\O X.(I(K_{X}+B)-S)=0.
\]
Similarly, we write 
\[
I(K_{X}+B)-S\sim_{\qq,f} -S \sim_{\qq,f} K_X +(B-S).
\]
First, assume that $B_{\rm hor} \neq S$.
Note that $X$ is klt and $B-S$ is $f$-ample, since $f$ is a Mori fiber space and the assumption that $B_{\rm hor} \neq S$.
Thus, by the relative version of Kawamata--Viehweg vanishing, we have
\[
R^1f_*\O X.(I(K_{X}+B)-S)=0.
\]
Now, assume that $B_{\rm hor}=S$.
By the equality $B_{\rm hor}=S$ and the fact that $f$ is a Mori fiber space, we have
\[
I(K_{X}+B)-S\sim_{\qq,f} -S \sim_{\qq,f} K_X +(B-S) \sim_{\qq,f} K_X + B_{\rm ver} \sim_{\qq,f} K_X.
\]
Thus, we obtain
$$
I(K_{X}+B)-S - K_X \sim_{\qq,f} 0.
$$
Since $X$ is a klt variety, by \cite[Theorem 1-2-7]{KMM87}, we have that $R^1f_* \O X.(I(K_X+B)-S)$ is torsion free.
To conclude that it vanishes, it suffices to show that it has rank 0.
As observed at the end of the previous paragraph, we have that $\dim X \geq \dim Z +2$ under the additional assumption $B_{\rm hor}=S$.
Then, by applying Kawamata--Viehweg vanishing to a general fiber \cite[Theorem 2.70]{KM98}, we conclude that the rank of $R^1f_*\O X.(I(K_{X}+B)-S)$ is 0, thus implying that $R^1f_*\O X.(I(K_{X}+B)-S)=0$.

Therefore, by pushing forward \eqref{ses_fiber_space2_slc} via $f$, we obtain 
\[
f_* \O X.(I(K_{X}+B)) \simeq f_* \O S.(I(K_{S}+B_{S})).
\]
Now, taking global sections, we have
\begin{equation} \label{eq:sections_slc}
H^0(X, \O X.(I(K_{X}+B))) = H^0(S, \O S.(I(K_{S}+B_{S})))=H^0(S, \O S.) \neq 0.
\end{equation}
By \cite[Lemma 3.1]{FMM22}, \eqref{eq:sections_slc} implies that $I(K_{X}+B) \sim 0$.
\end{proof}

\section{Relative complements}
\label{sec:rel}

In this section, we prove an inductive statement regarding the existence of complements for Fano type morphisms with bounded  coregularity.

\subsection{Lifting sections from a divisor}
\label{subsec:lifting}

In this subsection, we introduce some tools to lift complements from a divisor of a log Fano pair.
Let $X$ be a Fano type variety
and $(X,B,\bM.)$ be a generalized log canonical pair for which $-(K_X+B+\bM X.)$ is nef. 
The main theorem of this subsection implies that we can lift complements for $(X,B,\bM.)$ from a component $S$ of $\lfloor B\rfloor$
under some suitable conditions explained in the following theorem.

\begin{theorem}\label{thm:lifting-from-divisor}
Let $d$, $c$ and $p$ be nonnegative integers 
and $\mathcal{R} \subset \qq_{>0}$ be a finite set. Let $N \coloneqq N(D(\mathcal{R}),d-1,c,p)$ be the integer provided by Theorem~\ref{introthm:Fano-coreg-c}$(d-1,c)$. Assume that $N$ is divisible by $p$ and by  $I_\mathcal{R}$.
Let $\pi \colon  X \rightarrow Z$ be a Fano type morphism, where $X$ is a $d$-dimensional variety.
Let $(X,B,\bM.)$ be a generalized log canonical pair over $Z$ and $z \in Z$ a point satisfying the following conditions:
\begin{itemize}
    \item the generalized pair $(X,B, \bM.)$ has coregularity at most $c$ over $z$;
    \item the divisor $B$ has coefficients in $\mathcal{R}$;
    \item $p\bM.$ is b-Cartier; and 
    \item the divisor $-(K_X+B+\bM X.)$ is nef over $Z$.
\end{itemize}
Assume that there exists $B_1\leq B$ and $\alpha \in (0,1]$ for which:
\begin{itemize}
    \item the generalized pair 
    $(X,B_1,\alpha \bM.)$ is generalized log canonical but it is not generalized klt over $z$; and
    \item the divisor
    $-(K_X+B_1+\alpha \bM X.)$ is big and nef over $Z$.
\end{itemize}
Then, $(X,B,\bM .)$ admits an $N$-complement over $z$.
\end{theorem}

In order to prove the main theorem of this section, we take inspiration from \cite[\S~6.6]{Bir19}.
In particular, we will first prove a weaker statement. 

\begin{proposition}\label{prop:lifting-from-divisor}
Let $d$, $c$ and $p$ be nonnegative integers
and $\mathcal{R} \subset \qq_{>0}$ be a finite set.
Let $N \coloneqq N(D(\mathcal{R}),d-1,c,p)$ be the integer provided by Theorem~\ref{introthm:Fano-coreg-c}$(d-1,c)$. Assume that $N$ is divisible by $p$ and by  $I_\mathcal{R}$.
Let $\pi \colon  X \rightarrow Z$ be a Fano type morphism, where $X$ is a $d$-dimensional variety.
Let $(X,B,\bM.)$ be a $\qq$-factorial generalized log canonical pair over $Z$ and $z \in Z$ a point satisfying the following conditions:
\begin{itemize}
    \item the generalized pair $(X,B, \bM.)$ has coregularity at most $c$ over $z$;
    \item the divisor $B$ has coefficients in $\mathcal{R}$;
    \item $p\bM.$ is b-Cartier; and
    \item the divisor $-(K_X+B+\bM X.)$ is nef over $Z$.
\end{itemize}
Assume there exists a boundary $\Gamma$ on $X$ and $\alpha \in (0,1)$ for which:
\begin{itemize}
    \item the generalized pair
    $(X,\Gamma,\alpha \bM .)$ is generalized plt over $z$;
    \item we have that $S=\lfloor \Gamma \rfloor \subset \lfloor B \rfloor$ intersects the fiber over $z$; and 
    \item the divisor 
    $-(K_X+\Gamma+\alpha \bM X.)$ is ample over $Z$.
\end{itemize}
Then, $(X,B,\bM.)$ admits an $N$-complement over $z$.
\end{proposition}

\begin{proof}
We will proceed by induction on the dimension, keeping the coregularity constant.
Over several steps, we will lift a complement from a divisor.
Since the statement is local over $z \in Z$, in the course of the proof we are free to shrink $Z$ around $z$.
In particular, all linear equivalences that are relative to $Z$ can be assumed to hold globally.
We add the fractions with denominator $p$ to the set $\mathcal{R}$.
This does not change the value of $N$, hence proving the statement for this new finite set is the same as proving it for the original $\mathcal{R}$.\\

\noindent\textit{Step 1.}
In this step, we reduce to the case where $(X,B,\bM.)$ is $\qq$-factorial generalized dlt.\\

Let $(X',B', \bM.)$ be a $\qq$-factorial dlt modification of $(X,B,\bM.)$.
Pick $E$ exceptional such that $-E$ is ample over $X$.
Notice that the existence of $E$ is guaranteed by the hypothesis that $X$ is $\mathbb Q$-factorial.
Also, let $(X',\Gamma',\alpha \bM.)$ be the trace of $(X,\Gamma,\bM.)$ on $X'$.
We observe that $\Gamma'$ may no longer be effective.

Since $(X,\Gamma,\alpha \bM.)$ is generalized plt with $\lfloor \Gamma \rfloor \subset \lfloor B \rfloor$ and $X' \rar X$ only extracts divisor that appear with coeffcient 1 in $B'$, for $0 < \lambda \ll 1$, the datum of $(X',(1-\lambda)B'+\lambda \Gamma',(1-\lambda + \lambda \alpha)\bM.)$ is actually a generalized pair (i.e., its boundary is effective) and it is generalized plt with $1-\lambda + \lambda \alpha \in (0,1)$.
Furthermore,
$$
-(\K X'. + (1-\lambda)B'+\lambda \Gamma' + (1-\lambda + \lambda \alpha)\bM X'.)
$$
is the pull-back of a divisor on $X$ that is relatively ample over $Z$.
Thus, for $\varepsilon >0$ small enough, we have that
$$
-(\K X'. + (1-\lambda)B'+\lambda \Gamma' + \varepsilon E + (1-\lambda + \lambda \alpha)\bM X'.)
$$
is ample over $Z$.
Hence, up to replacing $(X,B,\bM.)$ with $(X',B',\bM.)$ and $(X,\Gamma,\alpha \bM.)$ with $(X',(1-\lambda)B'+\lambda \Gamma' + \varepsilon E,(1-\lambda + \lambda \alpha)\bM.)$
we can assume that $X$ is $\qq$-factorial and $(X,B, \bM .)$ is generalized dlt.\\

\noindent\textit{Step 2.}
In this step, we prove that $S \rightarrow \pi(S)$ is a contraction.\\

As $\alpha \bM X.$ is the push-forward of a divisor that is nef over $Z$, its diminished base locus does not contain any divisor.
Let $\pi: X' \rightarrow X$ be a model where $\bM .$ descends. Let $K_{X'}+\Gamma'+\alpha\bM X'.$ be the crepant pullback of $K_X +\Gamma + \alpha\bM X.$ (this $\Gamma'$ is different from the one defined in step 1). For any $0<\delta< 1 $, we can write
\[(1-\delta)(K_{X'}+\Gamma' +\alpha \bM X'.)=K_{X'}+\Gamma' +(\alpha \bM X'.-\delta (K_{X'}+\Gamma' +\alpha \bM X'.)).\] 

As $\alpha \bM X'.$ is nef and $-(K_{X'}+\Gamma' +\alpha \bM X'.)$ is big and nef over $Z$, by \cite[Example 2.2.19]{Laz04a} there exists an effective divisor $E'$ such that $\alpha \bM X'.-\delta (K_{X'}+\Gamma' +\alpha \bM X'.) \sim_{\qq} A_{k}'+\frac{1}{k}E'$, for all positive integers $k$, where each $A_k'$ is ample over $Z$.

So, we can write $(1-\delta )(K_{X'}+\Gamma' +\alpha \bM X'.)\sim_\qq K_{X'}+\Gamma'+A_k'+\frac{1}{k}E' $.
If we choose $k$ large enough and $A_k'$ generically, then the sub-pair $(X',\Gamma'+A_k'+\frac{1}{k}E')$ is sub-plt. With those choices fixed, we define $A=\pi_\ast A_k'$, $E=\frac{1}{k}\pi_\ast E'$. Therefore,
\[(1-\delta)(K_{X}+\Gamma +\alpha \bM X.)\sim _\qq K_X +\Gamma +A+E,\]
with $(X,A+E+\Gamma)$ being plt.
Call $G \coloneqq A+E+\Gamma$. For $\delta$ small enough, we have that $-(K_X+G)$ is ample over $Z$ and $\lfloor G \rfloor = S$. From the exact sequence
$$
0 \rightarrow \mathcal{O}_X (-S) \rightarrow \mathcal{O}_X\rightarrow \mathcal{O}_S \rightarrow 0,
$$
we get the exact sequence
$$
\pi_* \mathcal{O}_X\rightarrow \pi_* \mathcal{O}_S \rightarrow R^1 \pi_* \mathcal{O}_X(-S).
$$

Since $-S= K_X+G -S -(K_X +G) $, with $(X,G-S)$ being klt and $-(K_X+G)$ being ample over $Z$, we have that $R^1 \pi_* \mathcal{O}_X(-S)=0$ by the relative Kawamata--Viehweg vanishing theorem. Therefore $\pi_* \mathcal{O}_X \rightarrow \pi_*\mathcal{O}_S$ is surjective.

Let $g \circ \pi' \colon S \rightarrow Z' \rightarrow Z$ be the Stein factorization of $\pi \colon S \rightarrow Z$. Then 
$\mathcal{O}_Z=\pi_* (\mathcal{O}_X)\rightarrow \pi_* \mathcal{O}_S=g_*\mathcal{O}_{Z'}$ is surjective.
As $\mathcal{O}_Z \rightarrow g_*\mathcal{O}_{Z'}$ factors as $\mathcal{O}_Z \rightarrow \mathcal{O}_{\pi(S)}\rightarrow g_*\mathcal{O}_{Z'}$, the morphism $\mathcal{O}_{\pi(S)} \rightarrow g_*(\mathcal{O}_{Z'})$ is surjective.
Then, it is an isomorphism, as the map $Z' \rightarrow Z$ is finite.
Hence $Z' \rightarrow \pi(S)$ is an isomorphism and $S \rightarrow \pi(S)$ is a contraction.
Restricting $K_X+G$ to $S$ shows that $S$ is of Fano type over $\pi(S)$.\\

\noindent \textit{Step 3.}
In this step, we use adjunction and consider a complement on $S$.\\

Consider a log resolution $f \colon X' \rightarrow X$ of $(X,B, \bM.)$ such that $\bM.$ descends on $X'$ and write $K_{X'}+B_{X'} \coloneqq f ^ *(K_X+B)$.
Let $S'$ be the strict transform of  $S$ and $g \colon S' \rightarrow S$ be the induced morphism. Let $(S, B_S, \bN .)$ be the generalized pair obtained by adjunction of $(X,B,\bM.)$ to $S$. By Lemma~\ref{lemma-coeff-adj},
the coefficients of $B_S$ are in $D(\mathcal{R})$ and the b-divisor $p\bN.$ is b-Cartier.
By Lemma~\ref{lem:fano-coreg-adjunction},
the coregularity of $(S,B_S,\bN.)$ is at most $c$. By Theorem~\ref{introthm:Fano-coreg-c}$(d-1,c)$ if $\dim \pi(S) =0$ or by the inductive hypothesis if $\dim \pi(S) >0$, the divisor $K_S+B_S+\bM S.$ has an $N$ complement $B_{S}^{+}$ over $z$ with coregularity at most $c$.
In the following steps, we will lift  $B_{S}^{+}$ to an $N$-complement $B_X^{+}$ of $K_X+B+\bM X.$ over $z$ with coregularity at most $c$.\\

\noindent \textit{Step 4.}
In this step, we introduce some divisors and prove some properties of these divisors.\\

Define $\Omega_{X'} \coloneqq  B_{X'} - \lfloor B ^{\geq 0}_{X'}\rfloor$ and $T_{X'} \coloneqq N \Omega_{X'}- \lfloor (N+1)\Omega_{X'} \rfloor -N(K_{X'}+B_{X'}+ \bM X'.)$.
We write $K_{X'} +\Gamma_{X'} \coloneqq  f^*(K_X+\Gamma)$. Now, we define a divisor $P_{X'}$ in the following way. For any prime divisor $D_{X'} \neq S'$, we set
$\coeff_{D_{X'}}(P_{X'})=-\coeff_{D_{X'}}\lfloor \Gamma_{X'}+N\Omega_{X'}-\lfloor (N+1)\Omega_{X'} \rfloor\rfloor$
and $\coeff_{S'}(P_{X'})=0$.
Hence, $P_{X'}$ is an integral divisor such that $J_{X'} \coloneqq \Gamma_{X'}+N\Omega_{X'}-\lfloor(N+1)\Omega_{X'} \rfloor+P_{X'}$ is a boundary, $(X',J_{X'}, \alpha \bM X'.)$ is generalized plt, and $\lfloor J_{X'} \rfloor=S'$.
For $D_{X'}\neq S'$ not exceptional over $X$, as $NB$ is integral,
we have that $\coeff_{D_{X'}}(N\Omega_{X'})$ is an integer. Thus, $\coeff_{D_{X'}}\lfloor (N+1)\Omega_{X'}\rfloor= \coeff_{D_{X'}} (N\Omega_{X'})$. So, $\coeff_{D_{X'}}(P_{X'})=-\coeff_{D_{X'}}(\Gamma_{X'})=0.$ We conclude that $P_{X'}$ is exceptional over $X$.\\

\noindent\textit{Step 5.}
In this step, we lift sections from $S'$ to $X'$ using Kawamata--Viehweg vanishing.\\

Observe that:
\begin{align*}
    T_{X'}+ P_{X'} &= N \Omega_{X'} - \lfloor (N+1)\Omega_{X'} \rfloor -N (K_{X'}+B_{X'}+\bM X'.) +P_{X'}\\
    &=  K_{X'} +\Gamma_{X'}-(K_{X'} +\Gamma_{X'}) +N \Omega_{X'} - \lfloor (N+1)\Omega_{X'} \rfloor -N (K_{X'}+B_{X'}+ \bM X'.) +P_{X'}\\
    &= K_{X'} +J_{X'} -(K_{X'}+\Gamma_{X'})-N(K_{X'}+B_{X'}+ \bM X'.). 
\end{align*}
Then, we have that $-(K_{X'}+\Gamma_{X'}+\alpha \bM X'.)-N(K_{X'}+B_{X'}+\bM X'.)+\alpha \bM X'.$ is big and nef over $Z$ and $(X',J_{X'}-S')$ is klt. Therefore, up to shrinking $Z$ around $z$, the relative Kawamata--Viehweg vanishing theorem implies that $h^1(X',\O X'.(T_{X'}+P_{X'}-S'))=0$.
So, we obtain 
$$ H^0(X',\O X'.(T_{X'}+P_{X'})) \rightarrow H^0(X',\O X'.((T_{X'}+P_{X'} )\mid_{S'})) \rightarrow H^1(X',\O X'.(T_{X'}+P_{X'}-S')) =0. $$
This means that we can lift sections of $(T_{X'}+P_{X'})\mid _{S'}$ from $S'$ to $X'$.\\

\noindent \textit{Step 6.}
In this step, we introduce a divisor $G_{S'}$ 
which is linearly equivalent
to $(T_{X'}+P_{X'})|_{S'}$.\\

We have 
$-N(K_S +B_S+\bN S.) = -N(K_S+B^+_S+B_S-B^+_S +\bN S.)\sim -N(B_S-B^+_S)=N(B^+_S-B_S) \geq 0$.
Define 
$K_{S'}+B_{S'}+\bN S'. \coloneqq  g^*(K_S+B_S+\bN S.)$.
Then, we have that 
$-N(K_{S'}+B_{S'}+\bN S'.)\sim Ng^*( B^+_S-B_S) \geq 0$.
Then, it follows that $-N(K_{X'}+B_{X'}+\bM X'.)\mid_{S'}= -N(K_{S'}
+B_{S'}+\bN S'.)\sim N g^*(B^+_S-B_S)$.
We define $G_{S'} \coloneqq   N g^*(B^+_S-B_S)+N \Omega_{X'}\mid_{S'} - \lfloor (N+1)\Omega_{X'}\mid_{S'} \rfloor +P_{X'}\mid_{S'}$.
By definition, we have 
$G_{S'} \sim (T_{X'}+P_{X'})\mid_{S'}$.\\

\noindent\textit{Step 7.} 
In this step, we prove that $G_{S'}$ is effective and that it lifts to an
effective divisor $G_{X'}$ on $X'$.\\

Assume $G_{S'}$ is not effective, then there exists some  prime divisor $C_{S'}$ with $\coeff_{C_{S'}}(G_{S'})<0$. As $N g^*(B^+_S-B_S)$ and $P_{X'}$ are effective, we must have that $\coeff _{C_{S'}}(N \Omega _{X'}\mid_{S'} - \lfloor (N+1)\Omega_{X'}\mid_{S'} \rfloor)$ is negative.
Since $X'$ is a log resolution, we have that the restriction to $S'$ commutes with taking the integral (resp. fractional) part of a divisor whose support is involved in the log resolution.
In particular, observe that 
\[
\coeff _{C_{S'}}(N \Omega_{X'}\mid_{S'} - \lfloor (N+1)\Omega_{X'}\mid_{S'} \rfloor)=\coeff _{C_{S'}}(-\Omega _{X'}\mid_{S'} + \{ (N+1)\Omega_{X'}\mid_{S'} \}) \geq -\coeff_{C_{S'}}(\Omega_{X'} \mid_{S'}) >-1.
\]
As $G_{S'}$ is integral, by the previous inequality its coefficients cannot be negative.
Therefore, by Step 4, we can lift $G_{S'}$ to an effective divisor $G_{X'} \sim T_{X'}+P_{X'}$ with support not containing $S'$ and such that $G_{X'}\mid_{S'} =G_{S'}$.\\

\noindent\textit{Step 8.}
In this step, 
we introduce a divisor $B^+\geq B$
for which 
$NB^+ \sim -N(K_X+\bM X.)$.\\

Since $NB$ is integral, $\lfloor (N+1)\Omega \rfloor= N\Omega$, where $\Omega$ is the push-forward of $\Omega_{X'}$.
Similarly, we call $T$, $P$, and $G$ the push-forwards of $T_{X'}$, $P_{X'}$, and $G_{X'}$, respectively.
We have that $P=0$ as $P_{X'}$ is exceptional and therefore $T=T+P \sim G$.
Hence, we have that $-N(K_X+B+\bM X.)=T=T+P\sim G \geq 0$.
Therefore $N(K_X+B^+ + \bM X.)\sim 0$, where we define $B^+ \coloneqq  B+\frac{1}{N}G$.
\\

\noindent\textit{Step 9.} In this step, we prove that $(X,B^+, \bM.)$ is generalized log canonical over some neighbourhood of $z$, thus proving that $B^+$ is an $N$-complement for $(X,B,\bM.)$ over $z$. \\

We first prove that $\frac{1}{N}G \mid _S =B^+_S-B_S$.
Note that we have the following chain of $\qq$-linear equivalences:
\[
    R_{X'} \coloneqq G_{X'}-P_{X'} + \lfloor(N+1)\Omega_{X'} \rfloor - N\Omega_{X'} \sim T_{X'}+ \lfloor(N+1)\Omega_{X'} \rfloor - N\Omega_{X'}= -N f^*(K_X+B+\bM X.) \sim_{\qq} 0 / X.
\]
Since $N\Omega$ is integral, we have that $\lfloor(N+1)\Omega \rfloor=N\Omega$. Therefore, as $P_{X'}$ is $f$-exceptional, $f_{*}(R_{X'})=G$ and $R_{X'}$ is the pull-back of $G$.
Observe that
\[
Ng^*(B^+_S-B_S)=G_{S'}-P_{S'} + \lfloor(N+1)\Omega_{X'}\mid_{S'} \rfloor - N\Omega_{X'}\mid_{S'}= (G_{X'}-P_{X'}+\lfloor(N+1)\Omega_{X'} \rfloor - N\Omega_{X'})\mid_{S'}= R_{X'} \mid_{S'}.
\]
Therefore $g^*(B^+_S-B_S)=\frac{1}{N} R_{X'}\mid_{S'}=g^*(\frac{1}{N}G\mid_{S})$, implying that $B^+_S-B_S=\frac{1}{N}G\mid_{S}.$

We now have that $K_S +B_S^++\bN S.= K_S+B_S+B_S^+-B_S+\bN S.=(K_X+B+\frac{1}{N}R+\bM X.)\mid_S=(K_X+B^++\bM X.)\mid_S$.
By inversion of adjunction, $(X,B^+, \bM .)$ is generalized log canonical near $S$.
Moreover, it has coregularity $c$ by~\cite[Lemma 2.30]{FMP22}.

If $(X,B^+, \bM.)$ is not generalized log canonical near the fiber over $z$, then $(X,aB^++(1-a)\Gamma),\bM. )$ is also not generalized log canonical near the fiber over $z$ for $a<1$ close enough to 1. The generalized pair $(X,B^+,\bM.)$ is generalized log canonical near $S$, therefore a component of the generalized non-klt locus of $(X,aB^++(1-a)\Gamma),\bM. )$ is not near $S$. But $S$ is also a component of the generalized log canonical locus. Hence the generalized non-klt locus of $(X,aB^++(1-a)\Gamma),\bM. )$ is disconnected near the fiber over $z$. This is a contradiction as $-(K_X+aB^++(1-a)\Gamma+\bM.)=-a(K_X+B^++\bM.)-(1-a)(K_X+\Gamma+\bM.)$, is big and nef over $Z$, so the connectedness principle can be applied (see, e.g.,~\cite[Lemma 2.14]{Bir19}). Therefore $(X,B^+, \bM.)$ is generalized log canonical near the fiber over $z$. 
\end{proof}

\begin{proof}[Proof of Theorem~\ref{thm:lifting-from-divisor}]

We will proceed in several steps to reduce to Proposition~\ref{prop:lifting-from-divisor}.
Without loss of generality, we may replace $X$ with a small $\qq$-factorial modification.
\\

\noindent \textit{Step 1.} In this step, we define a boundary divisor $B_2 \leq B_1$ and reduce to the case where $-(K_X+B_2+\alpha b \bM X.)$ is big and nef for some $b \in (0,1)$.
\\

For any $0 < a < 1$, we have that 
\[a( K_X+B+\bM X.)+(1-a)(K_X+B_1+\alpha \bM X.)=K_X+(aB+(1-a)B_1)+(a+(1-a)\alpha) \bM X.\]
is anti-big and anti-nef, hence we can replace $B_1$ by $aB+(1-a)B_1$ and $\alpha$ by $a+(1-a)\alpha$, to obtain $B_1$ with coefficients as close as needed to the coefficients of $B$.

Let $B_2 \coloneqq  b B_1$ for some $b <1$.
Since $X$ is of Fano type over $Z$, $-(K_X+B_1+\alpha \bM X. )$ defines a contraction $X \rar V$ over $Z$.
We run an MMP on $-(K_X+B_2+\alpha b \bM X.)$ over $V$. In the resulting model $-(K_{X'}+B_2'+\alpha b \bM X'.)$ is big and nef over $V$.
By the definition of $V$,
$-(K_{X'}+B_2'+\alpha b \bM X'.)(1-t)-(K_{X'}+B_1'+\alpha \bM X'.)t$ is nef over $Z$ for $t$ close enough to 1, which is equivalent to saying that $-(K_{X'}+B_2'+ \alpha b \bM X'.)$ is nef over $Z$ for $b$ close enough to 1.

By taking $a$ close enough to 1, we have that $K_X+B+\bM X.$ is nonnegative over $V$, hence the MMP we ran is $(K_X+B+\bM X.)$-nonnegative. So, by Lemma~\ref{lem:coreg-nonnegative-contraction} the coregularity of $(X,B, \bM X.)$ remains unchanged.

Then, we can replace $(X,B,\bM.)$ with $(X',B', \bM.)$, $B_1$ with $B_1'$ and $B_2$ with $B_2'$ to have also that $-(K_{X'}+B_2'+\alpha b \bM X'.)$ is big and nef over $Z$.
As $-(K_X+B_1+\alpha \bM X.)$ is big and nef over $Z$, we have that there is $A$ ample and $E$ effective, such that 
$-(K_X+B_1+\alpha\bM X.) \sim_{\qq,Z} A+E$.
\\

\noindent \textit{Step 2.} In this step, we separate into cases depending on whether the generalized log canonical centers of $(X,B_1,\alpha\bM.)$ are contained in the support of $E$.\\

We can take a generalized dlt modification of $(X,B_1, \alpha \bM X.)$, so we can assume that $(X,B_1, \alpha \bM X.)$ is generalized dlt.
If $\Supp{E}$ contains no generalized log canonical center of $(X,B_1,\alpha \bM.)$, then $(X,B_1+\varepsilon E,\alpha \bM.)$ is generalized dlt for $\varepsilon>0$ small enough.

We have that $-(K_X+B_1+\varepsilon E+\alpha\bM X.) \sim_{\qq.Z}(1-\varepsilon)( \frac{\varepsilon}{1-\varepsilon}A+-(K_X+B_1+\alpha \bM X.))$ is ample over $Z$. 
Hence by altering the coefficients of $B_1+\varepsilon E$, we can produce a divisor $\Gamma$ that lets us conclude by Proposition~\ref{prop:lifting-from-divisor}.\\

If $\Supp{E}$ does contain some generalized log canonical center of $(X,B_1,\alpha \bM.)$, 
then for $0<r<1$ we define $B_r \coloneqq rB_1+(1-r)B_2$ and $\alpha_r \coloneqq (r+b(1-r))\alpha$.
Then, we define $t_r$ to be the generalized log canonical threshold of $E+B_1-B_r$ with respect to $(X,B_r,\alpha_r \bM .)$ over $z$.
Since $X$ is of Fano type, it is klt.
Furthermore, since $(X,B,\bM.)$ is generalized log canonical and $0 <b,r < 1$, it follows that $(X,B_r,\alpha_r \bM.)$ is generalized klt.
In particular, we have $t_r >0$.
We have 
\begin{align*}
    -(K_X+B_r+t_r(E+B_1-B_r)+ \alpha_r \bM X.)&=-(K_X+B_1+\alpha\bM X.)+B_1-B_r-t(E+B_1-B_r)+ (\alpha-\alpha_r) \bM X .\\
    &\sim_{\rr,Z} A+E +(1-t_r)(B_1-B_r)-tE+(\alpha-\alpha_r) \bM X.\\
    &=t_rA+(1-t_r)(A+E+(B_1-B_r))+ (\alpha-\alpha_r) \bM X.\\
    &\sim_{\rr,Z} t_rA -(1-t_r)(K_X+B_r+\alpha \bM X.)+ (\alpha-\alpha_r)\bM X.\\
    &\sim_{\rr,Z} t_rA -(1-t_r)(K_X+B_r+\alpha_r \bM X.)+ t_r(\alpha-\alpha_r)\bM X.\\
    &=t_r(A+(\alpha-\alpha_r) \bM X.)-(1-t_r)(K_X+B_r+\alpha_r \bM X.).
\end{align*}
If we pick $r$ close enough to 1, then $\alpha-\alpha_r$ tends to $0$, hence $A-(\alpha-\alpha_r) \bM X.$ is ample over $Z$ for $r$ close enough to 1. 
Fixing such an $r$, it follows that $-(K_X+B_r+t_r(E+B_1-B_r)+ \alpha_r \bM X.)$ is ample over $Z$.
\\ 

\noindent\textit{Step 3.} In this step, we separate into cases according to the round-down of the divisor 
$\Theta \coloneqq  B_r+t_r(E+B_1-B_r)$.\\

If we have $\lfloor \Theta \rfloor =0$, then we let $(X',\Theta',\alpha_r \bM .)$ be a dlt modification of $(X,\Theta,\alpha_r \bM .)$. We can assume that every component of $\lfloor \Theta' \rfloor$ intersects the fiber over $z$, after shrinking $Z$.
Furthermore, since $(X,B_r,\alpha_r \bM.)$ is generalized klt, $\lfloor \Theta' \rfloor $ is the exceptional divisor of $X' \rightarrow X$.
An MMP on $K_{X'}+\lfloor \Theta' \rfloor+\alpha b \bM X'.$ over $X$ ends with $X$, as $\lfloor \Theta' \rfloor $ is the exceptional divisor of $X' \rightarrow X$ and $X$ is klt and $\qq$-factorial.
The last step of this MMP would be a divisorial contraction $X''\rightarrow X$ contracting one prime divisor $S''$ with $(X'',S'', \alpha_r\bM X''.)$ generalized plt and $-(K_{X''}+S''+ \alpha_r\bM X''.)$ ample over $X$. Furthermore $S''$ is a component of both $\Theta''$ and $B_1''$, where $K_{X''}+\Theta''$ and $K_{X''}+B_1''$ are the pull-backs of $K_X+\Theta$ and $K_X+B_1$ respectively.

As $-(K_X+\Theta+\alpha_r \bM X.)$ is ample over $Z$ and $-(K_{X''}+S''+\alpha_r \bM X''.)$ is ample over $X$, a linear combination of $S''$ and $\Theta''$ yields $\Gamma''$ such that  $-(K_{X''}+\Gamma''+\alpha_r \bM X''.)$ is ample over $Z$ and $(X'',\Gamma'',\alpha_r \bM X''.)$ is plt with $\lfloor \Gamma'' \rfloor =S''$. We can apply Proposition~\ref{prop:lifting-from-divisor} here.
As an $N$-complement on $X''$ would induce an $N$-complement on $X$, we have reduced to the case where $\lfloor \Theta \rfloor \neq 0$.

It only remains to deal with the case where $\lfloor \Theta \rfloor \neq 0$.
In this case, there is a component $S$ of $\lfloor \Theta \rfloor \leq \lfloor B_1 \rfloor \leq \lfloor B \rfloor$.
On a dlt modification of $(X,S, \alpha_r \bM X.)$, we can perturb the coefficients of $\Theta'$ to get that $\lfloor \Theta'' \rfloor$ is irreducible and $-(K_{X'} +\Gamma+ \alpha_r \bM X.)$ is ample.
Therefore we can obtain the $N$-complement by Proposition~\ref{prop:lifting-from-divisor}, as desired. As an $N$-complement on $X''$ induces an $N$-complement on $X$, we are done.
\end{proof}

\subsection{Relative complements}

In this subsection, we study the existence of complements
in the relative setting. 
The main theorem of this subsection states that 
we can lift complements 
from lower coregularity pairs
when all the generalized log canonical centers
are horizontal over the base.

\begin{theorem}\label{thm:horizontal-relative-comp}
Let $d$, $c$ and $p$ be nonnegative integers and $\Lambda \subset \qq$ be a closed set satisfying the DCC. Assume Theorem~\ref{introthm:Fano-coreg-c}$(c-1)$ holds. There exists a constant $N \coloneqq N(\Lambda,c,p)$ satisfying the following.

Let $\pi \colon  X \rightarrow Z$ be a Fano type morphism, where $X$ is a $d$-dimensional variety, and $\dim Z >0$. Let $(X,B, \bM .)$ be a generalized log canonical pair over $Z$ and $z \in Z$ a point satisfying the following conditions:
\begin{itemize}
    \item the generalized pair $(X,B, \bM.)$ has coregularity at most $c$ over $z$;
    \item the coefficients of $B$ belong to $\Lambda$;
    \item every generalized log canonical center of $(X,B,\bM.)$ dominates $Z$;
    \item $p \bM.$ is b-Cartier; and
    \item the divisor $-(K_X+B+\bM X.)$ is nef over $Z$.
\end{itemize}

Then, there exists an $N$-complement for $(X,B,\bM .)$ over $z$. 

\end{theorem}

\begin{proof}
We will proceed in several steps to be able to lift complements using Theorem~\ref{thm:lifting-from-divisor}.\\

\noindent \textit{Step 1.} In this step, we reduce to the case in which the boundary coefficients belong to a finite set, $(X,B, \bM.)$ has coregularity $c-1$ over, and $\lfloor B \rfloor$ has a vertical component intersecting the fiber over $z.$\\

By Theorem~\ref{thm:reduction-of-coefficients}, there exists a finite set $\mathcal{R}$ and a relative pair $(X',B',\bM .)$ with $\coeff(B')\subset \mathcal{R}$, such that if $(X',B',\bM .)$ is $N$-complemented, then so is $(X,B,\bM.)$. So, we can replace our generalized pair $(X,B,\bM.)$ with $(X',B',\bM .)$. Hence, we may assume the coefficients of $B$ belong to the finite set $\mathcal{R}$.
Up to taking a $\qq$-factorial dlt modification, we may further assume that $X$ is $\qq$-factorial.

We pick an effective Cartier divisor $N$ on $Z$ passing through $z$. 
We let $t$ be the generalized log canonical threshold of $q^*N$ with respect to $(X,B, \bM.)$ over $z$.
By the connectedness principle~\cite[Theorem 1.7]{FS23} and the assumption that all the generalized log canonical centers of $(X,B,\bM.)$ dominate $Z$, 
we know that the coregularity of $(X,B+tq^*N,\bM.)$ is at most $c-1$.  

Let  $(X',T', \bM.)$ be a generalized dlt modification of 
$(X,B+tq^*N,\bM.)$ over $z$. Let $B'$ be the strict transform of $B$ on $X'$. Let $\Omega'$ be a boundary such that $B' \leq \Omega' \leq T'$, $\coeff(\Omega')\subset \mathcal{R} $, and some component $S$ of $\lfloor \Omega' \rfloor $ is vertical over $Z$ intersecting the fiber $\pi^{-1}(z)$. Let $\Omega = \pi_\ast \Omega'$.

We run an MMP over $Z$ on $-(K_{X'}+\Omega'+\bM X'.)$. As $-(K_{X'}+\Omega'+\bM X'.)=-(K_{X'}+T'+\bM X'.)+(T'-\Omega'),$ with $-(K_{X'}+T'+\bM'.)$ nef over $Z$ and $(T'-\Omega')$ effective, the MMP ends with a minimal model, which we denote by $X''$.

If $(X'',\Omega'', \bM X''.)$ has an $N$-complement over $z$, then $(X',\Omega', \bM X'.)$ has an $N$-complement over $z$ by  Lemma~\ref{lem:complements-and-K-positive}.
As $B \leq \Omega$, we have that also $(X,B, \bM X .)$ has an $N$-complement by Lemma~\ref{lem:complements-and-dlt-mod}. So, we can replace $(X,B, \bM X .)$ with $(X'',\Omega'', \bM X''.)$, to obtain $\lfloor B \rfloor$  having a component intersecting the fiber over $z$, with $-(K_X+B+\bM X.)$ nef over $Z$.
Notice that, after this reduction, the coregularity has decreased, and it is no longer the case that all generalized log canonical centers dominate $Z$.\\

\noindent \textit{Step 2.}
In this step, we define a divisor $B_1$ which satisfies the hypothesis of Theorem~\ref{thm:lifting-from-divisor}.\\

For any prime divisor $D$ that vertical is over $Z$, we set $\coeff_D(B_1)\coloneqq\coeff_D(B)$. For any prime divisor $D$ horizontal over $Z$, we set $\coeff_D(B_1)\coloneqq\coeff_D(aB)$ for $a<1$, close enough to 1.

As $X$ is of Fano type over $Z$, we have that $-K_X$ is big over $Z$. Therefore
$-(K_X+aB+a \bM X.)=-a(K_X+B+\bM X.)-(1-a)K_X$ and $-(K_X+ B_1+a\bM X.)$ are big over $Z$. The generalized pair $(X,B_1, a\bM.)$ is generalized log canonical as $B_1 \leq B$ and $\lfloor B_1 \rfloor$ contains the same vertical component as $\lfloor B \rfloor$ intersecting the fiber of $z$.\\

\noindent\textit{Step 3.}
In this step, we reduce to the case in which $-(K_X+B_1+\bM X.)$ is big and nef over $Z$.\\

Since $\pi$ is a Fano type morphism and $-(K_X+B+\bM X.)$ is nef over $Z$, $-(K_X+B+\bM X.)$ is semi-ample over $Z$. Let $X \rightarrow V$ over $Z$ be the contraction defined by $-(K_X+B+\bM X.)$. We run an MMP on $-(K_X+B_1+a\bM X.)$ over $V$. In the resulting model  $-(K_X'+B_1'+a\bM X' .)$ is big and nef over $V$. By the definition  of $V$, $-(K_{X'}+B_1'+a\bM X' .)(1-t)-(K_{X'}+B'+\bM X'.)t$ is nef over $Z$ for t close enough to 1, which is equivalent to picking $a$ close enough to 1. We have that $K_X+B+\bM X.$ is trivial  over $V$, hence the MMP is $(K_X+B+\bM X.)$-nonnegative.
Applying Lemma~\ref{lem:coreg-nonnegative-contraction}, the coregularity of $(X,B,\bM .)$ remains unchanged after this MMP.
Thus, we can replace $(X,B,\bM .)$, with $(X',B', \bM.)$ and $B_1$ with $B_1'$ with $a$ close enough to 1.\\

\noindent\textit{Step 4.} In this step, we conclude by applying Theorem~\ref{thm:lifting-from-divisor}.\\

Let $N(\mathcal{R},c-1,p)$ be the positive integer provided by Theorem~\ref{introthm:Fano-coreg-c}$(c-1)$.
Taking $N$ to be the least common multiple of $N(\mathcal{R},c-1,p)$, $I_\mathcal{R}$ and $p$. Hence $N$ depends only on $\Lambda$, $c-1$ and $p$.
Indeed, $\mathcal{R}$ only depends on $\Lambda,c$ and $p$.
By Theorem~\ref{thm:lifting-from-divisor} the generalized pair $(X,B, \bM.)$ admits an $N$-complement,
where $N$ only depends on
$\Lambda$, $c$ and $p$.
\end{proof}

\section{Canonical bundle formula}
\label{sec:cbf}

In this section, we prove a special version of the canonical bundle formula.
We obtain an effective canonical bundle formula that is independent of the dimension of the domain. 
It only depends on the coregularity of the fibers.

\begin{theorem}\label{thm:cbf-and-coreg-induction}
Let $d$, $c$ be nonnegative integers
and $\Lambda\subset \qq$ be a closed set satisfying the descending chain condition.
Assume Theorem~\ref{introthm:Fano-coreg-c}$(c-1)$ holds. 
There exists a set 
$\Omega \coloneqq \Omega(\Lambda,c)\subset \qq$ satisfying the descending chain condition
and a positive integer $q \coloneqq q(\Lambda,c)$ satisfying the following.
Let $\pi \colon X\rightarrow Z$ be a fibration from a $d$-dimensional projective variety $X$ to a projective base $Z$ with $\dim Z > 0$.
Let $(X,B)$ be a log canonical pair satisfying the following conditions:
\begin{itemize}
    \item the fibration $\pi$ is of Fano type over a nonempty open set $U$ of $Z$;
    \item every log canonical center of $(X,B)$ is horizontal over $Z$;
    \item the pair $(X,B)$ is log Calabi--Yau over $Z$;
    \item the coefficients of $B$ are in $\Lambda$; and 
    \item the coregularity of $(X,B)$ is at most $c$.
\end{itemize}
Then, we can write 
\[
q(K_X+B)\sim q\pi^*(K_Z+B_Z+\bN Z.),
\] 
where $(Z,B_Z,\bN.)$ is a generalized log canonical pair such that 
\begin{itemize}
    \item $B_Z$ is the discriminant part of the adjunction for $(X,B)$ over $Z$;
    \item the coefficients of $B_Z$ belong to $\Omega$; and
    \item the divisor $q\bN.$ is b-nef and b-Cartier.
\end{itemize}
\end{theorem}

\begin{proof}
The proof is given in several steps. In Step 1, we make a choice of $q$ and $\bN.$.
In Steps 2-4, we find $\Omega$ and show that the coefficients of $B_Z$ belong to $\Omega$. We also prove that $q\bN Z.$ is integral. In Step 5, we show that $q\bN .$ is b-Cartier.
We observe that, by the assumptions on the log canonical centers of $(X,B)$ and on the coregularity of $(X,B)$, it follows that $\dim Z \leq c$.\\

\noindent\textit{Step 1.} In this step, we find $q$ and make a choice of $\bN.$.\\

Let $q$ be the integer $N$ in the statement of Theorem~\ref{thm:horizontal-relative-comp}$(d,c-1)$. Here, we assumed Theorem~\ref{introthm:Fano-coreg-c}$(c-1)$. Fix a general closed point $z\in U$. Let $H$ be a general hyperplane section of $U$ passing through $z$. Then $(X,B+\pi^*H)$ is log Calabi--Yau and satisfies
\[ \text{coreg}(X,B+\pi^*H) \leq c-1.
\]
This implies that the absolute coregularity of $(X,B)$ over $z$ is at most $c-1$.
By Theorem~\ref{thm:horizontal-relative-comp}$(d,c-1)$, there is a $q$-complement $K_X+B^+$ of $K_X+B$ over $z$ with $B^+\geq B$.
Note that $q$ only depends on $\Lambda$ and $c$.
Since $K_X+B$ is $\qq$-trivial over $Z$, $B^+-B\sim_{\qq}0$ over $z$ and hence $B^+=B$ near the generic fiber of $\pi$. Therefore, $q(K_X+B)\sim 0$ over the generic point of $Z$. Thus, we can find a rational function $s$ on $X$ such that $qL \coloneqq  q(K_X+B) + \text{Div}(s)$ is zero over the generic point of $Z$. Note that $L\sim_{\qq} 0$ over $Z$, so we can write $L = \pi^* L_Z$ for some $\qq$-Cartier $\qq$-divisor $L_Z$ on $Z$.  Define
\[
\bN Z.  \coloneqq  L_Z-(K_Z+B_Z),
\]
where $B_Z$ is the discriminant part of adjunction for $(X,B)$ over $Z$. Similarly, for any birational morphism $g \colon  Z'\to Z$, we can define ${\bN.}_{Z'}$ as follows. Let $f \colon X'\to X$ be a higher birational model of $X$ such that the rational map $X'\dashrightarrow Z'$ is a morphism. Write $K_{X'} + B_{X'} $ for the pull-back of $K_X+B$ and $B_{Z'}$ be the discriminant part of adjunction for $(X', B_{X'})$ over $Z'$. We define 
\[
{\bN.}_{Z'}  \coloneqq  g^*L_Z - (K_{Z'} + B_{Z'}).
\]
The data of $\bN Z'.$, for all birational models $Z'\to Z$, determines
a b-divisor $\bN .$ on $Z$.\\ 

\noindent\textit{Step 2.} In this step, we reduce to the case when the base $Z$ is a curve.\\ 

Assume $\dim Z \geq 2$. Let $H$ be a general hyperplane section of $Z$ and $G$ be the pull-back of $H$ to $X$. By adjunction, we can write
\[
(K_X+B+G)|_G = K_G + B_G 
\]
for some divisor $B_G$ on $G$.
By Lemma~\ref{lemma-coeff-adj}, there exists a set $\Lambda'$ satisfying the DCC, having rational accumulation points, and depending only on $\Lambda$, such that the coefficients of $B_G$ belong to $\Lambda'$.
We may replace $\Lambda$ with $\Lambda'$ to assume that the coefficients of $B_G$ belong to $\Lambda$.
Since $\dim Z \leq c$, this replacement can only happen at most $c-1$ times and is hence allowed.
By Lemma~\ref{lem:hyperplane-section-inductive-step}, the pair $(G,B_G)$ over $H$ satisfies the same conditions as $(X,B)$ over $Z$ in the statement of this theorem. Furthermore, let $B_H$ denote the discriminant part of the adjunction for $(G,B_G)$ over $H$. Let $D$ be any prime divisor on $Z$ and $C$ a component of $D\cap H$. Then
\[
\coeff_D(B_Z) = \coeff_C(B_H).
\]

Pick a general $H'\sim H$ and let $K_H = (K_Z+H')|_H$, which is properly defined as a Weil divisor. Define
\[
\bN H.' = (L_Z+H')|_H - (K_H+B_H).
\]
Similarly to the way we defined the b-divisor $\bN.$ on $Z$, we can define $\bN.'$ as a b-divisor on $H$. Then, we have
\[
q(K_G+B_G) \sim q(L+G)|_G \sim q\psi^*(L_Z+H')|_H \sim q\psi^*(K_H+B_H+\bN H.').
\]
Thus, $\bN H.'$ is the moduli part of $(G,B_G)$ over $H$. Moreover, we have $B_H+{\bN H.'} = (B_Z+\bN Z.)|_H$. This implies that
$
\coeff_C(B_H+{\bN H.'} ) = \coeff_D(B_Z + \bN Z.)
$
and hence
$
\coeff_C({\bN H.'} ) = \coeff_D(\bN Z.).
$
In particular, $q\bN H.'$ is integral if and only if $q\bN H.$ is integral.

As a result, to prove that $\coeff(B_Z)$ belongs to a fixed set $\Omega$ and that $q\bN Z.$ is integral, we may replace $(X,B) \to Z$ with $(G,B_G)\to H$. By repeating this process until the base of the fibration is a curve, we may assume that $\dim Z = 1$.
\\

\noindent\textit{Step 3.} In this step, we show the existence of $\Omega$.\\

By Step 2, we can assume $\dim Z = 1$. By~\cite[Lemma 2.11]{Bir19}, the variety $X$ is of Fano type over $Z$. Pick any closed point $z\in Z$. Let $t$ be the log canonical threshold of $\pi^*z$ with respect to $(X,B)$ around $z$. Set $\Gamma = B+t\pi^*z$ and let $(X',\Gamma')$ be a $\qq$-factorial dlt modification of $(X,\Gamma)$.
Then $K_{X'}+\Gamma' \sim_{\qq} 0$ over $Z$ and $\Gamma'$ has a component with coefficient 1 mapping to $z$. Pick a boundary $B'$ on $X'$ satisfying the following conditions:
\begin{itemize}
    \item we have  $\Tilde{B}\leq B'\leq \Gamma'$, where $\Tilde{B}$ is the strict transform of $B$ on $X'$;
    \item the coefficients of $B'$ are in $\Lambda$; and
    \item the divisors $B'$ and $\Gamma'$ have the same round-down, i.e.,
    $\lfloor B'\rfloor = \lfloor \Gamma'\rfloor$.
\end{itemize}
By construction, $\supp(\Gamma'-B')$ is contained in the strict transform of $\supp(\pi^*z)$.   If $t<1$, $\lfloor B'\rfloor$ has a component $T$ mapping to $z$ which is exceptional over $X$.
Note that $X'$ is of Fano type over $Z$. We run a $-(K_{X'}+B')$-MMP with scaling over $Z$.
Since $-(K_{X'}+B')\sim_{\qq} \Gamma'-B'$ is pseudo-effective, this MMP terminates with a model $(X'',B'')$ such that $-(K_{X''}+B'')$ is semi-ample over $Z$, by \cite[Theorem 1.1]{Bir12}.
The divisor $T$ is not contracted by this MMP because $T\not\subseteq\supp(\Gamma'-B')$. Furthermore, by Lemma \ref{lem:coreg-nonnegative-contraction}, we have
\[
{\rm coreg}(X'',B'') = {\rm coreg}(X',B')\leq {\rm coreg}(X',\Gamma') \leq c.
\]
As a result, up to losing the dlt condition, we may replace $(X',B')$ by $(X'',B'')$ to assume that the following properties hold:
\begin{itemize}
    \item $X'$ is Fano type over $Z$;
    \item $(X',B')$ is a log canonical pair over $Z$;
    \item the coefficients of $B'$ are in $\Lambda$;
    \item ${\rm coreg}(X',B')\leq c$;
    \item $-(K_{X'}+B')$ is semi-ample over $Z$; and
    \item $\lfloor B'\rfloor$ has a component mapping to $z$.
\end{itemize}
If $t = 1$, we may simply take $B' = \Gamma'$ so the above conditions hold as well for $(X',B')$.

By Theorem~\ref{thm:lifting-from-divisor} and our choice of $q$,
the pair $(X',B')$ has a $q$-complement $(X',B'^+)$ over $z$ with $B'^+\geq B'$. Pushing forward $B'^+$ to $X$ gives a $q$-complement $(X,B^+)$ of $(X,B)$ over $z$ with $B^+\geq B$. Furthermore, $K_X+B^+$ has a non-klt center mapping to $z$, since its pull-back $K_{X'}+B'^+$ does. Note that $B^+-B\sim_{\qq} 0$ over $z$, so $B^+-B$ must be a multiple of $\pi^*z$ over $z$. This implies that $B^+ = B+t\pi^*z$ over $z$ since $K_X+B^+$ has a non-klt center mapping to $z$.

Pick a component $S$ of $\pi^*z$. Let $b$, $b^+$, and $m$ be the coefficients of $S$ in $B$, $B^+$, and $\pi^*z$ respectively. Then $b^+ = b+tm$ and $t = \frac{b^+-b}{m}$. By construction, $qb^+$ and $m$ are integers, $b^+\leq 1$, and $b\in \Lambda$. By Lemma~\ref{lem:coeff-dcc-and-rational-acc-points}, $t$ belongs to a fixed set $\Sigma$ (depending only on $I,\Lambda$) which satisfies ACC and has rational accumulation points. Since the coefficient of $z$ in $B_Z$ is $1-t$, it belongs to a set $\Omega$ depending only on $I$ and $\Lambda$, such that $\Omega$ satisfies DCC and has rational accumulation points.\\

\textit{Step 4.}
In this step, we show that $q\bN Z.$ is integral.\\ 

By Step 2, we can assume that $Z$ is a curve. By Step 1, the equivalence $q(K_X+B)\sim 0$ holds over some non-empty open subset $V\subseteq Z$ such that $\Supp B_Z \subseteq Z\setminus V$. Let
\[
\Theta = B + \sum_{z\in Z\setminus V} t_z \pi^*z,
\]
where $t_z$ is the log canonical threshold of $\pi^*z$ with respect to $(X,B)$ over $z$. Note that $K_X+\Theta\sim_{\qq} 0$ over $Z$. Let $\Theta_Z$ be the discriminant part of adjunction for $(X,\Theta)$ over $Z$. Then
\[
\Theta_Z = B_Z + \sum_{z\in Z\setminus V} t_z z = \sum_{z\in Z\setminus V} z
\]
is an integral divisor. Moreover, by Step 3, the pair $(X,\Theta)$ is a $q$-complement of $(X,B)$ over each $z\in Z\setminus V$. Over $V$, we have $q(K_X+\Theta) \sim q(K_X+B)\sim 0$. Thus, $q(K_X+\Theta) \sim 0$ over $Z$ by \cite[Lemma 2.4]{Bir19}. Furthermore, from the equalities
\[
    q(K_X+\Theta) = q(K_X+B) + q(\Theta-B) 
    \sim q\pi^*(K_Z+B_Z+\bN Z.) + q\pi^*(\Theta_Z-B_Z) 
    \sim q\pi^*(K_Z+\Theta_Z+\bN Z.),
\]
we obtain that $q(K_Z+\Theta_Z+\bN Z.)$ is an integral divisor and hence $q\bN Z.$ is integral.\\

\noindent\textit{Step 5.} In this step, we show that $q\bN Z'.$ is nef Cartier on some resolution $Z'\to Z$. \\

The nefness follows from \cite[Theorem 3.6]{Bir19}, so we just need to show that $q\bN Z'.$ is integral. Denote the birational morphism $Z'\to Z$ by $g$. As in Step 1, let $f \colon X'\to (X,B)$ be a log resolution such that the rational map $\pi' \colon X'\dashrightarrow Z'$ is a morphism. Let $U_0\subseteq U$ be a non-empty open set such that $U_0' \coloneqq  g^{-1}(U_0) \to U_0$ is an isomorphism. Let $\Delta'$ be the sum of the birational transform of $B$ and reduced exceptional divisors of $f$, but with all components mapping outside of $U_0$ removed. Then, the generic point of every log canonical center of $(X',\Delta')$ lies inside $U_0$.

Let $T$ be the normalization of the main component of $Z'\times_Z X$. Run an MMP on $K_{X'}+\Delta'$ over $T$ with scaling of some ample divisor. This MMP terminates with a $\qq$-factorial dlt pair $(X'',\Delta'')$ such that $K_{X''}+\Delta''$ is nef over $T$. Let $X_0 = \pi^{-1}(U_0)$. Over $U_0\cong U_0'$, we have $Z'\times_Z X\cong X_0$, and hence $K_{X''}+\Delta''$ is nef over $X_0\subseteq X$. By the negativity lemma, over $U_0'$,
the divisor $K_{X''}+\Delta''$ is equal to the log pull-back of $K_X+B$. This shows that
$(X'',\Delta'')$ is a dlt modification of $(X,B)$ over $U_0$. In particular, $X''$ is Fano type over $U_0'$. Furthermore, every log canonical center of $(X'',\Delta'')$ dominates $U_0'$. 
Indeed, $(X,B)$ satisfies the same property and the generic point of every log canonical center of $(X'',\Delta'')$ lies inside $U_0'$. Thus, we have
\[
\text{coreg}(X'', \Delta'' ) = \text{coreg} (X,B) \leq c.
\]
By \cite[Theorem 1.4]{Bir12}, we can run an MMP on $K_{X''}+\Delta'$ over $Z'$ which terminates with a good minimal model over $Z'$. Since this MMP is trivial over $U_0'$, the dual complex and hence the coregularity of $(X'',\Delta'')$ does not change under this MMP. Abusing the notation, we again denote the good minimal model by $X''$. Let $\pi'' \colon  X''\to Z''$ over $Z'$ be the morphism induced by the relatively semi-ample divisor $K_{X''}+\Delta''$. Since $K_{X''}+\Delta''\sim_{\qq} 0$ over $U_0'$, $Z''\to Z'$ is birational and $U_0''$, the preimage of $U_0'$ in $Z''$, is isomorphic to $U_0'$. Thus, every log canonical center of $(X'',\Delta'')$ also dominates $Z''$.

Let $W$ be a common resolution of $X$ and $X''$ and set $\alpha \colon  W\to X$ and $\beta \colon W\to X''$. By construction, over the preimage of $U_0$ in $W$,
we have 
$\alpha^*(K_X+B) = \beta^*(K_{X''}+\Delta'')$.
Write $K_{X''}+B'' = \beta_*\alpha^*(K_X+B)$
and $L'' = \beta_*\alpha^*L$,
where $qL = q(K_X+B) + \text{Div}(s)$ as in Step 1. Let $P'' = \Delta''-B''$, which is $\qq$-trivial over $Z''$ and supported outside of $U_0''$. Hence, $P'' = \pi''^*P_{Z''}$ for some $\qq$-divisor $P_{Z''}$ on $Z''$. Let $\Delta_{Z''}$ be the discriminant part of adjunction for $(X'',\Delta'')$ on $Z''$ . Let $B_{Z''} = \Delta_{Z''} - P_{Z''}$, then $B_{Z''}$ is the discriminant part of adjunction for $(X'',B'')$ on $Z''$. By the definition of $\bN.$ in Step 1, we have
\[
    q(K_{X''}+\Delta'') = q(K_{X''}+B''+P'') \sim 
    q\pi''^*(K_{Z''}+ B_{Z''} + \bN Z''. + P_{Z''}) \sim 
    q\pi''^*(K_{Z''}+\Delta_{Z''}+\bN Z''.).
\]
Here, $L_{Z''}$ is the pull-back of $L_Z$ in Step 1. This shows that $\bN Z''.$ is the moduli part of $(X'',\Delta'')$ over $Z''$. Since the coefficients of $\Delta''$ are in $\Lambda$ and the coregularity of $(X'',\Delta'')$ is at most $c$, we may apply Steps 2-4 to show that $q \bN Z''.$ is an integral divisor. Thus, $q\bN Z'.$ is also an integral divisor and hence Cartier as $Z'$ is smooth.
\end{proof}

We show that Theorem \ref{thm:cbf-and-coreg-induction} also holds for generalized pairs $(X,B,\bM.)$ in the special case that $\bM X.\sim_{\qq} 0$ over the base $Z$.

\begin{theorem}\label{thm:cbf-inductive}
Let $d$, $c$ and $p$ be nonnegative integers
and $\Lambda\subset \qq$ be a closed set satisfying the descending chain condition.
Assume Theorem~\ref{introthm:Fano-coreg-c}$(c-1)$ holds. 
There exists a set 
$\Omega \coloneqq \Omega(\Lambda,c,p)\subset \qq$ satisfying the descending chain condition
and a positive integer $q \coloneqq q(\Lambda,c,p)$, both $\Omega$ and $q$ only depending on $\Lambda$ and $c$, and satisfying the following.
Let $\pi \colon X\rightarrow Z$ be fibration from a $d$-dimensional projective variety $X$ to a projective base $Z$ with $\dim Z > 0$.
Let $(X,B,\bM.)$ be a generalized pair for which
\begin{itemize}
    \item the generalized pair $(X,B,\bM.)$ is generalized log canonical;
    \item the fibration $\pi$ is of Fano type over a nonempty open set $U$ of $Z$;
    \item every generalized log canonical center of $(X,B,\bM.)$ is horizontal over $Z$; 
    \item the divisors $K_X+B+\bM X.$ and $\bM X.$ are $\qq$-trivial over $Z$;
    \item the coefficients of $B$ are in $\Lambda$,
    \item $p\bM.$ is b-Cartier; and 
    \item the coregularity of $(X,B)$ is at most $c$.
\end{itemize}
Then, we can write 
\[
q(K_X+B+\bM X. )\sim q\pi^*(K_Z+B_Z+\bN Z.),
\] 
where $(Z,B_Z,\bN.)$ is a generalized log canonical pair such that 
\begin{itemize}
    \item $B_Z$ is the discriminant part of the adjunction for $(X,B,\bM.)$ over $Z$;
    \item the coefficients of $B_Z$ belong to $\Omega$; and
    \item the divisor $q\bN.$ is b-nef and b-Cartier.
\end{itemize}
\end{theorem}

\begin{proof}
We proceed in several steps.
In the first step, we apply the canonical bundle formula for pairs.
In the rest of the proof, we show that $\bM.$ is the pull-back of a b-nef divisor on the base 
and control the Cartier index of this b-nef divisor where it descends.\\

\noindent\textit{Step 1.} We show that the pair $(X,B)$ satisfies the conditions in the statement of Theorem \ref{thm:cbf-and-coreg-induction}. \\ 

By Lemma~\ref{lem:glc-implies-lc}, $(X,B)$ is log canonical. Furthermore, every log canonical center of $(X,B)$ is also a log canonical center of $(X,B,\bM.)$, and hence it dominates $U$. 

Thus, by Theorem \ref{thm:cbf-and-coreg-induction}, there exists $\Omega$ and $I$, depending only on $\Lambda$ and $c$, such that 
\[
q(K_X+B) \sim q \pi^*(K_Z+B_Z+\bP Z.),
\]
where $B_Z$ and $\bP.$ are the discriminant and moduli part of adjunction for the fibration $(X,B) \to Z$. Since $(X,B)$ is log canonical, $(Z,B_Z,\bP.)$ is generalized log canonical.  Furthermore, the coefficients of $B_Z$ belong to $\Omega$ and $q\bP.$ is nef Cartier on any high resolution of $Z$. By replacing $q$ with $pq$, we can assume that $q$ is a multiple of $p$.\\

\noindent\textit{Step 2.} In this step, we express $\bM.$ as a pull-back of some b-divisor from $Z$.\\

Let $g \colon  Z'\to Z$ be a log resolution of $(Z,B_Z+\bP.)$ such that $\bP.$ descends on $Z'$ and $\bP Z'.$ is nef. Let $f \colon  X'\to X$ be a  resolution such that the rational map $\pi' \colon  X'\dashrightarrow Z'$ is a morphism and $\bM X'.$ is nef. By the negativity lemma, we can write
\[
f^* \bM X. = \bM X' . +E_{X'}
\]
for some effective $f$-exceptional $\qq$-divisor $E_{X'}$. Since $\bM X. \sim_{\qq} 0$ over $Z$, $E_{X'}$ is vertical over $Z$, so there is a non-empty subset of $U$ over which $E_{X'}=0$ and $\bM X'.\sim_{\qq} 0$. Since $\pi$ is of Fano type over $U$, the general fibers of $\pi$ are of Fano type and hence rationally connected.
Thus, the general fibers of $\pi'$ are also rationally connected. By \cite[Lemma 2.44]{Bir19}, after replacing $X'$ and $Z'$ with possibly higher birational models, we can write
\[
p\bM X'. \sim p \pi'^* T_{Z'}
\]
for some $\qq$-divisor $T_{Z'}$ on $Z'$ such that $pT_{Z'}$ is nef Cartier. Let $\bT.$ be the b-divisor on $Z$ with the data $g \colon  Z'\to Z$ and $T_{Z'}$ (i.e., $\bT.$ descends on $Z'$ as $T_{Z'}$).\\

\noindent\textit{Step 3.} In this step, we show that $q\bM X.   \sim q\pi^* \bT Z.$. \\

As in Step 2, write
\[
f^*\bM X. = \bM X'. + E_{X'}.
\]
Then $E_{X'}$ is vertical and $\qq$-linearly trivial over $Z'$ (since $\bM X'. \sim_{\qq} 0$ over $Z'$), so we can write $E_{X'} = \pi'^* E_{Z'}$ for some effective $\qq$-divisor $E_{Z'}$. If $E_{Z'}$ has a component $D_{Z'}$ which maps onto a divisor $D$ in $Z$, then $E_{X'} = \pi'^*E_{Z'}$ has a component mapping onto $D$, contradicting the fact that $E_{X'}$ is $f$-exceptional. Thus, $E_{Z'}$ is $g$-exceptional. Note that $\pi'^*(\bT Z'. + E_{Z'}) = \bM X'.+E_{X'} \sim_{\qq} 0$ over $Z$, so $\bT Z'. +E_{Z'} \sim_{\qq}0$ over $Z$. This implies that $g^*\bT Z. = \bT Z'. + E_{Z'}$. Now, we have
\[
qf^*\bM X. = q(\bM X'. + E_{X'}) \sim q \pi'^*(\bT Z'.+ E_{Z'}) = q \pi'^* g^*\bT Z. = q f^*\pi^*\bT Z.
\]
and hence $q\bM X.\sim q\pi^*\bT Z.$. In particular, $q\bT.$ is a b-Cartier divisor.
From now on, we consider $(Z, B_Z,\bP. + \bT.)$ as a generalized pair, with moduli part $\bP.+\bT.$.\\

\noindent\textit{Step 4.} In this step, we show that the generalized pair $(Z,B_Z,\bP.+\bT.)$ is generalized log canonical. \\

Write $K_{X'}+B_{X'} = f^*(K_X+B)$. Let $B_{Z'}$ be the discriminant part of the adjunction for $(X',B_{X'})$ over $Z'$. We can assume that $(Z', \Supp B_{Z'}+\Supp E_{Z'}+\Supp \bP Z'.+\Supp\bT Z'.)$ is log smooth. By construction, we have
\[
K_{Z'} + B_{Z'} + E_{Z'} + \bP Z'. + \bT Z' .= g^*(K_Z+B_Z+\bP Z.+\bT Z.).
\]
Thus, it suffices to show that every coefficient of $B_{Z'} + E_{Z'}$ is at most one. Let $D$ be a prime divisor on $Z'$. Let $t_D$ be the log canonical threshold of $(X',B_{X'})$ with respect to $\pi'^*D$ over the generic point of $D$. Since $(X,B,\bM.)$ is generalized log canonical, $(X',B_{X'}+E_{X'})$ is sub log canonical. Furthermore, $E_{X'} = \pi'^*E_{Z'}$, so the coefficient of $D$ in $E_{Z'}$ is at most $t_D$ (otherwise $E_{X'}\geq (t_D+\epsilon)\pi'^*D$ for some $\epsilon > 0$ and this violates the sub-log canonical condition). By definition, $\coeff_D(B_{Z'}) = 1-t_D$. Thus, $\coeff_D(B_{Z'}+E_{Z'}) \leq 1-t_D + t_D = 1$, as desired. \\

\noindent\textit{Step 5.}
In this step, we conclude that the generalized pair $(Z, B_Z, \bP.+\bT.)$ satisfies the desired properties. \\

By Step 1, the coefficients of $B_Z$ belong to $\Omega$ and $q\bP.$ is b-nef and b-Cartier.
By Steps 2 and 3, $q\bT.$ is b-Cartier. By Step 4, $(Z,B_Z, \bP.+\bT.)$ is generalized log canonical.
Finally, by Steps 1 and 3, we have
\[
q(K_X+B+\bM X.)\sim q\pi^*(K_Z+B_Z+\bP Z. + \bT Z.).
\]
\end{proof}

\begin{proposition}\label{prop:complements-imply-cbf}
Assume that Theorem~\ref{introthm:Fano-coreg-c}$(c-1)$ holds.
Then, Theorem~\ref{introthm:cbf-and-coreg}$(c)$ holds. 
\end{proposition}

\begin{proof}
Theorem~\ref{introthm:Fano-coreg-c}$(c-1)$ implies that Theorem~\ref{thm:cbf-inductive}$(c)$ holds. 

Let $(X,B,\bM.)$ be a generalized pair in the statement of Theorem~\ref{introthm:cbf-and-coreg}$(c)$. We may replace $(X,B,\bM.)$ by a $\qq$-factorial generalized dlt modification and assume that $(X,B,\bM.)$ is $\qq$-factorial generalized dlt. Since $X$ is of Fano type over $Z$, we can run an MMP on $\bM X.$ over $Z$ to get a model $(X',B',\bM.)$ such that $\bM X'.$ is semi-ample over $Z$. After replacing $(X,B,\bM.)$ with $(X',B',\bM.)$, up to losing the generalized dlt property for $(X,B,\bM.)$, we may assume that $\bM X.$ is semi-ample over $Z$. Let $X\to Z'$ be the morphism induced by $\bM.$. Since $\bM.$ is trivial on a general fiber of $\pi$, the morphism $Z'\to Z$ is birational. After replacing $Z$ with $Z'$, we may assume that $\bM.\sim_{\qq}0$ over $Z$. Now, the result follows from Theorem~\ref{thm:cbf-inductive}$(c)$.
\end{proof}

\section{Proof of the theorems}

In this section, we prove the main theorems of this article.
In this section, we use the notation from Section: ``Strategy of the Proof",
we write Theorem $X(d,c)$ for Theorem $X$ in dimension $d$ and coregularity at most $c$.
The following is the boundedness of complements
for Fano type pairs of coregularity $0$. Note that the following theorem is an unconditional version of Theorem~\ref{introthm:Fano-coreg-c}$(0)$.

\begin{theorem}\label{thm:coreg-0-comp}
Let $p$ be a positive integer. 
Let $\Lambda\subset \qq$ be a closed set satisfying the descending chain condition.  
There exists a constant $N \coloneqq N(\Lambda,p)$ satisfying the following.
Let $X$ be a Fano type variety
and 
$(X,B,\bM.)$ be a generalized pair of absolute coregularity $0$.
Assume that the following conditions hold:
\begin{itemize}
    \item the coefficients of $B$ belong to $\Lambda$; and 
    \item $p\bM.$ is b-Cartier.
\end{itemize}
Then, there exists a boundary $B^+\geq B$ such that
\begin{itemize}
    \item the generalized pair $(X,B^+,\bM.)$ is generalized log canonical;
    \item we have that $N(K_X+B^+ +\bM.)\sim 0$; and 
    \item the equality ${\rm coreg}(X,B^+,\bM.)=0$ holds.
\end{itemize}
\end{theorem}

\begin{proof}
First, we replace $\Lambda$ with its derived closure (see Lemma~\ref{lem:derived-closure}).
We let $\mathcal{R} \coloneqq \mathcal{R}(\Lambda,0,p) \subset \Lambda$ 
be the finite subset provided by Theorem~\ref{thm:reduction-of-coefficients}.
This finite subset only depends on $\Lambda$ and $p$.

By~\cite[Theorem 1.2]{FM20}, there is a constant $N(\Lambda,d,0,p)$ such that every generalized pair
$(X,B,\bM.)$ as in the statement and of dimension at most $d$ admits an $N(\Lambda,d,0,p)$-complement.
We will proceed by induction on $d$.
We may assume that 
$N(\Lambda,d,0,p)$ is divisible
by $I(\Lambda)$ and $p$
for every $d$.
Throughout the proof, we assume that $N(\Lambda,d,0,p)$ is minimal with such properties.

By Theorem~\ref{thm:reduction-of-coefficients}, 
we may assume that the coefficients of $B$ belong to $\mathcal{R}$.
Let $B+\Gamma$ be a $\qq$-complement
of $(X,B,\bM.)$ of coregularity $0$. 
Let $(Y,B_Y+\Gamma_Y,\bM.)$ be a generalized dlt modification
of $(X,B+\Gamma,\bM.)$.
Here, $\Gamma_Y$ is the strict transform 
of the fractional part of $\Gamma$
and $B_Y$ is the reduced exceptional plus the strict transform of $B+\lfloor \Gamma \rfloor$. By Lemma~\ref{lem:complements-and-dlt-mod}, $Y$ is a Fano type variety. Thus, 
we may run a $-(K_Y+B_Y+\bM Y.)$-MMP which terminates with a good minimal model $Z$ since $-(K_Y+B_Y+\bM Y.)\sim_{\qq}\Gamma_Y$ is effective.
Let $B_Z$ be the strict transform of $B_Y$ on $Z$
and $\bM Z.$ be the trace of $\bM.$ on $Z$.
Note that $(Z,B_Z,\bM.)$ is a generalized pair of coregularity $0$
and $-(K_Z+B_Z+\bM Z.)$ is a semi-ample divisor.
In order to produce an $N$-complement for $(X,B,\bM.)$,
it suffices to produce an $N$-complement for $(Z,B_Z,\bM.)$ (see Lemma~\ref{lem:complements-and-K-positive} and Lemma~\ref{lem:complements-and-dlt-mod}).
Hence, we may replace $(X,B,\bM.)$ with $(Z,B_Z,\bM.)$
and assume that $-(K_X+B+\bM X.)$ is semi-ample and that $(X,B,\bM.)$ has coregularity 0.
By~\cite[Theorem 1.2]{FM20}, then this reduction shows that Theorem~\ref{introthm:Fano-coreg-c}$(d,0)$ holds for any $d$; this will be used to then appeal to Theorem~\ref{thm:lifting-from-divisor}.

If $K_X+B+\bM X. \sim_\qq 0$, then the statement follows from~\cite[Theorem 1]{FMM22}.
Hence, we may assume that 
the ample model $W$ of $-(K_X+B+\bM X.)$ is positive dimensional.
Since $(X,B,\bM.)$ has coregularity $0$ and $\dim W >0$, then
some generalized log canonical center of $(X,B,\bM.)$ is vertical over $W$.
We may replace $(X,B,\bM.)$ with a generalized dlt modification
and assume that some $S\subset \Supp \lfloor B\rfloor$ is vertical over $W$.
Write $\Xi=B-S$.
Let $X'$ be the ample model of $\Xi+\bM X.$ over $W$.
Notice that, since $X$ is of Fano type and $S$ is vertical over $W$, $X \drar X'$ is a birational contraction.
Let $\Xi', B'$ be the push-forward of $\Xi, B$ on $X'$.
Note that $S$ is not contained in ${\rm Bs}(\Xi+\bM X./W)$.
Hence, $S$ is a generalized log canonical place
of $(X',B'-\epsilon\Xi',(1-\epsilon)\bM.)$ for
every $\epsilon>0$ small enough.
Let $(Y,B_Y,\bM.)$ be a $\qq$-factorial generalized dlt modification of
$(X',B',\bM.)$. Let $\pi \colon Y\rightarrow X'$ be the associated projective morphism.
We write $\Xi_Y+\bM Y.=\pi^*(\Xi+\bM X.)$.

By construction, the following conditions hold:
\begin{itemize}
    \item the generalized pair 
    $(Y,B_Y,\bM.)$ is generalized dlt and
    $-(K_Y+B_Y+\bM Y.)$ is nef; and 
    \item the generalized pair
    $(Y,B_Y-\epsilon \Xi_Y,(1-\epsilon)\bM .)$ is generalized dlt, it is not generalized klt, and the divisor
    $-(K_Y+B_Y-\epsilon \Xi_Y+(1-\epsilon) \bM Y.)$ is big and nef.
\end{itemize}
By Lemma~\ref{lem:complements-and-K-positive} and Lemma~\ref{lem:complements-and-dlt-mod}, an $N$-complement of $(Y,B_Y,\bM.)$
induces an $N$-complement of $(X,B,\bM.)$.
By Theorem~\ref{thm:lifting-from-divisor}
we conclude that $(Y,B_Y,\bM.)$ admits 
an $N(D(\mathcal{R}),d-1,0,p)$-complement.
Notice that we can rely on Theorem~\ref{thm:lifting-from-divisor} in lower dimension, since we are proceeding by induction on $d$, and the conjectures to which Theorem~\ref{thm:lifting-from-divisor} is conditional are known in coregularity 0 (by the remarks under the statement of Theorem~\ref{introthm:Fano-coreg-c}).
Hence, $(X,B,\bM.)$ admits a 
$N(D(\mathcal{R}),d-1,0,p)$-complement.
Note that $D(\mathcal{R})\subset \Lambda$.
So this is also a 
$N(\Lambda,d-1,0,p)$-complement.
Thus, by the minimality of $N(\Lambda,d,0,p)$, we have that
\[
N(\Lambda,d,0,p) \leq N(\Lambda,d-1,0,p).
\] 
This implies that $N(\Lambda,d,0,p)$ is bounded above by $N(\Lambda,1,0,p)$.
This finishes the proof.
\end{proof}

\begin{proof}[Proof of Theorem~\ref{introthm:Fano-coreg-0}]
We follow the notation of the proof of
Theorem~\ref{thm:coreg-0-comp}.
By~\cite[Corollary 3]{FMM22}, a log Calabi--Yau pair
with standard coefficients and coregularity $0$
has coefficients in $\left\{\frac{1}{2},1\right\}$.
In particular, a generalized log canonical threshold with standard coefficients and coregularity $0$ belongs to $\left\{\frac{1}{2},1\right\}$ (see
Definition~\ref{def:coreg-lct} and~\cite[Theorem 4.3]{FMP22}).
On the other hand, 
by Corollary~\ref{cor:pseff-treshold-coreg-0}, a generalized pseudo-effective threshold with standard coefficients 
and coregularity 0 is either $\frac{1}{2}$ or $1$.
In the proof of Theorem~\ref{thm:reduction-of-coefficients}, 
the set $\mathcal{R}(\mathcal{S},0,2)$
only consists of log canonical thresholds
of coregularity 0
and pseudo-effective thresholds of coregularity 0.
Hence, by the proof of Theorem~\ref{thm:reduction-of-coefficients},
we may assume that the coefficients of $B$ belong to $\{\frac{1}{2},1\}$, i.e.,
we have that
\[
\mathcal{R}_0 \coloneqq \mathcal{R}(\mathcal{S},0,2)=
\left\{\frac{1}{2},1 \right\}.
\]
Note that $I_{\mathcal{S}}=1$ (see Definition~\ref{def:index-of-a-set})
and $p=2$ in this case. 
By the proof of Theorem~\ref{thm:coreg-0-comp}, we conclude that 
\[
N(\mathcal{S},d,0,2) =
N(\mathcal{R}_0,d,0,2) 
\leq 
N(\mathcal{R}_0,d-1,0,2),
\] 
for every $d\geq 2$.
Then, the proof follows as
$N(\mathcal{S},1,0,2)=2$.
\end{proof}

\begin{proposition}\label{prop:cbf-implies-index}
Assume that Theorem~\ref{introthm:cbf-and-coreg}$(c)$ holds. 
Then, Theorem~\ref{introthm:index-higher-coreg}$(c)$ holds. 
\end{proposition}

\begin{proof}
By Lemma~\ref{lem:derived-closure}, we may assume that
$\Lambda$ is derived.
By~\cite[Theorem 2]{FMM22}, we may assume that
the coefficients of $B$ belong to a finite subset $\Lambda_0\subset \Lambda$.
Let $\lambda_0$ be the smallest positive integer such that $\lambda_0\Lambda_0 \subset \zz$
and $p$ divides $\lambda_0$.
Let $I(\Lambda_0,c,d,p)$ be the smallest positive integer that is divisible by the index of all the generalized pairs
as in the statement of dimension at most $d$.
A priori, $I(\Lambda_0,c,d,p)$
may not exist.
However, due to Conjecture~\ref{conj:index}
and Lemma~\ref{lem:index-gen-klt}, 
we know that
$I(\Lambda_0,c,c,p)$
is finite.
We proceed by induction on the dimension $d$.
Assume that $I(\Lambda_0,d-1,c,p)$ is finite.
We may assume that $I(\Lambda_0,d-1,c,p)$ is divisible by $\lambda_0$.
Let $(X,B,\bM.)$ be a $d$-dimensional generalized pair as in the statement.
We write $I(X,B,\bM.)$ for the index of this generalized log Calabi--Yau pair.

By Theorem~\ref{thm:fully-support-big}, we may replace $(X,B,\bM.)$ with a Koll\'ar--Xu model and assume the following conditions hold:
\begin{itemize}
    \item the generalized pair $(X,B,\bM.)$ is generalized dlt;
    \item there is a contraction
    $\pi\colon X\rightarrow Z$ for which $\lfloor B\rfloor$ fully supports a $\pi$-semi-ample
    and $\pi$-big divisor; and
    \item every generalized log canonical center of $(X,B,\bM.)$ dominates $Z$.
\end{itemize}
In particular, we know that $Z$ has dimension at most $c$.

We will proceed in two different cases, depending
on the coefficients of $B+\bM X.$ and the rational
connectedness of $X$.\\

\noindent\textit{Case 1:} 
We assume that $\bM.$ is numerically non-trivial and $X$ is rationally connected.\\

We will proceed in two different sub-cases, 
depending on the coefficients of $\{B\}+\bM X.$.\\

\noindent\textit{Case 1.1:} We assume that $\{B\}+\bM X.$ is $\qq$-trivial on the general fiber of $\pi$.\\

We observe that, given the running assumption that $\bM.$ is numerically non-trivial, in this subcase we have $\dim Z > 0$.
By Theorem~\ref{introthm:cbf-and-coreg}$(c)$, we can write
\[
q(K_X+B+\bM X.) \sim q\pi^*(K_Z+B_Z+\bN Z.),
\]
where $(Z,B_Z,\bN.)$ is a generalized klt log Calabi--Yau pair and
the positive integer $q$ only depends on $\Lambda_0$, $c$ and $p$.
The coefficients of $B_Z$ belong to a set $\Omega$,
which satisfies the DCC,
and only depends on $\Lambda_0$, $c$ and $p$.
Finally, the divisor $q\bN .$ is b-Cartier.
We conclude that
\[
I(X,B,\bM .) \leq {\rm lcm}(q,I(\Omega,c,c,q)).
\] 
Note that the value on the right-hand side 
only depends on $\Lambda_0,c$ and $p$.\\ 

\noindent\textit{Case 1.2:} We assume that $\{B\}+\bM X.$ is non-trivial on the general fiber of $\pi$.\\

We run a $(K_X+\lfloor B\rfloor)$-MMP over $Z$.
Since $K_X+\lfloor B\rfloor$ is not pseudo-effective over $Z$, this minimal model program
terminates with a Mori fiber space
$\pi'\colon X' \rightarrow W$ over $Z$.
We denote the push-forward of $B$ to $X'$ by $B'$.
We may replace $(X,B,\bM.)$ with
$(X',B',\bM.)$ 
and assume that $K_X+\lfloor B\rfloor$ is anti-ample over $W$.
In this reduction, we may give up the generalized dlt property for $(X,B,\bM.)$, while $(X,\lfloor B \rfloor,\bM.)$ remains generalized dlt.
By the reduction to the Koll\'ar--Xu model,
the divisor $\lfloor B\rfloor$ contains a prime component $S$ which dominates $W$.
By construction, the general fiber of $S\rightarrow W$ is rationally connected.
Since $W$ is rationally connected, we conclude that $S$ is rationally connected.
Since $(X,\lfloor B \rfloor,\bM.)$ is generalized dlt, $S$ is normal.
Let $(S,B_S,\bN.)$ be the generalized pair obtained by adjunction.
Then, by Lemma~\ref{lemma-coeff-adj} and~\cite[Theorem 2]{FMP22}, we know that the coefficients of $B_S$ belong to $\Lambda_0$ and
$p\bN.$ is b-Cartier.
We conclude that 
\[
I(\Lambda_0,d-1,c,p)(K_S+B_S+\bN S.)\sim 0.
\]
By Theorem~\ref{thm:lifting-fibration-gen-pairs}, 
we conclude that 
\[
I(\Lambda_0,d-1,c,p)(K_X+B+\bM X.)\sim 0.
\] 

Putting Case 1.1 and Case 1.2 together,
we conclude that if $\bM.$ is numerically non-trivial
and $X$ is rationally connected, then we have that
\[
I(X,B,\bM.) \leq \max\{I(\Lambda_0,d-1,c,p),
{\rm lcm}(q,I(\Omega,c,c,q))\}.
\]

\noindent\textit{Case 2:} We assume that either $\bM.$
is the trivial b-divisor
or $\bM.$ is numerically trivial and $X$ is rationally connected.\\ 

In the latter case, by~\cite[Lemma 3.9]{FMM22} we know that $p\bM X'.\sim 0$ where $X'\rightarrow X$ is a resolution on which $\bM.$ descends.
Replacing $\bM.$ with $p\bM.$, 
we may assume that the first case holds.

Let $(X,B)$ be a log Calabi--Yau pair as in the statement.
We assume that the coefficients of $B$ belong to $\Lambda_0$. We may assume that $d>c$.
We replace $(X,B)$ with a $\qq$-factorial dlt modification.
Let $S\subset \lfloor B\rfloor$ be a prime component.
We run a $(K_X+B-S)$-MMP.
This terminates with a Mori fiber space
$\pi'\colon X'\rightarrow W$.
We denote by $B'$ the push-forward of $B$ to $X'$.
We replace $(X,B)$ with $(X',B')$.
Note that $S$ is ample over the base $W$ and $(X,B-S)$ is dlt.
Let $(S,B_S)$ be the pair obtained by adjunction.
Then, $(S,B_S)$ is a semi-log canonical log Calabi--Yau pair 
by~\cite[Example 2.6]{FG14}.
Furthermore, by Lemma~\ref{lemma-coeff-adj} and~\cite[Theorem 2]{FMP22}, the coefficients of $B_S$ belong to $\Lambda_0$. 
By Theorem~\ref{thm:slc-case}, up to replacing
$I(\Lambda_0,d-1,c)$
with
${\rm lcm}(I(\Lambda_0,d-1,c),I_a(\Lambda_0,c))$, we may assume that
\[
I(\Lambda_0,d-1,c,0)(K_S+B_S)\sim 0.
\] 
Here, $I_a(\Lambda_0,c)$ is the constant from Theorem~\ref{thm:slc-case}.
By construction, 
we have that $(X,B-S)$ is dlt,
$X$ is $\qq$-factorial and klt, 
and 
$S\subset \lfloor B\rfloor$ is ample over $W$.
If the fibers of $S\rightarrow W$ are connected, then we can apply
Theorem~\ref{thm:lifting-fibration-pairs}, to conclude that
\[
I(\Lambda_0,d-1,c,0)(K_X+B)\sim 0.
\] 
Otherwise, we can apply  Theorem~\ref{thm:lifting-fibration-pairs-2}, to conclude that
\[
I(\Lambda_0,d-1,c,0)(K_X+B)\sim 0.
\] 

Putting case Case 1 and Case 2 together,
we conclude that every generalized pair
$(X,B,\bM.)$ of dimension $d$ as in the statement
satisfies that
\[
I(X,B,\bM.) \leq \max\{ 
I(\Lambda_0,d-1,c,p),
{\rm lcm}(q,I(\Omega,c,c,q))
\}. 
\] 
Hence, we have that
\[
I(\Lambda_0,d,c,p) 
\leq \max\{ 
I(\Lambda_0,d-1,c,p),
{\rm lcm}(q,I(\Omega,c,c,q))
\} 
\]
Proceeding inductively, we conclude that
\[
I(\Lambda_0,d,c,p) \leq 
\max\{ 
I(\Lambda_0,c,c,p),qI(\Omega,c,c,c)
\}. 
\] 
The right-hand side does not depend on $d$.
This finishes the proof of the proposition.
\end{proof}

\begin{lemma}\label{lem:admissible:pp1}
Let $\lambda$ be a positive integer.
Let $(\pp^1,B_{\pp^1},\bM \pp^1.)$
be a generalized log Calabi--Yau pair.
Assume that the coefficients of $B_{\pp^1}$ belong to $D_\lambda$
and $2\lambda \bM \pp^1.$ is Weil.
Then, $I(K_\pp^1+B_{\pp^1}+\bM \pp^1.)\sim 0$ for some $I=m\lambda$,
where $m \leq 120\lambda$ is an even positive integer.
Furthermore, if $\bM \pp^1.=0$, then
$H^0(\pp^1,\mathcal{O}_{\pp^1}(I(K_\pp^1+B_{\pp^1})))$
admits an admissible section.
\end{lemma}

\begin{proof}
We prove the second statement.
Let $G={\rm Aut}(\pp^1,B_{\pp^1})$.
The group $G$ is a finite extension of a torus.
First, we assume that $G$ is finite, which turns to imply 
that $B_{\pp^1}$ is supported in at least three points.
Let $(\pp^1,B_{\pp^1})\rightarrow (\pp^1,B'_{\pp^1})$ be the quotient by $G$ given by the Hurwitz formula. 
By pulling back to $\pp^1$, it suffices to show that 
\[
h^0(\pp^1,\mathcal{O}_{\pp^1}(I(K_{\pp^1}+B'_{\pp^1}))) \neq 0
\] 
for some $I=m\lambda$, where $m\leq 120\lambda$.
We write
\[
B'_{\pp^1}=\sum_{i=1}^k\left(
1-\frac{1}{m_i}+\frac{\sum_{j=1}^{l_k}\frac{p_{j,k}}{\lambda}}{m_i}
\right)\{p_i\},
\] 
where the $p_{j,k}$'s are positive integers.
We may assume that $m_1\geq \dots\geq m_k$.
If $m_1\leq 5$, then for $I=120\lambda$, we have that $IB'_{\pp^1}$ is Weil.
If $m_1>5$, then $m_3=\dots=m_k=1$.
If $m_2>2\lambda$,
then $m_1>2\lambda$
and 
\[
{\rm coeff}_{p_1}(B'_{\pp^1})
+
{\rm coeff}_{p_2}(B'_{\pp^1})
+
{\rm coeff}_{p_3}(B'_{\pp^1})
>
\left(1-\frac{1}{2\lambda}\right)
+
\left(1-\frac{1}{2\lambda}\right)
+\frac{1}{\lambda} = 2.
\] 
This leads to a contradiction.
Hence, we may assume that $m_2\leq 2\lambda$.
In this case, we have that 
$m_2\lambda B'_{\pp^1}$ is Weil.
Thus, it suffices to take $I=m_2\lambda$ with $m_2\leq 2\lambda$.

Now, we assume that $G$ is a finite extension of a torus.
This implies that $B_{\pp^1}$ is supported in two points, so we may assume that 
$B_{\pp^1}=\{0\}+\{\infty\}$
and that 
$G\simeq \mathbb{G}_m\rtimes \zz_2$.
Note that $\mathbb{G}_m$ acts trivially on 
B-representations as it is connected.
Hence, in this case, $2(K_{\pp^1}+B_{\pp^1})$ admits an admissible section.

In the first statement, we need to control the index of the generalized pair.
The same argument 
we used in the previous paragraph
to control the index of $(\pp^1,B'_{\pp^1})$ applies
to $(\pp^1,B_{\pp^1},\bM \pp^1.)$.
The only difference is that due to the presence of $\bM.$, it could be that $B_{\pr 1.}$ is supported at only one point or it could even be empty.
In this case, $2\lambda(K_{\pr 1.}+ B_{\pr 1.} + \bM \mathbb{P}^1.)$ is integral, and the claims follow.
\end{proof}

\begin{proof}[Proof of Theorem~\ref{introthm:index-coreg-1}]
Let $(X,B,\bM.)$ be a $d$-dimensional generalized log Calabi--Yau pair as in the statement.
We follow the proof of Proposition~\ref{prop:cbf-implies-index}.
In Case 1.1, we can write
\[
q(\lambda)(K_X+B+\bM X.) \sim q(\lambda)
\pi^*(K_{\pp^1}+B_{\pp^1}+\bM \pp^1.),
\] 
where $(\pp^1,B_{\pp^1},\bM \pp^1.)$
is a generalized log Calabi--Yau pair
for which the coefficients 
of $B_{\pp^1}$ belong to a DCC set $\Omega$
and $2\lambda \bM \pp^1.$ is Weil.
Here, we can take $q(\lambda)=2\lambda$.
Indeed, the constant $q(\lambda)$ in the canonical bundle formula depends on the existence of bounded complements with standard coefficients and relative absolute coregularity $0$.
By Theorem~\ref{thm:horizontal-relative-comp} and 
Theorem~\ref{introthm:Fano-coreg-0} such relative complement can be chosen to be a $2\lambda$-complement.
On the other hand, we can take
$\Omega=D_\lambda$.
Indeed, in this case, $(X,B,\bM.)$ admits a log canonical center which has a finite dominant map to ${\pp^1}$.
Thus, the coefficients of $B_{\pp^1}$ can be computed by the adjunction formula and Riemann--Hurwitz.
By Lemma~\ref{lem:admissible:pp1}, we conclude that $I(K_{\pp^1}+B_{\pp^1}+\bM \pp^1.)\sim 0$ for some integer $I=m\lambda$
where $m\leq 120\lambda$.
We conclude that $I(K_X+B+\bM X.)\sim 0$ for the same choice of $I$.

In Case 1.2 and Case 2, the index of $(X,B,\bM.)$ divides the index of a possibly non-normal log Calabi--Yau pair of coregularity 1 and dimension $d-1$.
Inductively, we reduce to the 1-dimensional case which follows by
Theorem~\ref{thm:slc-case} and
Lemma~\ref{lem:admissible:pp1}.
\end{proof} 

\begin{proof}[Proof of Theorem~\ref{introthm:index-coreg-1-standard}]
In a similar fashion as the proof of Theorem~\ref{introthm:index-coreg-1},
this follows from the proof of Proposition~\ref{prop:cbf-implies-index} and the classification
of log Calabi--Yau pair structures on $\pp^1$ with standard coefficients.
\end{proof}

\begin{proposition}\label{prop:index-implies-complements}
Assume that  Theorem~\ref{introthm:index-higher-coreg}$(c)$ holds and
Theorem~\ref{introthm:cbf-and-coreg}$(c)$ holds. 
Then, Theorem~\ref{introthm:Fano-coreg-c}$(c)$ holds.
\end{proposition}

\begin{proof}
By Lemma~\ref{lem:derived-closure}, we may assume that $\Lambda$ is derived.
Let $N(\Lambda,d,c,p)$ be the smallest positive integer for which
every generalized pair $(X,B,\bM.)$ of
dimension $d$ as in the statement
admits an $N$-complement.
By Theorem~\ref{thm:reduction-of-coefficients}, there exists a finite subset $\mathcal{R}\subset\Lambda$
for which 
\[
N(\mathcal{R},d,c,p)=N(\Lambda,d,c,p)
\] 
for every $d$.
We proceed by induction on $d$.
We may assume that every complement throughout the proof is divisible
by $p$ and $I_{\mathcal{R}}$.
We write $N(X,B,\bM.)$ for the smallest positive integer for which $(X,B,\bM.)$ admits an $N$-complement of coregularity $c$.
Let $B+\Gamma$ be a $\qq$-complement
of $(X,B,\bM.)$ that computes the absolute coregularity.
Let $(Y,B_Y+\Gamma_Y+E,\bM.)$ be a generalized $\qq$-factorial dlt modification of
$(X,B+\Gamma_Y,\bM.)$, where $B_Y$ (resp. $\Gamma_Y$) is the strict transform of the fractional part of $B$ (resp. $\Gamma$).
Then, the generalized pair
$(Y,B_Y+E,\bM.)$ has coregularity $c$.
We run a $-(K_Y+B_Y+E)$-MMP with scaling,
which terminates with a good minimal model $Z$.
By Lemma~\ref{lem:coreg-nonnegative-contraction}, the generalized pair $(Z,B_Z+E_Z)$ has coregularity $c$.
By Lemma~\ref{lem:complements-and-dlt-mod}
and Lemma~\ref{lem:complements-and-K-positive}, 
we have that
$N(X,B,\bM.)\leq N(Z,B_Z+E_Z,\bM.)$.
In order to give an upper bound for $N(X,B,\bM.)$, we may replace
$(X,B,\bM.)$ with $(Z,B_Z+E_Z,\bM.)$.
Thus, we may assume that 
${\rm coreg}(X,B,\bM.)=c$
and $-(K_X+B+\bM X.)$ is semi-ample.
Let $W$ be the ample model of $-(K_X+B+\bM X.)$.
We proceed in three different cases, depending on the dimension of $W$.\\

\noindent\textit{Case 1:} In this case, we can assume that $\dim W=0$.\\

Then, we have that $N(K_X+B+\bM X.)\sim 0$
for some $N$ that only depends on $\Lambda$, $c$ and $p$ by Theorem~\ref{introthm:index-higher-coreg}$(c)$.\\

\noindent\textit{Case 2:} In this case, we
assume that $\dim W=\dim X$.\\ 

In this case, we have that $-(K_X+B+\bM X.)$ is a nef and big divisor.
We may assume that $(X,B,\bM.)$ is $\qq$-factorial and generalized dlt.
By Theorem~\ref{thm:lifting-from-divisor}, we conclude that $N(X',B',\bM.)\leq N(\Lambda,d-1,c,p)=N(\mathcal{R},d-1,c,p)$.
In this case, we conclude that 
\[
N(X,B,\bM.)\leq N(\Lambda,d-1,c,p)=N(\mathcal{R},d-1,c,p).
\]

\noindent\textit{Case 3:} In this case, we assume that $0<\dim W < \dim X$.\\

We run a $(\{B\}+\bM X.)$-MMP over $W$ which terminates with a good minimal model $X\dashrightarrow X'\rightarrow W$ over $W$.
By Lemma~\ref{lem:coreg-nonnegative-contraction}, the coregularity of $(X,B,\bM.)$ is unaffected by this MMP.
Let $W'\rightarrow W$ be the ample model of $\{B\}+\bM X.$ over the base.
First, assume that $\dim W'=\dim X$.
In this case, $\{B'\}+\bM X'.$ is big over $W$.
Hence, for $\epsilon>0$ small enough, 
we have that the generalized pair
$(X',B'-\epsilon \{B'\},(1-\epsilon)\bM.)$ is generalized log canonical but not generalized klt and the divisor
\[
-(K_{X'}+B'-\epsilon \{B'\}+(1-\epsilon)\bM X'.)
\] 
is big and nef.
Note that $N(X,B,\bM.)=N(X',B',\bM.)$.
By Theorem~\ref{thm:lifting-from-divisor}, we conclude that
\[
N(X,B,\bM.)\leq N(\Lambda,d-1,c,p)=N(\mathcal{R},d-1,c,p).
\] 
From now on, we assume that
$0<\dim W' < \dim X$.
We separate in two cases, depending on the log canonical centers of $(X',B',\bM.)$.\\

\noindent\textit{Case 3.1:}
In this case, we assume that there is a generalized log canonical center 
of $(X',B',\bM.)$ that is vertical over $W'$.\\

We may assume that $(X',B',\bM.)$ is generalized dlt.
Let $S'\subset \lfloor B'\rfloor$ be a prime component that is vertical over $W'$. In this case, $B'_{\rm hor}$ is big over the base.
We run a $B'_{\rm hor}$-MMP over $W'$ which terminates with a good minimal model $X'_0$
and we consider its ample model $X''$ over $W'$.
We have the following commutative diagram:
\begin{align*}
\xymatrix{
(X,B,\bM.)\ar@{-->}[r]\ar[d]& (X',B',\bM.)\ar@{-->}[r]\ar[d]^-{\pi'} & (X'_0,B'_0,\bM.) \ar[r]^-{\psi}\ar[ld] & (X'',B'',\bM.)\ar[lld] \\
W &  W'.\ar[l] & &
}
\end{align*} 
Note that all the previous models are crepant.
Hence, we have that
\begin{equation}
\label{eq:seq-N}
N(X,B,\bM.)=
N(X',B',\bM.)=
N(X'',B'',\bM.).
\end{equation}
By construction, the following conditions are satisfied:
\begin{itemize}
    \item the variety $W$ is an ample model for $-(K_X+B+\bM X.)$;
    \item the variety $W'$ is an ample model for $\{B'\}+M_{X'}$ over $W$;
    \item the variety $X'_0$ is a good minimal model for
    $B'_{0,{\rm hor}}$ over $W'$; and
    \item the variety $X''$ is an ample model for $B''_{\rm hor}$ over $W'$.
\end{itemize}
We conclude that the divisor
\[
-(K_{X''}+B''-\epsilon\{B''\}-\delta B''_{\rm hor}+ (1-\epsilon)\bM X''.) 
\]  
is ample for $\epsilon\gg \delta >0$ small enough.
We claim that the generalized pair
\[
(X'',B''-\epsilon\{B''\}-\delta B''_{\rm hor}, (1-\epsilon)\bM .)
\] 
is generalized log canonical
but not generalized klt.
Note that
$(X'_0, B'_0-\epsilon \{B'_0\}-\delta B'_{0,{\rm hor}},(1-\epsilon)\bM.)$ is generalized log canonical and not generalized klt
as $S'_0$ is a component
of $\lfloor B'_0-\epsilon \{B'_0\}-\delta B'_{0,{\rm hor}}\rfloor$. 
The morphism $\psi$ is
$(K_{X_0'}+B_0'-\epsilon\{B_0'\}-\delta B'_{0,{\rm hor}}+(1-\epsilon)\bM X'_0.)$-trivial,
so the claim follows.
By Theorem~\ref{thm:lifting-from-divisor} and the sequence of equalities~\eqref{eq:seq-N}, we conclude that
\[
N(X,B,\bM.)\leq
N(\Lambda,d-1,c,p)=N(\mathcal{R},d-1,cp).
\] 
Thus, in this case, we have that
$N(X,B,\bM.)\leq N(\mathcal{R},d-1,c)$.\\ 

\noindent\textit{Case 3.2:}
In this case, we assume that all the generalized log canonical centers of $(X',B',\bM.)$ are horizontal over $W'$.\\

Let $\pi'\colon X'\rightarrow W'$ be the projective contraction.
We may apply Theorem~\ref{introthm:cbf-and-coreg}$(c)$ to obtain a linear equivalence:
\[
q(K_{X'}+B'+\bM X'.)\sim 
q\pi'^\ast(K_{W'}+B_{W'}+\bN W'.),
\]
where the following conditions are satisfied:
\begin{itemize}
    \item the generalized pair
    $(W',B_{W'},\bN .)$ is of Fano type, has dimension $d_{W'}\leq c$, and is exceptional (i.e., its absolute coregularity is equal to its dimension $d_{W'}$);
    \item the positive integer $q$
    only depends on $\Lambda$, $c$ and $p$;
    \item the coefficients of $B_{W'}$ 
    belong to a DCC set $\Omega$
    which only depends on $\Lambda$, $c$ and $p$; and
    \item the b-nef divisor $q\bN.$ is b-Cartier.
\end{itemize}
Indeed, if $(W',B_{W'},\bN.)$ is not exceptional, by pulling back a non-klt complement of it, we obtain a complement for $(X',B',\bM.)$ of coregularity strictly less than $c$. 
This leads to a contradiction.
By~\cite[Theorem 1.2]{FM20}, any generalized pair $(W',B_{W'},\bN.)$ as above admits an $N(\Omega,d_{W'},d_{W'},q)$-complement.
By pulling back, we obtain an $N$-complement for $(X',B',\bM.)$
for some 
\[
N\leq {\rm lcm}(q,N(\Omega,d_{W'},d_{W'},q)).
\] 
Putting Case $1$ through Case $3$ together, we conclude that 
\[
N(X,B,\bM.) \leq \max\{ N(\mathcal{R},d-1,c,p),
{\rm lcm}(q,N(\Omega,1,1,q)),
\ldots,
{\rm lcm}(q,N(\Omega,c,c,q)) 
\}.
\] 
for every $d$-dimensional generalized pair $(X,B,\bM.)$ as in the statement.
Proceeding inductively, we conclude that every generalized pair $(X,B,\bM.)$ as in the statement of the theorem satisfies that
\[
N(X,B,\bM.) 
\leq 
\max\{
N(\mathcal{R},c,c,p),
{\rm lcm}(q,N(\Omega,1,1,q)), 
\ldots, 
{\rm lcm}(q,N(\Omega,c,c,q))
\}.
\] 
Observe that the number on the right-hand side only depends on $\Lambda,c$ and $p$.
This finishes the proof of the implication.
\end{proof}

\begin{lemma}\label{lem:coreg=1-increase-coeff}
Let $(\pp^1,B_{\pp^1},\bM \pp^1.)$ be a generalized log Calabi--Yau pair for which
${\rm coeff}(B_{\pp^1})\in D_t(\zz_{>0})$
and $2\bM \pp^1.$ is Cartier.
If $t>\frac{5}{6}$, then $t=1$.
\end{lemma}

\begin{proof}
Each coefficient of $B_{\pp^1}$ is either standard 
or of the form
\begin{equation}\label{eq:coeff-t} 
1-\frac{1}{m}+\frac{t}{m}> 1-\frac{1}{6m}>5/6.
\end{equation} 
The only sets of standard coefficients whose sum is less than $\frac{7}{6}$ are $\{1\}$ and 
$\{\frac{1}{2},\frac{1}{2}\}$.
Thus, $B_{\pp^1}$ is supported in at most $3$ points and $t=1$.
\end{proof}

\begin{proof}[Proof of Theorem~\ref{introthm:Fano-coreg-1}]
Let $(X,B,\bM.)$ be a $d$-dimensional generalized pair as in the statement.
Following step 1 of the proof of Theorem~\ref{thm:reduction-of-coefficients}, let $t$ be
a log canonical threshold of coregularity $1$
(or a pseudo-effective threshold)
of a prime divisor with respect
to $(X,B,\bM.)$.
By~\cite[Lemma 3.2]{FMP22} (or proof of Lemma~\ref{lem:lcy-horizontal-coeff} in the case where $t$ is a pseudo-effective threshold), we can construct a generalized log Calabi--Yau pair
\[
(\pp^1,B_{\pp^1},\bM \pp^1.)
\]
for which ${\rm coeff}(B_{\pp^1})\in D_{t}(\{1\})$ and $2\bM \pp^1.$ is Cartier.
By Lemma~\ref{lem:coreg=1-increase-coeff}, we conclude that $t=1$ provided that $t>5/6$.
Hence, by the proof of Theorem~\ref{thm:reduction-of-coefficients}, we may assume that the coefficients of $B$ belong to
\[
\mathcal{R}_1 \coloneqq 
\left\{
\frac{1}{2},\frac{2}{3},
\frac{3}{4},\frac{4}{5},
\frac{5}{6}
\right\}.
\]
Proceeding as in the proof of
Proposition~\ref{prop:index-implies-complements}, we conclude that
\[
N(X,B,\bM.) 
\leq 
N(\mathcal{R}_1,d,1,2) 
\leq 
{\rm lcm}(2,
N(\mathcal{R}_1,1,1,2)
). 
\] 
Note that we can take $q=2$ due to Theorem~\ref{introthm:index-coreg-0}.
We conclude that $N\in \{2,4,6\}$.
\end{proof}

We prove the three main theorems of the article. The theorems are proved together inductively.

\begin{proof}[Proof of Theorems~\ref{introthm:index-higher-coreg}, \ref{introthm:Fano-coreg-c}, and~\ref{introthm:cbf-and-coreg}]
Note that Theorem~\ref{introthm:cbf-and-coreg}$(0)$ is trivial.
By~\cite[Theorem 1]{FMM22}, we conclude that
Theorem~\ref{introthm:index-higher-coreg}$(0)$ holds.
By Theorem~\ref{thm:coreg-0-comp}, we know that
Theorem~\ref{introthm:Fano-coreg-c}$(0)$ holds.
Assume that 
Theorem~\ref{introthm:index-higher-coreg}$(c-1)$, 
Theorem~\ref{introthm:Fano-coreg-c}$(c-1)$ and
Theorem~\ref{introthm:cbf-and-coreg}$(c-1)$ hold.
By Proposition~\ref{prop:complements-imply-cbf}, we conclude that Theorem~\ref{introthm:cbf-and-coreg}$(c)$ holds.
By Proposition~\ref{prop:cbf-implies-index}, we conclude that
Theorem~\ref{introthm:index-higher-coreg}$(c)$ holds.
By Proposition~\ref{prop:index-implies-complements}, we conclude that
Theorem~\ref{introthm:Fano-coreg-c}$(c)$ holds.
This finishes the proof of the theorems.
\end{proof}

Finally, we prove the application to klt singularities.

\begin{proof}[Proof of Theorem~\ref{introthm:klt-0}]
Let $(X;x)$ be a klt singularity of absolute coregularity $0$.
Let $(X,\Gamma_0;x)$ be a strictly log canonical pair of coregularity $0$ at $x$.
By~\cite[Lemma 1]{Xu14}, there exists a plt blow-up $\pi\colon Y\rightarrow X$ that extracts a unique exceptional divisor $E$ that is a log canonical place of $(X,\Gamma_0;x)$.\footnote{The fact that the divisor computing the plt blow-up of $(X;x)$ can be chosen to be a log canonical place of $(X,\Gamma_0;x)$ is implicit in the proof of~\cite[Lemma 1]{Xu14}.}
In particular, the pair $(E,\mathrm{Diff}_E(0))$ is a Fano pair of absolute coregularity 0 and with standard coefficients (see, e.g.,~\cite[Proposition 3.9]{Sho92}).
By Theorem~\ref{introthm:Fano-coreg-0}, $(E,\mathrm{Diff}_E(0))$ admits a 1 or 2-complement.
By steps 4-9 of the proof of Proposition~\ref{prop:lifting-from-divisor}, we can lift this complement to a 1 or 2-complement of $(X;x)$ at $x$.
\end{proof} 

\begin{proof}[Proof of Theorem~\ref{introthm:klt-1}]
The proof is analogous to the one of Theorem~\ref{introthm:klt-0}
by replacing 
Theorem~\ref{introthm:Fano-coreg-0}
with Theorem~\ref{introthm:Fano-coreg-1}.
\end{proof}

\bibliographystyle{habbvr}
\bibliography{bib}

\end{document}